\documentclass[a4paper,11pt, reqno]{amsart}

	\usepackage{fullpage}

	\usepackage[foot]{amsaddr}
	\calclayout 

	\usepackage[titletoc]{appendix}

	\usepackage[english]{babel} 

	\usepackage[utf8]{inputenc} 
	\usepackage[T1]{fontenc} 
	
	\usepackage[dvips]{color}
	\usepackage[dvipsnames,svgnames]{xcolor}

	\usepackage{enumitem} 
	\usepackage[normalem]{ulem} 

	\usepackage[colorlinks=true,linkcolor=MidnightBlue,citecolor=OliveGreen]{hyperref}
	
	\usepackage{amsmath,amsfonts,amssymb,amscd,stmaryrd,bbm,amsthm,mathrsfs}

	\usepackage{graphicx} 
	\usepackage{float} 
	\usepackage{multicol} 

	\newcommand{\cA}{\mathcal{A}}	 	
	\newcommand{\cB}{\mathcal{B}}	
	\newcommand{\cC}{\mathcal{C}}	
		\newcommand{\DD}{\mathbbm{D}}
	\newcommand{\cE}{\mathcal{E}}	\newcommand{\EE}{\mathbbm{E}}
	\newcommand{\cF}{\mathcal{F}}	
		
	\newcommand{\cH}{\mathcal{H}}

	\newcommand{\cK}{\mathcal{K}}

	\newcommand{\cN}{\mathcal{N}}	\newcommand{\NN}{\mathbbm{N}}
		
	\newcommand{\cP}{\mathcal{P}}	\newcommand{\PP}{\mathbbm{P}}
		
	\newcommand{\cR}{\mathcal{R}}	\newcommand{\RR}{\mathbbm{R}}
	\newcommand{\cS}{\mathcal{S}}	
	\newcommand{\cT}{\mathcal{T}}

	\newcommand{\bB}{\mathbf{B}}

	\newcommand{\bJ}{\mathbf{J}}

	\newcommand{\bV}{\mathbf{V}}	
	\newcommand{\bW}{\mathbf{W}}	
	\newcommand{\bX}{\mathbf{X}}	
		
	\newcommand{\bZ}{\mathbf{Z}}	
	
	\newcommand{\eps}{\varepsilon}
	\newcommand{\dd}{\mathop{}\!\mathrm{d}}

	\newcommand{\blue}[1]{#1}
	
	\newcommand{\ind}{\mathbbm{1}}
	\newcommand{\1}{\mathbbm{1}}

	\DeclareMathOperator{\var}{Var}
	\newcommand{\bx}{\mathbf{x}}
	\newcommand{\by}{\mathbf{y}}
	\newcommand{\bl}{\mathbf{l}}
	\newcommand{\bw}{\mathbf{w}}
	\newcommand{\bv}{\mathbf{v}}

	\theoremstyle{plain}
		\newtheorem{theorem}{Theorem}[section]
		\newtheorem*{theorem*}{Theorem}
		\newtheorem{corollary}[theorem]{Corollary}
		\newtheorem{lemma}[theorem]{Lemme}
		\newtheorem{proposition}[theorem]{Proposition}
		
		\newtheorem{remark}[theorem]{Remark}
		\newtheorem{assumption}[theorem]{Assumption}

	\theoremstyle{definition}
		\newtheorem{definition}[theorem]{Definition}
		
		\newtheorem{example}[theorem]{Example}

	\usepackage{comment}
	\newif\ifhidecorrections
	\hidecorrectionstrue 
		\ifhidecorrections
		\excludecomment{cor}
		\else
		
		\fi

	\usepackage[most]{tcolorbox}
	\newtcolorbox{hp}[1][]{
   		colback=Gainsboro, enhanced, frame hidden, borderline west = {0.5pt}{0pt}{dashed}, #1
		}
	
	\author{R\'emi Catellier}
	\email{remi.catellier@univ-cotedazur.fr}
	\address[R\'emi Catellier]{Université Côte D'Azur, CNRS, INRIA, LJAD}
	\date{\today}
	\title{Regularization by noise for rough differential equations driven by Gaussian rough paths}

\author{Romain Duboscq}
\address[Romain Duboscq]{IMT; UMR5219, Universit\'e de Toulouse; CNRS, INSA IMT}
\email{romain.duboscq@math.univ-toulouse.fr}

\newcommand{\UMDp}{$\text{UMD}_p$ }

\begin{document}

\begin{abstract}
We consider the rough differential equation with drift driven by a Gaussian geometric rough path. Under natural conditions on the rough path, namely non-determinism, and uniform ellipticity conditions on the diffusion coefficient, we prove path-by-path well-posedness of the equation for poorly regular drifts. In the case of the fractional Brownian motion $B^H$ for $H>\frac14$, we prove that the drift may be taken to be $\kappa>0$ Hölder continuous and bounded for $\kappa>\frac32 - \frac1{2H}$. A flow transform of the equation and Malliavin calculus for Gaussian rough paths are used to achieve such a result. 
\end{abstract}

\maketitle

\section{Introduction}

In this paper, we want to study a regularization by noise property for the following rough differential equation :
\begin{equation}\label{eq:RDE_first}
\dd x_t = b(x_t) \dd t + \sigma(x_t) \dd \bw_t,
\end{equation}
where $\bw$ is a (weakly) geometric rough path. We intend to show that under suitable conditions on $\sigma$ and $\bw$ this equation is wellposed when $b$ has (very) poor regularity properties.

This phenomenon of regularization by noise is now well-studied in several situations. A lot of work has been devoted to the additive case $\sigma\equiv 1$ and for several kind of processes. One can think of the seminal work of \cite{veretennikovStrongSolutionsExplicit1981,krylovStrongSolutionsStochastic2005} for strong solutions of this equation when $w$ is a Brownian motion. In this additive case and when $w$ is a Brownian motion, Davie \cite{davieUniquenessSolutionsStochastic2007} exhibits a new and stronger notion of uniqueness and proved that the previous equation has a strong solution and enjoys "path-by-path" uniqueness whenever $b\in L^\infty(\RR^d)$, whereas standard theory (with supporting counterexamples) requires $b$ to be Lipschitz continuous. One can consult \cite{shaposhnikovPathwiseVsPathbypath2020} for a deep discussion about notions of solutions in the additive case.

Since then, Davie's work has led to a certain number of results in several directions which are usually done in the additive cases. One can take $w$ to be a more general stochastic process and still has (even better) regularization by noise phenomenon. For example one can consider $w$ to be a   fractional Brownian motion \cite{catellierAveragingIrregularCurves2016}, a L\'evy process, see \cite{priolaDavieUniquenessDegenerate2020,athreyaStrongExistenceUniqueness2020} and the references therein, or a more general stochastic process \cite{gerencserRegularisationRegularNoise2022a,harangRegularizationODEsPerturbed2021,duboscqStochasticRegularizationEffects2016}. 
In a more general context some similar results when $w$ is non random and as general as possible can be exhibited \cite{galeatiNoiselessRegularisationNoise2021,catellierAveragingIrregularCurves2016} and \cite{galeatiPrevalenceRhoIrregularity2020,romitoAnotherNotionIrregularity2022}. 
This phenomenon of regularization by noise in a path-by-path manner was also study in other kind of problems, such as SPDEs \cite{catellierPathwiseRegularizationStochastic2021,choukNonlinearPDEsModulated2015a,choukNonlinearPDEsModulated2014,choukPathbypathRegularizationNoise2019,athreyaWellposednessStochasticHeat2022,catellierRoughLinearTransport2016}, mean field differential equations \cite{galeatiDistributionDependentSDEs2021,bauerStrongSolutionsMeanfield2018}, and in a mixed additive/multiplicative setting \cite{bechtoldWeakSolutionsSingular2022a,galeatiRegularizationMultiplicativeSDEs2020}.

A general strategy in this additive context is to look at the averaged field
\[t,x \mapsto \int_0^t b(x+w_r) \dd t\]
and to study its space-time regularity properties. To do so, one could use for example an Itô-Tanaka trick (see  \cite{coutinItOTanakaTrick2019} in the fractional Brownian motion setting) or the corresponding Kolmogorov equation \cite{priolaDavieUniquenessDegenerate2020}.
More recently, techniques involving stochastic sewing lemma \cite{leStochasticSewingLemma2020} and/or properties of the occupation measure/local time of the process \cite{harangRegularizationODEsPerturbed2021} where used. Once a space-time regularity of the averaged field is exhibited, one can then use some pathwise non-linear calculus \cite{galeatiNonlinearYoungDifferential2021} to conclude. We will see in the following how this strategy of proof is implemented in our setting.

In the full multiplicative case, there are very few results (up to our knowledge only two). First of all, one must make sense of the equation in a pathwise sense. To do so, one usually relies on rough path theory  \cite{lyonsDifferentialEquationsDriven1998,
davieDifferentialEquationsDriven2007,
lyonsSystemControlRough2002,
bailleulFlowsDrivenRough2015,
gubinelliControllingRoughPaths2004,friz2020course,friz2010multidimensional}.

Up to our knowledge in this setting, the first related result is due (again) to Davie \cite{davieIndividualPathUniqueness2011a}, where he studies Equation \eqref{eq:RDE_first} where $\bw$ is the Stratonovitch Brownian rough path. In this setting, when $\sigma$ is an invertible $C^{3}_b$ matrix, he shows that uniqueness holds whenever $b\in L^\infty$. The strategy of the previous work is to consider a strong solution $(X_t)_{t\in[0,T]}$ of the Stratonovitch SDE
\[X_{t} = X_0 + \int_0^t b(X_r)\dd r + \int_0^t \sigma(X_r)\circ\dd B_r,\]
to use a Girsanov transform on
\[W_t = B_t + \int_0^t \sigma(X_r)^{-1} b(X_r) \dd r\]
and to get back to the driftless  Stratonovitch SDE on an equivalent probability measure $\tilde{\PP}$.
\[Y_t = Y_0 +  \int_0^t \sigma(Y_t) \circ\dd W_r,.\]
Then the use  of the Kolmogorov PDE related to the previous equation enables to  prove that $Y$ enjoys some regularizing properties. Finally, one can write $X=Y + Z$ and $Z$ solve the following equation 
\[Z_t = X_0 + \int_0^t b(Z_r + Y_r )\dd r.\]
Hence, under $\tilde{\PP}$, one can use the regularizing properties of $Z$ and prove the result.

It is not directly possible to implement such a strategy in our context. Firstly, one lacks of a suitable Girsanov transform in a general (Gaussian) rough path setting. Secondly, there is no Kolmogorov equation. Nevertheless, one can see two ingredients of the proof : going back to the driftless equation, and considering an ODE where the drift and the solution of the driftless equation appears. 

The second result about uniqueness of SDE in a poor regularity setting is due to Athreya, Bhar and Shekhar \cite{athreyaSmoothnessFlowPathbyPath2017} when $\bw$ is the fractional Brownian motion rough path (for $H>\frac13$). There strategy is to make Lamperti transform to get back to an additive problem, and then to use the result of \cite{catellierAveragingIrregularCurves2016}. \blue{Whereas} in dimension $d=1$ this strategy works pretty well, in dimension $d\ge 2$ one has to ask that $\sigma^{-1}$ is conservative. This is a huge restriction, since one could expect that $\sigma$ strictly elliptic \blue{should} be sufficient.

While we were finishing the writing of this paper, a paper by Dareiotis and Gerenc\'ser with similar results \cite{dareiotisPathbypathRegularisationMultiplicative2022} came up. Note that the techniques involved are quite different, and we are able to handle more general cases (general Gaussian rough paths instead of only fractional Brownian motion). The price to pay in our case is a slightly worse condition on $\sigma$ and on $b$, but a far better result on the flow of the equation.

Finally let us mention that in the rough path setting, some results have appeared considering non-Lipschitz drift \cite{bonnefoiPrioriBoundsRough2022,riedelRoughDifferentialEquations2017}. These \blue{works} focus on growth (non-linear damping) of the coefficients and wellposedness in that context. Furthermore, the \emph{local}-Lipschitz continuity is always needed. 

Our strategy to tackle the problem of uniqueness  of Equation \eqref{eq:RDE_first} with singular coefficients may be summarize as follow, and is inspired by the work of Davie \cite{davieIndividualPathUniqueness2011a}. In order to compensate the lack of \blue{a} Girsanov transform in this setting, we rely on a flow transformation presented in the rough path setting in \cite{riedelRoughDifferentialEquations2017}. Indeed, it  allows us to consider the regularization properties of the flow of the driftless equation. Furthermore, in order to replace the Kolmogorov equations arising in a Brownian context, we will use some Malliavin calculus in a Gaussian rough path setting, which have been developed  by several authors \cite{baudoinProbabilityLawsSolutions2016,cass2010densities,cass2009non,louLocalTimesStochastic2017,cass2015smoothness,gess2020density,inahama2014malliavin}

\subsection{Difficulties and extended plan of the paper}

Let us emphasize the difficulties and achievements of the paper. We have divided the study of the uniqueness of solutions of Equation \eqref{eq:RDE_first} in three parts.

The first one, in Section \ref{section:rough_paths}, recalls the basic definitions and results about rough path theory. The main idea of this part is a development of an idea of \cite{riedelRoughDifferentialEquations2017}, and consists of proving an equivalence for the uniqueness of solution (in the sense of Davie) of the rough differential equation  \eqref{eq:RDE_first}  with the uniqueness of the solution to the standard ODE 
\[y'(t) = D\varphi_t\big(y(t)\big)^{-1} b\big(\varphi_t(y(t)\big),\] 
where $\varphi$ is the flow arising from the rough differential equation when $b\equiv 0$. This is done in Theorem \ref{theorem:main1}. Somehow, this is a way of avoiding the Girsanov transform from Davie's work \cite{davieIndividualPathUniqueness2011a} while still working with the solution of the driftless equation. However, this comes at \blue{a} price: we need to work with the flow instead of a trajectory of the driftless system. Furthermore, we can see the appearance of the inverse of the Jacobian $(D\varphi)^{-1}$ that we need to handle. Another objective of this part is to exploit an idea from \cite{catellierAveragingIrregularCurves2016} and to exhibit a criterion for this averaged field 
\[t,x \mapsto \int_{0}^t \big(D\varphi_r(x)\big)^{-1} b\big(\varphi_r(x)\big)\dd r\]
such that the ODE (and hence the RDE) has a unique solution. 

The second part (Section \ref{section:kolmo}) focuses on the regularity of the averaged field whenever the flow $\varphi$ has some stochastic properties, this is done in Theorem \ref{theorem:kolmo} and Corollary \ref{corollary:averaged_holder}. The idea here is to use Burkholder-Davis-Gundy (BDG) inequality for martingales in an infinite dimensional setting (namely in $L^p(\RR^d)$). This is quite close to the use of the stochastic sewing lemma in infinite dimensional spaces \cite{leStochasticSewingBanach2022}. Nevertheless, since those work were developed in parallel we keep here with our presentation. Furthermore, since a BDG inequality in Banach spaces is not such a common topic, and in order to be as self contained as possible, we have included in the Appendix \ref{appen:BDG} \blue{for some key elements on this topic}. Note that in this section we have proved a general Kolmogorov criterion for regularity of the averaged field in Besov spaces which could be interesting \blue{by} itself. 

The third part (Section \ref{section:Malliavin} and \ref{section:Malliavin2}) focuses on a tool to avoid the use of Kolmogorov equations: the Malliavin calculus. Indeed, by using Section \ref{section:kolmo}, the main idea to deduce the regularizing properties of the averaged field constructed thanks to the flow of the driftless RDE is to have a integration by part formula, such that, formally
\blue{\[ \EE[ D\varphi_t(x)^{-1} \nabla b\big(\varphi_t(x)\big) | \cF_s] = \EE[b\big(\varphi_t(x)\big) H_{s,t}(x) |\cF_s],\]
}
with
\[\|H_{s,t}(x)\|_{L^q(\Omega)} \lesssim |t-s|^{-H},\]
for some $H>0$. This is precisely one of the key points of Malliavin calculus. Note that, in that setting, one has to focus on conditional Malliavin calculus. Hence, Section \ref{section:Malliavin} recalls the standard facts and results about Malliavin calculus in a Gaussian context. In this section, some conditional integrations by part results are also proved. Section \ref{section:Malliavin2} focuses itself on Malliavin calculus for solutions of driftless rough differential equation driven by Gaussian rough paths. 

Finally, in Section \ref{section:main} we are able to prove the desired wellposedness result which can be stated as follow (see Theorem \ref{cor:MAIN000} and \ref{thm:MAIN000} for precise statements and assumptions):

\begin{theorem*}
Let $2<p<4$. Let $\bW$ be a $p$-geometric Gaussian rough path such that its first component $(W_t)_{t\in[0,T]}=(\bW^1_{0,t})_{t\in[0,T]}$ is $\alpha$-locally non-determinism, namely
\begin{equation*}
\inf_{0\leq s\leq t\leq T} (t-s)^{-\alpha} \var\left[W_t-W_s\middle| \cF_{[0,s]}\vee\cF_{[t,1]}\right] = c_W >0.
\end{equation*}
Assume that $\sigma \in C^\infty_b(\RR^d;(\RR^d)^{\otimes 2})$ is uniformly elliptic. Let $b\in \cC^{\kappa}$ with $\kappa + \frac{1}{\alpha} > \frac32 $ and $\kappa>0$. Then  almost surely there is a unique solution of Equation \eqref{eq:RDE_first}. Furthermore the solutions of this equation are locally-Lipschitz continuous with respect to the initial condition.

Moreover, when $\bB^H$ is the geometric rough path above the fractional Brownian motion of Hurst parameter $\frac14<H<\frac12$, one can take $\alpha=2H$.
\end{theorem*}

Finally, in order to be self-contained as possible, we have included some multidimensional Garsia-Rodemich-Rumsey inequalities in the Appendix (Appendix \ref{appendix:GRR}) as well as some background on Besov spaces (Appendix \ref{sec:Besov}).

\subsection{Notations and preliminary}

We gather here some useful notations and definitions for the rest of the article.

Throughout this paper, we consider $(\Omega,\cF,(\cF_t)_{t\in [0,T]},\PP)$ to be a filtered probability space. We recall that $\cS$ is the \blue{Schwartz} space (see \cite{trevesTopologicalVectorSpaces2016}) defined  as 
\begin{equation*}
\cS(\RR^d) = \left\{ f\in\cC^{\infty}_0(\RR^d):\;\forall k\in\NN, \|f\|_{\cS,k}<+\infty \right\},
\end{equation*}
where
\begin{equation*}
\|f\|_{\cS,k} = \max_{|\beta|+|\gamma|\leq k}\sup_{x\in\RR^d} |x^\beta D^{\gamma} f(x)|.
\end{equation*}
It is a topological vector space that becomes a Fréchet space when it is equipped with the distance
\begin{equation*}
d_{\cS}(\phi,\varphi) = \sum_{k = 0}^{\infty} 2^{-k}\frac{\|\phi-\varphi\|_{\cS,k}}{1+\|\phi-\varphi\|_{\cS,k}}.
\end{equation*}
Its continuous dual is $\cS'$ the space of tempered distributions and, moreover, we have the following Gelfand triple
\begin{equation*}
\cS\subset L^2(\RR^d) \subset \cS',
\end{equation*}
where $\cS$ is dense in $\cS'$ with respect to the weak topology of $L^2(\RR^d)$. We denote $\langle\cdot,\cdot\rangle$ the duality product which is an extension of the scalar product in $L^2(\RR^d)$. The Besov spaces \blue{$B^s_{p,r}$} (associated to the Littlewood-Paley blocks) are Banach spaces constructed by a completion in $\cS'$ 
\begin{align*}
B_{p,r}^s : = \left\{ u \in \cS'(\RR^d): \|u\|_{B_{p,r}^s} : = \left(\sum_{j = -1}^{+\infty} 2^{r j s}\|\Delta_{j} u\|_{L^p(\RR^d)}^r \right)^{1/r}<\infty\right\},
\end{align*}
where $s\in\RR$ and $p,q\in [1,\infty]$ (see Appendix \ref{sec:Besov} for details and some results).

We also denote, for any $k\geq 0$, $C^{k}(\RR^d;\RR^\ell)$ the set of functions from $\RR^d$ to $\RR^\ell$ that are continuous, $k$ times differentiable and whose derivatives are continuous. We denote $C^{\infty}(\RR^d;\RR^\ell) = \bigcap_{k = 1}^{+\infty}C^{k}(\RR^d;\RR^\ell)$. We denote by $C^k_b$ (respectively $C^\infty_b$) the functions in $C^k$ (respectively in $C^\infty)$ bounded together with all their derivatives.

\begin{definition}\label{def:weight}
A function from $\RR_+$ to $\RR_+\backslash\{0\}$ is called a weight. Let $f$ be a function from $\RR^d$ to $\RR$. Let us define the weighted sup norm of $f$ by
\[\|f\|_{\infty,w} = \sup_{x\in\RR^d} \frac{|f(x)|}{w(|x|)}.\]
Furthermore, we define 
\begin{equation*}
L^{\infty}_{w}(\RR^d;\RR) = \left\{ f:\RR^d\to\RR:\;\|f\|_{\infty,w}<+\infty\right\}.
\end{equation*}
\end{definition}
When there exists $\gamma>0$ such that $\sup_{x\in\RR_+} \frac{w(x)}{(1+x)^\gamma} <+\infty$, we say that $w$ has $\gamma$-power growth.
 Let us now (with a slight abuse of notations) consider $w = (w_k)_{k\ge 0}$ to be a sequence of positive functions from $\RR_+$ to $\RR_+\backslash\{0\}$. We say that $w$ has $\gamma$-power growth if each $w_k$ has a $\gamma$-power growth.

\begin{definition}\label{def:weighted_holder}
Let $w$ be a sequence of weights and $\alpha>0$. We denote
 \[
\|f\|_{\cC^\alpha_{w}} := \sum_{k = 0}^{\lfloor \alpha \rfloor-1} \|D^k f\|_{\infty,w_k} + \sup_{x\neq y} \frac{|D^{\lfloor \alpha \rfloor-1}f(x)-D^{\lfloor \alpha \rfloor-1}f(y)|}{|x-y|^{\alpha-\lfloor\alpha \rfloor}w_{\lfloor\alpha\rfloor} (|x| + |y|)},
 \]
and define the Banach space
\[\cC^\alpha_{w}(\RR^d;\RR^\ell) := \{f:\RR^d \mapsto \RR^\ell\, : \, \|f\|_{\cC^{\alpha}_{w}}< + \infty \}.\]
\end{definition}
To simplify notations, we denote $\cC^{\alpha} = \cC^{\alpha}_{w}$ when $w = (1,1,\ldots,1)$, and we have (see Appendix \ref{sec:Besov}) $\cC^\alpha = B^\alpha_{\infty,\infty}$ for all $\alpha\in\mathbb{R}^{+*}\backslash \mathbb{N}$. We also denote 
\[C^k_{lin}(\RR^d) = C^{k}_{(1+|\cdot|,1,\cdots,1)}(\RR^d)\]

\begin{remark}
Note that when $w = (w_0,1)$, then $\cC^1_{w}$ denotes the set of Lipschitz continuous function with a growth rate $w_0$.
\end{remark}

Furthermore, we also  extend the definition of the H\"older space for functions of two variables. That is, for some Banach space $E$, and for any $T,\nu>0$, we define
\begin{equation*}
    \mathcal{C}^\nu_T = \mathcal{C}^{\nu}(\Delta^2_T;\RR^d) = \left\{ \psi:\Delta^2_T\mapsto\mathbb{R}^d:\; \|\psi\|_{\mathcal{C}^\nu_T} : = \sup_{(s,t)\in\Delta^2_T}\frac{\|\psi_{s,t}\|_E}{|t-s|^{\nu}}<\infty \right\},
\end{equation*}
where $\Delta^2_T = \{(s,t)\in [0,T]:\; s<t\}$.

We finally recall some properties of linear operator. Let $T$ be a linear operator from $D(T)\supset\cS$ to $E$ where $E$ is a Banach space. For any Banach space $F\supset D(T)$, if we have, for any $x\in D(T)$,
\begin{equation*}
\|Tx\|_{E}\leq C \|x\|_F,
\end{equation*}
for some constant $C>0$ independent of $x$, then $T$ admits a closure from $F$ to $E$. The space of bounded linear operators from $E$ to $F$ is denoted $\cB(F,E)$ and is equipped with the norm
\begin{equation*}
\|T\|_{\cB(F,E)} = \sup_{\|x\|_F\leq 1} \|Tx\|_E
\end{equation*}

Finally, we write $a \lesssim b$ when there exists a constant $c>0$ such that $a \le cb$. 
\section{Rough differential equation with drift}\label{section:rough_paths}
In this section, we recall some basic facts about rough paths and rough differential equations. For more details, one can consult the seminal papers \cite{lyonsDifferentialEquationsDriven1998,davieDifferentialEquationsDriven2007,
gubinelliControllingRoughPaths2004}, but also the monographs \cite{lyonsSystemControlRough2002,friz2020course,friz2010multidimensional}. Note that an interesting perspective for flow driven by rough differential equations is given by Bailleul in \cite{bailleulFlowsDrivenRough2015}. We will focus here on the Davie's definition of solutions of rough differential equations driven by Hölder (weakly) geometric $p$-rough path, with $p\ge 2$.
	\subsection{Rough paths and differential equations in a nutshell}
	
For any $N\in\NN$, a truncated tensor algebra $\cT^N(\RR^d)$ is defined by
\begin{equation*}
\cT^N(\RR^d) = \bigoplus_{k = 0}^N (\RR^d)^{\otimes k},
\end{equation*}
with the convention $(\RR^d)^{\otimes 0} = \RR$. Let $(e_1,\cdots,e_d)$ be the canonical basis of $\RR^d$, then for any $k\in \{1,\cdots,N\}$, 
\[(e_I)_{I\in\{1,\cdots,d\}^k}:=\big(e_{i_1}\otimes\cdots \otimes e_{i_k}\big)_{I=(i_1,\cdots,i_k)\in \{1,\cdots,d\}^k}\]
is the canonical basis of $(\RR^d)^{\otimes k}$, and for any $\bx\in \cT^N(\RR^d)$,
\[\bx = \sum_{k\in\{0,\cdots,N\}} \bx^k = x^0 + \sum_{\substack{k\in\{1,\cdots,N\}\\ I\in \{1,\cdots,d\}^k }} \bx^{k,I} e_I,\]
where $\bx^k$ is the projection of $\bx$ on the $k$-th tensor.

This space is equipped with a vector space structure as well as an operation $\otimes$ defined by
\begin{equation*}
(\bx\otimes \by)^k = \sum_{\ell = 0}^N (\bx^{k-\ell})\otimes\by^\ell,\quad\forall \bx,\by\in \cT^N(\RR^d),
\end{equation*} 
In the end, $(\cT^N(\RR^d),+,\otimes)$ is an associative algebra with unit element $\mathbf{1}\in(\RR^d)^{\otimes 0}$. 

For any $s<t$ and $k\geq 1$, we define the simplex
\begin{equation*}
\Delta^k_{s,t} = \{(u_1,\ldots,u_k)\in[s,t]^k; u_1<\ldots<u_k\}.
\end{equation*}
We also denote $\Delta^k_T : = \Delta^k_{0,T}$. A multiplicative functional is a continuous map $\bw:\Delta^2_T\to \cT^N(\RR^d)$ that verifies, for any $s<u<t$,
\begin{equation*}
\bw_{s,t} = \bw_{s,u}\otimes\bw_{u,t}
\end{equation*}
A fundamental example of such a map are the iterated integrals of a smooth paths $w:[0,T]\to \RR^d$
\begin{equation*}
\bw^k_{s,t} = \sum_{1\leq i_1,\ldots, i_k\leq d} \left(\int_{\Delta^k_{s,t}} \dd w^{i_1}\ldots \dd w^{i_k}\right) e_{i_1}\otimes \ldots\otimes e_{i_k},
\end{equation*}
where $(e_1,\cdots,e_d)$ is the canonical basis of $\RR^d$ and $k\ge 1$.
Then, we call the signature of $w$ the mapping $S_N(w):\Delta^2_T\to \cT^N(\RR^d)$ given by
\begin{equation*}
(s,t)\to \mathbf{1} + \sum_{k = 1}^N\bw_{s,t}^k.
\end{equation*}
It turns out that every signature is a multiplicative functional that belongs to $G^N(\RR^d)$ which is a subset of $\cT^N(\RR^d)$ of group-like elements given by
\begin{equation*}
G^N(\RR^d) : = \exp^{\oplus} (L^N(\RR^d)),
\end{equation*}
where $L^N(\RR^d)$ is the linear span of elements that can be written as a commutator $a\otimes b - b\otimes a$ with $a,b\in \cT^N(\RR^d)$. Furthermore, there is a Carnot-Caratheodory norm on $G^N(\RR^d)$, denoted $\|\cdot\|_{\mathrm{CC}}$, which is homogeneous with respect to the natural scaling operation on $\cT^N(\RR^d)$. We can now introduce the notion of rough paths.

\begin{definition}
The space of weakly geometric $p$-rough paths is the set of multiplicative functionals $\bx:\Delta^2\to G^{\lfloor p\rfloor}(\RR^d)$ such that
\begin{equation*}
\|\bw\|_{p-var ;[0,T]} := \sup_{\pi\in\Pi([0,T])}\left( \sum_{[u,v]\in\pi} \|\bw_{u,v}\|_{\mathrm{CC}}^p\right)^{1/p}<\infty,
\end{equation*}
where $\Pi([0,T])$ is the set of all subdivisions of $[0,T]$. The space of geometric $p$-rough paths is the closure under $\|\cdot\|_{p-var ;[0,T]}$ of smooth signatures $\left\{S_{\lfloor p \rfloor}(w); \; w\in\cC^{\infty}([0,T];\RR^d)\right\}$.
\end{definition}

In the following, we will also use the H\"older norm for rough paths that is given by, for $\alpha\in(0,1)$,
\begin{equation*}
\|\bw\|_{\alpha\textrm{-H\"ol};[0,T]} : = \sup_{[s,t]\subset[0,T]} \frac{\|\bw_{s,t}\|_{\mathrm{CC}}}{(t-s)^{\alpha}}.
\end{equation*}

\subsection{Flow generated by driftless differential equations driven by rough paths}

In this subsection, we give known and basic properties about solutions of rough differential equations of the form

\begin{equation}\label{eq:rde_basic}
\dd y_t = \sigma(y_t) \dd \bw_t,\quad y_0 = x,
\end{equation}
where $\sigma:\RR^d \mapsto (\RR^d)^{\otimes 2}$ is such that 
\[\sigma(x) w = \sum_{i=1}^d \sigma_i(x) w^i,\]
for any $w = (w^i)_{i\in\{1,\cdots,d\}} \in \RR^d$.
Following \cite{bailleulDefinitionSolutionRough2021} we identify $\sigma_i$ with a first order differential operator, namely for any smooth function $f : \mathbb{R}^d\mapsto \mathbb{R}^d$,
\[\sigma_i f (x) = Df(x) \sigma_i(x), \quad x \in \RR^d.\]
For $k\ge 1$ and $I \in \{1,\cdots,d\}^k$ we define the $k$-th order differential operator 
\[\sigma_I = \sigma_{i_1}\cdots \sigma_{i_k}\]
via $\sigma_I f = \sigma_{i_1}(\cdots (\sigma_{i_k}f))$, with the slight abuse of notations $\sigma_I(x) = (\sigma_I id )(x)$.

\begin{definition}\label{def:davie_solutions}
Let $p\ge 2$, $T>0$, 
 $\bw$ be a $\frac1p$-Hölder weakly geometric $p$-rough path  and  $\sigma\in \cC^{\beta}$ for some $\beta\ge \lfloor p \rfloor$. 
A function $y\in \cC^{\frac1p}\big([0,T];\RR^d\big)$ is a solution (in the sense of Davie) to the driftless Rough Differential Equation (RDE) \eqref{eq:rde_basic} if there exists $a>1$  such that for all $(s,t)\in \Delta^2_{T}$,
\[\Bigg|(y_t-y_s) - \Bigg(\sum_{\substack{k\in\{1,\cdots,\lfloor p \rfloor\} \\ I\in \{1,\cdots,d\}^k }} \bw^{k,I}_{s,t}\sigma_{I}(y_s)\Bigg) \Bigg| \lesssim_{\|\bw\|_{\frac{1}{p}-\text{H\"ol};[0,T]},\sigma} |t-s|^a.\]
\end{definition}

The existence, uniqueness and flow properties of solution to driftless RDE is given in \cite{gubinelliControllingRoughPaths2004,friz2010multidimensional,friz2020course,davieDifferentialEquationsDriven2007,bailleulNonexplosionCriteriaRough2020,
bailleulFlowsDrivenRough2015,bailleulDefinitionSolutionRough2021}. 

Note that, by using \cite{bailleulDefinitionSolutionRough2021,cassTreeAlgebrasTopological2017}, one can see the equivalence between Davie's notion of solutions to RDE and Bailleul's one. 
Namely, a path $y$ is a solution to Equation \eqref{eq:rde_basic} if and only if, for any $f\in C^{\lfloor p \rfloor + 1}_{lin}(\RR^d;\RR^d)$ and for all $(s,t)\in \Delta^2_{T}$,
\begin{equation}\label{eq:driftlessRDE_equivalent}\Bigg|\big(f(y_t)-f(y_s)\big) - \Bigg(\sum_{\substack{k\in\{1,\cdots,\lfloor p \rfloor\} \\ I\in \{1,\cdots,d\}^k }} \bw^{k,I}_{s,t}\sigma_{I}f(y_s)\Bigg) \Bigg| \lesssim_f \|f\|_{C^{\lfloor p \rfloor +1}_{lin}} |t-s|^a.
\end{equation}
Furthermore, by using \cite{bailleulNonexplosionCriteriaRough2020}, one has the following result.

\begin{theorem}\label{theorem:driftlessRDE}
Let $p\ge 2$,  $T>0$,  $n\ge 0$, $\bw$ be a $\frac{1}{p}$-Hölder weakly geometric $p$-rough path, and  $\sigma \in C^{\lfloor p \rfloor + n + 1}_b(\RR^d;(\RR^d)^{\otimes 2})$. Then, there is a unique 
\[\varphi:\Delta^2_{T} \mapsto C_{lin}^{n}(\RR^d;\RR^d)\]
and a unique 
\[\psi:\Delta^2_{T} \mapsto C_{lin}^{n}(\RR^d;\RR^d)\]
such that the following points hold.
\begin{enumerate}
\item \label{driftlessRDE:point1}
For all $s\in [0,T)$ and $x\in \RR^d$, $t:[s,T] \mapsto \varphi_{s,t}(x)$ is the unique solution to 
\[\dd x_t = \sigma(x_t) \dd \bw_t,\quad x_s  =x,\quad t\in[s,T].\]
\item \label{driftlessRDE:point2} For all $x\in \RR^d$ and all $(s,u),(u,t) \in \Delta^2_T$, 
\[\varphi_{u,t}(\varphi_{s,u}(x)) = \varphi_{s,t}(x) \quad \text{and} \quad \psi_{s,u}(\psi_{u,t}(x)) = \psi_{s,t}(x).\]
\item \label{driftlessRDE:point3} For all $(s,t)\in \Delta^2_T$, all $x\in\RR^d$  and all $0\le m\le n$, we have
\[
\Bigg|
D^m \varphi_{s,t}(x) - \Bigg((D^m id)(x)+\sum_{\substack{ k\in \{1,\cdots,\lfloor p \rfloor \}\\ I \in \{1,\cdots,d\}^k }} \bw^{k,I}_{s,t} (D^m\sigma_I)(x)\Bigg) 
\Bigg| 
\lesssim_{\|\bw\|_{\frac{1}{p}\text{-H\"ol},[0,T]},\sigma} |t-s|^{\frac{\lfloor p \rfloor + 1}{p}}.
\]
\item  \label{driftlessRDE:point4}
For all $(s,t) \in \Delta^2_T$ and all $x\in \RR^d$	,
\[\varphi_{s,t}(\psi_{s,t}(x)) = \psi_{s,t}(\varphi_{s,t}(x)) = x.\]
\item \label{driftlessRDE:point5}
For all $(s,t)\in \Delta^2_T$, all $x\in\RR^d$  and all $0\le m\le n$, we have
\[\Bigg|D^m \psi_{s,t}(x) - \Bigg((D^m id)(x)+\sum_{\substack{ k\in \{1,\cdots,\lfloor p \rfloor \}\\ I \in \{1,\cdots,d\}^k }} \bv^{k,I}_{s,t} (D^m\sigma_I)(x)\Bigg) \Bigg| \lesssim_{\|\bw\|_{\frac{1}{p}\text{-H\"ol},[0,T]},\sigma} |t-s|^{\frac{\lfloor p \rfloor +1}{p}}\]
where
\[\bv_{s,t} = \sum_{k=0}^{\lfloor p \rfloor} (\mathbf{1}-\bw_{s,t})^{\otimes k}.\]
\item
The maps from $\cC^{\frac1p}\big([0,T];G^{\lfloor p \rfloor}(\RR^d)\big)\big)$ to $\cC^{\frac{1}{p}}([0,T];C^n_{lin}(\RR^d;\RR^d))$ given by
\[\bw \mapsto \varphi\quad \text{and}\quad \bw \mapsto \psi\]
are continuous.
\end{enumerate}
\end{theorem}

\begin{remark}\label{remark:equivalent_sol_1}
Note that Theorem \ref{theorem:driftlessRDE} \eqref{driftlessRDE:point5} with $m=0$, \eqref{driftlessRDE:point2} and Definition \ref{def:davie_solutions}  is equivalent to say that for all $s_0\in[0,T)$ and for all $x\in \RR^d$, $t\mapsto \psi_{s_0,t}(x)$ is the unique solution to the driftless RDE
\[\dd y_t = \sigma(y_t) \dd \bv_t,\quad y_{s_0} = x,\quad t\in[s_0,T].
\]
In that setting, one can use the equivalent definition of solutions seen in Equation \eqref{eq:driftlessRDE_equivalent}.

Indeed let $s\le t \in (s_0,T]$. We have thanks to Theorem \ref{theorem:driftlessRDE} \eqref{driftlessRDE:point2}
\begin{equation*}
    y_t - y_s =  \psi_{s_0,t}(x) -  \psi_{s_0,s}(x) 
    =  \psi_{s,t}(\psi_{s_0,s}(x)) -  \psi_{s_0,s}(x),
\end{equation*}
which gives exactly Definition \ref{def:davie_solutions} thanks to Theorem \ref{theorem:driftlessRDE} \eqref{driftlessRDE:point5}.
\end{remark}

\begin{proof}
The proof is somehow classical. For a similar proof without the full Euler scheme expansion, one can consult \cite[Lemma 2.3]{riedelRoughDifferentialEquations2017} or \cite[Theorem 10.14 and Theorem 10.26]{friz2010multidimensional}.

The proof here is simply an adaptation of the one from \cite{bailleulNonexplosionCriteriaRough2020}. Indeed, one can combine Theorem 2.2, Corollary 3.5, Remark 3.6 and  Theorem 4.2 from \cite{bailleulNonexplosionCriteriaRough2020} to get the existence and uniqueness of $\varphi$, as well as the points \eqref{driftlessRDE:point1}, \eqref{driftlessRDE:point2} and \eqref{driftlessRDE:point3}. Note also that the continuity of $\varphi$ with respect to $\bw$ is also a consequence of the previous theorems.

To construct $\psi$ and prove the desired Euler expansion, let us remind the strategy of proof in \cite{bailleulNonexplosionCriteriaRough2020} (see also \cite{bailleulFlowsDrivenRough2015}). Since $\bw$ is in $G^{\lfloor p \rfloor}(\RR^d)$, there exists an $\bl \in L^{\lfloor p \rfloor}(\RR^d)$ such that 
\[\exp^{\oplus}(\bl) = \bw.\]
Remark that $\mathbf{l} \in L^{\lfloor p \rfloor}(\RR^d)$ and a basis of this space is 
\[e_{[i]}= e_i, \quad e_{[i,I]} = [e_i,e_{[I]}] = e_i \otimes e_{[I]} - e_{[I]}\otimes e_i.\]
Hence 
\[\mathbf{l} = \sum_{k=1}^{\lfloor p \rfloor} \sum_{I\in \{1,\cdots,d\}^k} \mathbf{l}^{k,[I]} e_{[I]}\]
and one can show that, in $\cC^{\frac{1}{p}}([0,T];C^n_{lin}(\RR^d;\RR^d))$, 
\[
\varphi_{s,t} 
	= 
\lim_{\substack{\pi \in \Pi([s,t])\\ |\pi|\to 0} }
\mu_{t_{N-1},t_N} \circ \cdots \circ \mu_{t_0,t_1},\]
where $\Pi([s,t])$ denotes the set of partitions of $[s,t]$,  $\pi=\{s=t_0 < t_1 \cdots < t_{N-1} < t_N=t\}$ is a partition of $(s,t)$ and $|\pi| = \sup_{i\in\{0,\cdots,N-1\}} |t_{i+1}-t_i|$ is the radius of $\pi$, and $\mu_{s,t} = y(1)$ where $y$ is the solution of the following ordinary differential equation :
\[y'(r) = \sum_{\substack{ k\in\{1,\cdots, \lfloor p \rfloor\} \\ I\in\{1,\cdots,d\}^k}}  \bl^{k,[I]}_{s,t} \sigma_{[I]}(y(r))  ,\quad y(0) = x, r\in[0,1].\]
Here, for $i\in\{1,\cdots,d\}$, $\sigma_{[i] }=\sigma_i$ and for $k\ge 1$, $i\in\{1,\cdots,d\}$ and $I\in\{1,\cdots,d\}^{k-1}$, $\sigma_{[(i,I)]} = [\sigma_i,\sigma_{[I]}] = \sigma_i \sigma_{[I]} - \sigma_{[I]} \sigma_i $.
Now, let us define the terminal value solution of the previous equation
\[z'(r) = \sum_{\substack{ k\in\{1,\cdots, \lfloor p \rfloor\} \\ I\in\{1,\cdots,d\}^k}}  \bl^{k,[I]}_{s,t} \sigma_{[I]}(z(r))  ,\quad z(1) = x, r\in[0,1]\] 
and
\[\nu_{s,t}(x) = z(0).\]
We have, for all $x\in \RR^d$ and all $(s,t) \in \Delta^2_T$,
\[\nu_{s,t}(\mu_{s,t})(x) = \mu_{s,t}(\nu_{s,t})(x) = x.\]
We  define, for any partition $\pi \in\Pi([s,t])$,
\[\nu^{\pi}_{s,t} = \nu_{t_{0},t_1} \circ \cdots \circ \nu_{t_{N-1},t_N}\]
and
\[\mu^{\pi}_{s,t} = \mu_{t_{N-1},t_N} \circ \cdots \circ \mu_{t_0,t_1}.\]
Then, we observe that
\[\mu^{\pi}_{s,t}(\nu^\pi_{s,t}(x))=\nu^{\pi}_{s,t}(\mu^\pi_{s,t}(x))=x.\]
Hence, one only has to prove that $\nu^{\pi}$ converges to a (backward semi-)flow in the suitable space. Note that this follows from the same arguments as the proof of Theorem 2.2 in \cite{bailleulNonexplosionCriteriaRough2020}, which guaranties that $\psi$ exists in $C^n_{lin}$, is continuous with respect to $\bw$ and that, for $(s,t)\in \Delta^2_T$ small enough (with respect to $\bw$), one has, for any $0\le m\le n$,
\[\sup_{x\in \RR^d}|D^m\psi_{s,t} - D^m\nu_{s,t}(x)| \lesssim_{\bw} (t-s)^{\frac{\lfloor p \rfloor + 1}{p}}.\]
Furthermore, one also have (see Lemmas 3.3, 3.4 and Corollary 3.5 in \cite{bailleulNonexplosionCriteriaRough2020}), for  $a=\frac{\lfloor p \rfloor +1}{p}$,

\[\Bigg|
	D^m\nu_{s,t}(x) 
		- 
	\Bigg(
		(D^m id)(x) 
			+ 
		\sum_{j=1}^{\lfloor p \rfloor} 
			\frac{(-1)^j}{j!} 
				\sum_{\substack{1\le k_1,\cdots, k_j \le \lfloor p \rfloor \\ \sum_{l=1}^{j} k_l \le \lfloor p \rfloor \\ I_1 \in \{1,\cdots,d\}^{k_1}\\ \cdots\\ I_j \in \{1,\cdots,d\}^{k_j} }}
				 	\prod_{l=1}^j \bl^{k_l,I_l}_{s,t}
				 		D^m\big(\sigma_{[I_j]}\cdots \sigma_{[I_1]}\big)(x) 
	\Bigg)
\Bigg| 
\lesssim_{\bw} 
(t-s)^{a}. \]
Hence, if we  define $\bv = \exp^{\oplus}(-\bl)$, we have
\[\Bigg|D^m\nu_{s,t}(x) - \Bigg((D^m id)(x) + \sum_{k=1}^{\lfloor p \rfloor} \sum_{I \in \in\{1,\cdots,d\}^k} \bv^{k,I}_{s,t} D^m\big(\sigma_{I}\big)(x) \Bigg)\Bigg| \lesssim_{\bw} (t-s)^{a}. \]
We can see that 
\[\bv_{s,t} \otimes \bw_{s,t} = \exp^{\oplus}(-\bl_{s,t}) \otimes \exp^{\oplus}(\bl_{s,t}) = \exp^{\oplus}(0) = \mathbf{1} = \bw_{s,t} \otimes \bv_{s,t}.\]
Furthermore, we also have that, for $k> \lfloor p \rfloor$,
\[(1-\bw_{s,t})^{\otimes k} = 0,\]
which yields
\[\bv_{s,t}= \sum_{k=0}^{\lfloor p \rfloor} (\mathbf{1} - \bw_{s,t} )^{\otimes k},\]
and the result follows.
\end{proof}

\begin{corollary}\label{corollary:rde_inverse_jacobian}
Let $n\ge 1$, $p,T,\bw$ and $\sigma$ as in the previous Theorem. Then 
\[(D\varphi)^{-1},(D\psi)^{-1} \in \cC^{\frac1p}\big([0,T];C^{n-1}_b(\RR^d;(\RR^d)^{\otimes 2})\big).\]
\end{corollary}

\begin{proof}
Let us remark that under the previous hypothesis, both $\varphi$ and $\psi$ are in $C^n_{lin}$. Hence
$D\varphi_{s,t}(x) D\psi_{s,t}\big(\varphi_{s,t}(x)\big) = x$, for any $x\in\mathbb{R}^d$.
It follows that
 $\big(D\varphi_{s,t}(x)\big)^{-1} = D\psi_{s,t}(\varphi_{s,t}(x))$,
and there exists a constant, depending on $\bw,T,\sigma$ such that 
\[\sup_{(s,t)\in \Delta^2_T}\frac{\left\|\big(D\varphi_{s,t}(x)\big)^{-1}\right\|_{C^{n-1}_b}}{|t-s|^{\frac1p}} \lesssim_{\bw} 1,\]
which ends the proof.
\end{proof}

	\subsection{Solution of RDE with drift : flow transform}
	
	We turn now to the study of rough differential equations with a drift:
	\begin{equation}\label{eq:rde_with_drift}
		\dd x_t = b(x_t) \dd t + \sigma(x_t) \dd \bw_t,\quad x_0 = x,
		\quad t\in [0,T].
	\end{equation} 
	In the spirit of Definition \ref{def:davie_solutions}, let us define the solutions of the rough differential equation with drift as follow
	
\begin{definition}\label{def:rde_with_drift}
Let $p\ge 2$, $T>0$ and $n\ge 0$. Let $b\in C^0_b(\RR^d;\RR^d)$, $\sigma \in C^{\lfloor p \rfloor}(\RR^d;(\RR^d)^{\otimes 2})$ and $\bw$ be a $\frac{1}{p}$-H\"older weakly geometric $p$-rough path. A path $(x_t)_{t\in [0,T]}\in\cC^{\frac1p}([0,T];\RR^d)$ is a solution to Equation \eqref{eq:rde_with_drift} if $x_0 = x$ and there exists a constant $a>1$ independent of $\bw$ such that for all $(s,t)\in \Delta^2_T$, 
\[\Bigg|x_{t} - \Bigg(x_s + b(x_s)(t-s) + \sum_{k=1}^{\lfloor p \rfloor } \sum_{I \in \{1,\cdots,d\}^k} \bw^{k,I}_{s,t} \sigma_I(x_s) \Bigg)\Bigg| \lesssim_{\bw,\sigma,b} |t-s|^a.\]
\end{definition}	

\begin{remark}\label{remark:equivalent_sol_2}
As proved in \cite{bailleulDefinitionSolutionRough2021} and \cite{cassTreeAlgebrasTopological2017} (we also refer to \cite{bailleulNonexplosionCriteriaRough2020} for the precise value of the following constants), the following equivalent \blue{notion of} solutions can be used :
Let $p,\bw,T,b,\sigma$ be as in the previous Definition. Then $(x_t)_{t\in[0,T]}$ is a solution to Equation \eqref{eq:rde_with_drift} if and only if there exists $a>1$ independent of $\bw,\sigma,b$, $\eps=\eps(\bw)>0$ such that for all $f\in C^{\lfloor p \rfloor +1}_b(\RR^d;\RR^d)$ and all $(s,t)\in \Delta^2_T$ with $|t-s|\le \eps$, 
\[\bigg|f(x_t) - \bigg((bf)(x_s)(t-s) + \sum_{\substack{k\in\{1,\cdots,\lfloor p \rfloor\}\\ I \in \{1,\cdots,d\}^k}} \bw^{k,I}_{s,t} (\sigma_I f)(x_s)\bigg)\bigg| \lesssim_{\bw,\sigma,b} \|f\|_{C^{\lfloor p \rfloor+1}_{lin}} |t-s|^a.\]
This point will be crucial in the following.
\end{remark}

Several results concerning the (optimal) regularity of the drift are available in order to have existence and uniqueness for rough differential equations. Namely, in \cite{friz2010multidimensional}, one can see that whenever $\sigma \in \cC_b^{\lfloor p \rfloor +1}$ and $b$ is globally Lipschitz continuous with linear growth, there is a unique solution to equation \eqref{eq:rde_with_drift}. Furthermore, some improvement of this criteria appears in     \cite{riedelRoughDifferentialEquations2017} (weak local Lipschitz condition and some control on the growth) and in \cite{bonnefoiPrioriBoundsRough2022} (Lyapounov conditions on $b$). In order to go beyond (using some stochasticity of the rough path), one needs to develop an other approach of the solution to the RDE with drift. Following a classical ideas of flow transform (see for example \cite{riedelRoughDifferentialEquations2017}) we give an equivalent formulation for the solutions of Equation \eqref{eq:rde_with_drift}. 
		
\begin{theorem}\label{theorem:main1}
Let $p\ge 2$, $T>0$, $\kappa>0$, $b\in \cC^\kappa(\RR^d;\RR^d)$, $\sigma \in C^{2(\lfloor p \rfloor+1)}_b(\RR^d;(\RR^d)^{\otimes 2})$ and $\bw$ be a $\frac{1}{p}$-H\"older weakly $p$-geometric rough path. Then $(x_t)_{t\in[0,T]}$ is a solution to the rough differential equation with drift \eqref{eq:rde_with_drift} if and only if 
\[x_t= \varphi_{0,t}(z(t))\]
and $(z(t))_{t\in[0,T]}$ is a solution to the ordinary differential equation 
\begin{equation}\label{eq:ODE_rough}
z'(t) = \big(D\varphi_{0,t}(z(t))\big)^{-1} b\big(\varphi_{0,t}(z(t))\big),\quad z_0 = x,\quad t\in[0,T],
\end{equation}
where $\varphi$ is the flow constructed in Theorem \ref{theorem:driftlessRDE} and Corollary \ref{corollary:rde_inverse_jacobian}.
\end{theorem}

\begin{proof}
First, we remark that since $\sigma \in C^{\lfloor p \rfloor+1}_b(\RR^d;(\RR^d)^{\otimes 2})$, thanks to Theorem \ref{theorem:driftlessRDE} and Corollary \ref{corollary:rde_inverse_jacobian},
\[\Big((t,x)\mapsto \big(D\varphi_{0,t}(x)\big)^{-1} b\big(\varphi_{0,t}(x)\big) \Big) \in C^{0}_b([0,T];C^0(\RR^d;\RR^d)).\]
Hence, it follows from Peano's existence theorem that there exists a solution $(z(t))_{t\in[0,T]}$ to Equation \eqref{eq:ODE_rough}. Let us define for all $t\in[0,T]$, $x_t=\varphi_{0,t}(z(t))$. Note that, thanks to the hypothesis and Corollary \ref{corollary:rde_inverse_jacobian},
\[|z(t) -x| \lesssim_{\bw,\sigma,T} 1,\]
and  
\[|z(t) -z(s)| \lesssim_{\bw,\sigma,T} |t-s|.\]
Furthermore, note that for any $s,t\in \Delta^2_T$,
\begin{align}
z(t) - z(s) 
= &
\int_s^t 
	\left(
		D\varphi_{0,r}(z(r))
	\right)^{-1} b\big(\varphi_{0,r}(z(r))\big) \dd r 
\nonumber\\
=&
\left(D\varphi_{0,s}(z(s))\right)^{-1} b\big(\varphi_{0,s}(z(s))\big)(t-s) \nonumber\\
& +
\int_s^t \Big(\left(D\varphi_{0,r}(z(r))\right)^{-1}-\left(D\varphi_{0,s}(z(r))\right)^{-1}\Big) b\big(\varphi_{0,r}(z(r))\big) \dd r
\label{eq:rde_equiv_direct_1}\\
& +
\int_s^t \Big(\left(D\varphi_{0,s}(z(r))\right)^{-1}-\left(D\varphi_{0,s}(z(s))\right)^{-1}\Big) b\big(\varphi_{0,r}(z(r))\big) \dd r
\label{eq:rde_equiv_direct_2}\\
& +
\left(D\varphi_{0,s}(z(s))\right)^{-1}\int_s^t \left( b\big(\varphi_{0,r}(z(r))\big) -  b\big(\varphi_{0,s}(z(s))\big)\right) \dd r\label{eq:rde_equiv_direct_3}
\end{align}

Note that thanks to Corollary \ref{corollary:rde_inverse_jacobian}, $(t,x)\mapsto \big( D\varphi_{0,t}(x)\big)^{-1} \in \cC^{\frac{1}{p}}([0,T];C^1_b(\RR^d;(\RR^d)^{\otimes 2}))$. Hence, since $b$ is bounded, the following estimates hold
\[|\eqref{eq:rde_equiv_direct_1}| \lesssim_{\bw,\sigma,T,b} |t-s|^{1+\frac{1}{p}},\]
\[|\eqref{eq:rde_equiv_direct_2}| \lesssim_{\bw,\sigma,T,b} |t-s|^{2}\]
and, since $b\in \cC^{\blue{\kappa}}(\RR^d;\RR^d)$, $r\mapsto\psi_{0,s}(z)$ is H\"older continuous uniformly in $z$ and $r\mapsto z(r)$ is H\"older continuous, we deduce
\[|\eqref{eq:rde_equiv_direct_3}| \lesssim_{\bw,\sigma,T,b} |t-s|^{1+\kappa}.\]
Let us define $R'_{s,t}(x) = \eqref{eq:rde_equiv_direct_1} + \eqref{eq:rde_equiv_direct_2} + \eqref{eq:rde_equiv_direct_3}$.

Furthermore, let us remind that we have for all $(s,t)\in\Delta^2_T$,
\[\varphi_{s,t}(x) = x + \sum_{\substack{k\in\{1,\cdots,\lfloor p \rfloor\}}} \bw^{k,I}_{s,t} \sigma_I(x) + R_{s,t}(x)\]
with
\[\sup_{x\in \RR^d} |R_{s,t}(x)| \lesssim_{\bw,\sigma,T} |t-s|^{\frac{\lfloor p \rfloor + 1}{p}}.\]
Finally
\begin{align*}
x_t - x_s 
= & 
	\bigg(\varphi_{s,t}\Big(\varphi_{0,s}\big(z(t)\big)\Big)-\varphi_{0,s}\Big(z(t)\Big) \bigg)
+ 
	\bigg(\varphi_{0,s}\big(z(t)\big)-\varphi_{0,s}\big(z(s)\big) \bigg)
\\
= & 
	\sum_{\substack{k\in\{1,\cdots,\lfloor p \rfloor\} \\ I\in\{1,\cdots,d\}^k}} \bw^{k,I}_{s,t} \sigma_I(\varphi_{0,s}(z(s)))
\\
& +
	\sum_{\substack{k\in\{1,\cdots,\lfloor p \rfloor\} \\ I\in\{1,\cdots,d\}^k}} \bw^{k,I}_{s,t} \big(\sigma_I(\varphi_{0,s}(z(t)))-\sigma_I(\varphi_{0,s}(z(s))) \big)
+ R_{s,t}(\varphi_{0,s}(z(t)))
\\
&+
	D\varphi_{0,s}(z(s))\left(D\varphi_{0,s}(z(s))^{-1}b\big(\varphi_{0,s}(z(s))\big) + R'_{s,t}(x)\right) 
\\
+& \int_0^1 \int_{0}^1 \lambda \left(D^2\varphi_{0,s}\big(\lambda\mu(z(t)-z(s)) + z(s) \big)\right)(z(t)-z(s))^{\otimes 2} \dd \mu \dd \lambda 
\\
=& b\big(\varphi_{0,s}(z(s))\big) + \sum_{\substack{k\in\{1,\cdots,\lfloor p \rfloor\} \\ I\in\{1,\cdots,d\}^k}} \bw^{k,I}_{s,t} \sigma_I(\varphi_{0,s}(z(s))) + R''_{s,t}(x),
\end{align*}
where 
\begin{align*}
R''_{s,t}(x) 
=& R_{s,t}(\varphi_{0,s}(z(t))) + D \varphi_{0,s}(z(s)) R'_{s,t}(x) 
\\ &+ \sum_{\substack{k\in\{1,\cdots,\lfloor p \rfloor\} \\ I\in\{1,\cdots,d\}^k}} \bw^{k,I}_{s,t} \big(\sigma_I(\varphi_{0,s}(z(t)))-\sigma_I(\varphi_{0,s}(z(s))) \big)
\\ &+
\int_0^1 \int_{0}^1 \lambda \left(D^2\varphi_{0,s}\big(\lambda\mu(z(t)-z(s)) + z(s) \big)\right)(z(t)-z(s))^{\otimes 2} \dd \mu \dd \lambda.
\end{align*}
Thanks to the hypothesis and the previous computations, there exists $a>1$ such that 
\[\sup_{x\in\RR^d} |R''_{s,t}(x)| \lesssim_{\bw,\sigma,b,T} |t-s|^a,\]
which proves that whenever $(z(t))_{t\in[0,T]}$ is a solution of \eqref{eq:ODE_rough}, $(x_t)_{t\in[0,T]} = (\varphi_{0,t}(z(t)))_{t\in[0,T]}$ is a solution of Equation \eqref{eq:rde_with_drift}.

Now, let us take $(x_t)_{t\in[0,T]}$ a solution of Equation \eqref{eq:rde_with_drift}. Let us denote by $(z(t))_{t\in[0,T]} = \big(\psi_{0,t}(x_t)\big)_{t\in[0,T]}$ and let us prove that $ (z(t))_{t\in[0,T]}$ is a solution of \eqref{eq:ODE_rough}. In the following, we will crucially use Remarks \ref{remark:equivalent_sol_1} and \ref{remark:equivalent_sol_2}, and  the fact that since $\sigma \in C^{2(\lfloor p \rfloor + 1)}$,
 for all $s\in[0,T]$, 
\[\|\psi_{0,s}\|_{C^{\lfloor p \rfloor + 1}_{lin}} \lesssim_{\bw,T,\sigma} 1.\]

We have, for all $(s,t)\in \Delta^2_T$ small enough (depending on $\bw$),
\begin{align*}
z(t) - z(s) =& \psi_{0,t}(x_t) - \psi_{0,s}(x_s) 
\\
=&
\psi_{0,s}(\psi_{s,t}(x_t)) - \psi_{0,s}(x_t)
+
\psi_{0,s}(x_t)-\psi_{0,s}(x_s) \\
=&
\sum_{\substack{ k\in\{1,\cdots,\lfloor p \rfloor\} \\ I \in\{1,\cdots,d\}^k  }}\bv^{k,I}_{s,t} \big((\sigma_I \psi_{0,s})(x_t)-(\sigma_I \psi_{0,s})(x_s)\big) + R_{s,t}
\\
&+
\sum_{\substack{ k\in\{1,\cdots,\lfloor p \rfloor\} \\ I \in\{1,\cdots,d\}^k  }} \bv^{k,I}_{s,t} (\sigma_I \psi_{0,s})(x_s)
\\
& + (b\psi_{0,s})(x_s)(t-s) + \sum_{\substack{ k\in\{1,\cdots,\lfloor p \rfloor\} \\ I \in\{1,\cdots,d\}^k  }} \bw^{k,I}_{s,t} (\sigma_I \psi_{0,s})(x_s) + R'_{s,t}\\
=&
\sum_{\substack{ k\in\{1,\cdots,\lfloor p \rfloor\} \\ I \in\{1,\cdots,d\}^k  }} \bv^{k,I}_{s,t}\left((b \sigma_I\psi_{0,s})(x_s)(t-s)\sum_{\substack{ k'\in\{1,\cdots,\lfloor p \rfloor\} \\ I' \in\{1,\cdots,d\}^k  }}\bw^{k',I'}_{s,t} (\sigma_{I'}\sigma_I \psi_{0,s})(x_s)+R''_{s,t}\right)
\\
&+
\sum_{\substack{ k\in\{1,\cdots,\lfloor p \rfloor\} \\ I \in\{1,\cdots,d\}^k  }} \bv^{k,I}_{s,t} (\sigma_I \psi_{0,s})(x_s) + \sum_{\substack{ k\in\{1,\cdots,\lfloor p \rfloor\} \\ I \in\{1,\cdots,d\}^k  }} \bw^{k,I}_{s,t} (\sigma_I \psi_{0,s})(x_s)\\
&+ (b\psi_{0,s})(x_s)(t-s)+ R_{s,t} + R'_{s,t}.
\end{align*}
Here, we have 
\[\sup_{x\in\RR^d} \big(|R_{s,t}| +|R'_{s,t}|+|R''_{s,t}|\big) \lesssim_{\bw,\sigma,b,T} |t-s|^a \]
for a certain $a>1$. Furthermore, note that since $b$ is bounded, $\sigma\in C^{2(\lfloor p \rfloor  +1 )}_b$ and $\psi_{0,s} \in C^{\lfloor p \rfloor +1}_{lin}$, for all $k'\in \{1,\cdots,\lfloor p \rfloor + 1\}$ and for all $I,I' \in \{1,\cdots, d \}^k$, 
\[\sup_{x\in\RR^d} |(b\sigma_I\psi_{0,s})(x_s)(t-s)|\lesssim |t-s| \quad \text{and} \quad \sup_{x\in\RR^d} |\bw^{k',I'}_{s,t} (\sigma_{I'}\sigma_I \psi_{0,s})(x_s)| \lesssim |t-s|^{\frac{k'}{p}},\]
where the previous bounds depend on $\bw,\sigma,T,b$. 
Hence, there exists $a>1$ and $\tilde{R}_{s,t}$ with $\sup_{x\in\RR^d} |\tilde{R}_{s,t}| \lesssim |t-s|^a$
such that

\begin{align*}
z(t) -z(s) 
=&
(b\psi_{0,s})(x_s) + \tilde{R}_{s,t} + 
\sum_{\substack{k\in \{1\cdots,\lfloor p \rfloor\} \\ j\in\{1,\cdots,k-1\} \\ J\in\{1,\cdots,d\}^j\\J'\in\{1,\cdots,d\}^{k-j}}} \bv^{j,J}_{s,t}\bw^{k-j,J'}_{s,t} (\sigma_{J}\sigma_{J'} \psi_{0,s})(x_s)
\\
&+
\sum_{\substack{ k\in\{1,\cdots,\lfloor p \rfloor\} \\ I \in\{1,\cdots,d\}^k  }} \bv^{k,I}_{s,t} (\sigma_I \psi_{0,s})(x_s) + \sum_{\substack{ k\in\{1,\cdots,\lfloor p \rfloor\} \\ I \in\{1,\cdots,d\}^k  }} \bw^{k,I}_{s,t} (\sigma_I \psi_{0,s})(x_s)
\\
=& (b\psi_{0,s})(x_s) + \tilde{R}_{s,t} + \sum_{\substack{k\in\{1,\cdots,\lfloor p \rfloor\} \\ I\in\{1,\cdots,d\}^k}} \big(\bv_{s,t}\otimes \bw_{s,t}\big)^{k,I} \big(\sigma_I \psi_{0,s}\big)(x_s).
\end{align*}
Note that since $\bv_{s,t}\otimes \bw_{s,t} = \mathbf{1}$, we have for all $k\ge 1$, $(\bv_{s,t}\otimes \bw_{s,t})^k=0$. 
 
Furthermore $D\psi_{0,s}(x) = \big(D\varphi_{0,s}(\psi_{0,s}(x))\big)^{-1}$, hence
\[(b\psi_{0,s})(x_s) = \big(D\varphi_{0,s}(\psi_{0,s}(x_s))\big)^{-1} b\left(\varphi_{0,s}(\psi_{0,s}(x_s))\right)= \big(D\varphi_{0,s}(z(s))\big)^{-1} b\left(\varphi_{0,s}(z(s))\right).\]
Finally, for $(s,t)\in \Delta^2_T$ small enough
\[z(t) - z(s) = \big(D\varphi_{0,s}(z(s))\big)^{-1} b\left(\varphi_{0,s}(z(s))\right) + \tilde{R}_{s,t}.\]
We also know that $s\mapsto \big(D\varphi_{0,s}(z(s))\big)^{-1} b\left(\varphi_{0,s}(z(s))\right) $ is a continuous functions. Hence, thanks to a standard Riemann sum argument, 
\[z(t) - z(s) = \int_s^t \big(D\varphi_{0,r}(z(r))\big)^{-1} b\left(\varphi_{0,r}(z(r))\right)\dd r\]
and $(z(t))_{t\in[0,T]}$ is a solution of Equation \eqref{eq:ODE_rough}, and $(x_t)_{t\in[0,T]} = (\varphi_{0,t}(z(t))\big)_{t\in[0,T]}$, which ends the proof.
\end{proof}

\subsection{Rough differential equation with drift : the averaged field}

Whenever $\sigma$ is regular (and bounded), and $b$ is a bit more than continuous, Theorem \ref{theorem:main1} allows us to focus, in order to prove existence and uniqueness of solutions to  rough differential equations with drift
\[\dd x_t = b(x_t) \dd t + \sigma(x_t) \dd \bw_t, \quad x_0 \in \RR^d,\quad t\in[0,T],\]
to existence and uniqueness of solutions, the ordinary differential equation 
\[z'(t) = \big(D\varphi_{0,t}(z(t))\big)^{-1} b\big(\varphi_{0,t}(z(t))\big),\quad z_0 = x_0 \in \RR^d,\quad t\in[0,T],  \]
where $\varphi$ is the flow generated by the driftless RDE \eqref{eq:rde_basic}.

The previous equation may be rewritten in its integral form as 
\begin{equation}\label{eq:integral_equation}
z(t) = x_0 + \int_0^t \big(D\varphi_{0,r}(z(r))\big)^{-1} b\big(\varphi_{0,r}(z(r))\big) \dd r, \quad t\in[0,T]. 
\end{equation}

In order to study wellposedness of the previous equation, we rely on ideas from \cite{catellierAveragingIrregularCurves2016} and we will try to exhibit a general criterion of regularity of the space-time averaged fields defined in the following definition~:

\begin{definition}
Let $p\ge 2$, $T>0$, $\sigma\in C^{\lfloor p \rfloor + 2}_b(\RR^d;(\RR^d)^{\otimes 2})$ and $b\in\cC^\kappa_b(\RR^d;\RR^d)$. We define the averaged field $\cT b$ of $b$ along the flow generated by the driftless rough differential equation driven by $\bw$ has the following space-time vector field :
\begin{equation}\label{eq:averaged_field}
(t,x)\in[0,T] \times \RR^d \mapsto \cT b_t(x):=\int_{0}^t \big(D\varphi_{0,r}(x)\big)^{-1} b\big(\varphi_{0,r}(x)\big) \dd r.
\end{equation}
For all $(s,t)\in\Delta^2_T$ we denote by $\cT b_{s,t}=\cT b_t - \cT b_s$.
\end{definition}

When $\sigma = 1$ and $\bw$ is a Brownian motion (or fractional Brownian motion), one can see (in \cite{davieUniquenessSolutionsStochastic2007} and \cite{catellierAveragingIrregularCurves2016} for example) that the previous averaged field enjoys better regularity properties (in space) than $b$. On can also consult \cite{galeatiNoiselessRegularisationNoise2021} and \cite{galeatiPrevalenceRhoIrregularity2020} for a systematic study of the previous averaged field in a more general context. 

In our setting, one generally use non-linear Young integration theory (see \cite{catellierAveragingIrregularCurves2016} and \cite{galeatiNonlinearYoungDifferential2021}) to link the space-time regularity of the averaged field \eqref{eq:averaged_field} and the wellposedness of the integral equation \eqref{eq:integral_equation}. Nevertheless the examples that our techniques will allow us to treat (Gaussian rough path, see Section \ref{section:Malliavin}) will not allow us to go beyond $b\in\cC^\kappa$ for $\kappa>0$. In that case, Equation \eqref{eq:integral_equation} always makes sense, it has a solution (thanks to Peano's existence theorem) and one only needs to focus on uniqueness.  The following theorem is our main result in a general context~:

\begin{theorem}\label{theorem:main2}
Let $p\ge 2$, $T>0$, $\kappa>0$, $b\in \cC^{\kappa}(\RR^d;\RR^d)$, $\sigma\in C^{2(\lfloor p \rfloor + 1)}_b(\RR^d;(\RR^d)^{\otimes 2})$ and $\bw$ be a $\frac{1}{p}$-H\"older weakly-geometric $p$ -rough path. Assume that there exists $1\ge\nu>\frac12$, $\alpha>0$ such that $\kappa + \frac{1}{\alpha}>2-\nu$, a weight $w_0$ (see Definitions \ref{def:weight} and \ref{def:weighted_holder}) such that 
\[\|\cT b\|_{\cC^\nu_T \cC^{\kappa+\frac1\alpha}_{w_0}} < +\infty.\]
Then there is a unique solution to Equation \eqref{eq:integral_equation}. 

Furthermore, let $x_1, x_2 \in \RR^d$ and $b_1, b_2$ satisfying the previous hypothesis, and let $z_1$ (respectively $z_2$) be the unique solution to Equation \eqref{eq:integral_equation} with $b=b_1$ and $x=x_1$ (respectively $b=b_2$ and $x=x_2$). Let us suppose that $\cT (b_1)-\cT (b_2) \in \cC^\nu_T \cC^{1}_{w_0}$. 

Then there is a non decreasing positive function 
$K:\RR_+ \mapsto \RR_+\backslash\{0\}$, which depends on $\|\cT b_1\|_{\cC^\nu_T\cC^{\kappa+\frac1\alpha}_{w_0}}$, $\|\cT b_2\|_{\cC^\nu_T\cC^{\kappa+\frac1\alpha}_{w_0}}$, $\bw$, $\|b_1\|_\infty$, $\|b_2\|_\infty$, $\sigma$ and $T$
such that 
\[\|z_1 - z_2\|_{\infty,[0,T]} \le K\big(|x_1|+ |x_2|\big) \big(|x_1-x_2| + \|\cT b_1 - \cT b_2\|_{\cC^\nu_T\cC^1_{w_0}}).\]

In this setting, there is a unique solution to the RDE with drift \eqref{eq:rde_with_drift} in the sense of the Definition \ref{def:rde_with_drift}. Furthermore it generates a locally Lipschitz-continuous semiflow with respect to the initial condition.
\end{theorem}

The proof is an adaptation of \cite[Theorem 2.21]{catellierAveragingIrregularCurves2016}. We also refer to Section 5 (and in particular to Theorem 5.6) of \cite{galeatiNonlinearYoungDifferential2021}. In order to be self-contained, we give here another proof, relying on the standard sewing lemma (see for example \cite{gubinelliControllingRoughPaths2004, feyelNoncommutativeSewingLemma2007,braultNonlinearSewingLemma2019}  and the reference therein). We recall here its standard formulation :

\blue{
\begin{lemma}\label{lemma:sewing}
Let $V$ be a Banach space and $A:\Delta^2_T \mapsto V$ be such that there exists two constants $C>0$ and $a>1$ such that for all $(s,u),(u,t) \in \Delta^2_T$, 
\[|A_{s,t} - A_{s,u} - A_{u,t}|_V \le C|t-s|^a\]
Then there exists a unique function $\cA : [0,T] \mapsto V$ such that $\cA_0=0$ and for all $(s,t)\in \Delta^2_T$,
\[\cA_t - \cA_s = \lim_{\substack{\pi \in \Pi([s,t])\\ |\pi|\to 0}} \sum_{(u,v) \in \pi} A_{u,v}.\]
Furthermore there exists a constant $k=k(a)>0$ such that for all $(s,t)\in\Delta^2_T$,
\[|\cA_t - \cA_s - A_{s,t}|_{V} \le k C|t-s|^a.\]
\end{lemma}
}

The end of this section is devoted to the proof of Theorem \ref{theorem:main2}. We begin by the following lemma.

\begin{lemma}\label{lemma:riemann} Let us suppose that the hypothesis of Theorem \ref{theorem:main2} are satisfied. 
Let $(z(t))_{t\in[0,T]}$ be a continuous functions. Then for all $(s,t)\in \Delta^2_T$,
\[\int_s^t \big(D\varphi_{0,r}(z(r))\big)^{-1} b\big(\varphi_{0,r}(z(r))\big) \dd r = \lim_{\substack{\pi \in \Pi([s,t])\\ |\pi|\to 0}} \sum_{(u,v) \in \pi} \cT b_{u,v}(z(u)).\]
\end{lemma}

\begin{proof}
First, let us remark that thanks to Theorem \ref{theorem:driftlessRDE}, for all $(s,t)\in\Delta^2_T$, $x\mapsto \varphi_{s,t}(x)$ and $x\mapsto \big(D\varphi_{s,t}(x)\big)^{-1}$ are Lipschitz continuous uniformly in time, and $x\mapsto \big(D\varphi_{s,t}(x)\big)^{-1}$ is bounded. Furthermore,  since $b\in\cC^\kappa_b$ for some $\alpha>0$, we have
\begin{align*}
\bigg|\int_s^t \big(D\varphi_{0,r}(z(r))\big)^{-1} &b\big(\varphi_{0,r}(z(r))\big) \dd r - \cT b_{s,t}(z(s))\bigg| 
\\
&\le
\left|\int_s^t \big(D\varphi_{0,r}(z(r))\big)^{-1}\Big( b\big(\varphi_{0,r}(z(r))\big) - b\big(\varphi_{0,r}(z(s))\big)\Big)\dd r\right|
\\
&\quad + 
\left|\int_s^t \Big(\big(D\varphi_{0,r}(z(r))\big)^{-1}-D\varphi_{0,r}(z(s))\big)^{-1}\Big) b\big(\varphi_{0,r}(z(s))\big) \dd r\right|
\\
&\lesssim_{\bw}
\|b\|_{\cC^{\kappa}_b}\int_{s}^t |z(r) - z(s)|^\kappa \dd r.
\end{align*}
Since $z$ is uniformly continuous on $[0,T]$ the result follows by standard arguments.
\end{proof}

Until the end of this Section, we write $\kappa'=\kappa + \frac1\alpha$.

\begin{lemma} Again, let us work in the setting of Theorem \ref{theorem:main2}.
Let $z_1,z_2\in \cC^1_b([0,T];\RR^d)$ be two Lipschitz continuous paths. There exists a positive, non-decreasing locally bounded function 
$K_0:\RR_+ \mapsto \RR_+\backslash\{0\}$ depending on 
$\|\cT b_1\|_{\cC^\nu_T\cC^{\kappa'}_{w_0}}$, $\|\cT b_2\|_{\cC^\nu_T\cC^{\kappa'}_{w_0}}$, $T$ such that for all $(s,t)\in\Delta^2_T$, we have the following bound 
\begin{multline*}
\bigg| \int_s^t \big(D\varphi_{0,r}(z_1(r))\big)^{-1} b_1\big(\varphi_{0,r}(z_1(r))\big) \dd r - \int_s^t \big(D\varphi_{0,r}(z_2(r))\big)^{-1} b_2\big(\varphi_{0,r}(z_2(r))\big) \dd r \\ - \big((\cT b_1)_s(z_1(s)) - (\cT b_2)_s(z_2(s))\big)\bigg| \\
\lesssim
K_0\big(\|z_1\|_{\cC^1_T}+\|z_2\|_{\cC^1_T}\big)
\left(\|\cT (b_1-b_2)\|_{\cC^\nu_T \cC^1_{w_0}}+\|z_1 - z_2\|_{\infty,[s,t]} + \llbracket z_1 - z_2\rrbracket_{\nu,[s,t]} \right)|t-s|^a,
\end{multline*}
where $\tilde{w_0}(x) = (1+x )w_0(x)$, $a=\min\{2\nu,\kappa'-1 +\nu\}$ and for a path $z:[0,T] \mapsto \RR^d $
\[\llbracket z\rrbracket_{\nu,[s,t]} = \sup_{(r,r')\in\Delta^2_{s,t}} \frac{|z(r')-z(r)|}{|r'-r|^\nu}.\]
\end{lemma}

\begin{proof}
Let us define for $(s,t) \in \Delta^2_T$, 
\[A_{s,t} = (\cT b_1)_{s,t}(z_1(s)) - (\cT b_2)_{s,t}(z_2(s)).\]
Thanks to the previous lemma, we already know that each integral is the limit of the Riemann sum involving $\cT b_1,\cT b_2$, hence
\begin{multline*}
\int_s^t \big(D\varphi_{0,r}(z_1(r))\big)^{-1} b_1\big(\varphi_{0,r}(z_1(r))\big) \dd r - \int_s^t \big(D\varphi_{0,r}(z_2(r))\big)^{-1} b_2\big(\varphi_{0,r}(z_2(r))\big) \dd r \\= \lim_{\substack{\pi \in \Pi([s,t])\\ |\pi|\to 0}} \sum_{(u,v) \in \pi} A_{u,v}.
\end{multline*}
Furthermore, we have for $(s,u),(u,t) \in \Delta^2_T$,
\begin{align*}
A_{s,u} + A_{u,t} - A_{s,t}
=&
\Big((\cT b_1)_{u,t}(z_1(u)) - (\cT b_2)_{u,t}(z_2(u))\Big) - \Big((\cT b_1)_{u,t}(z_1(s)) - (\cT b_2)_{u,t}(z_2(s))\Big)
\\
&=
(\cT b_1)_{u,t}(z_1(u)) - (\cT b_1)_{u,t}(z_2(u))
\\
& - \Big((\cT b_1)_{u,t}\big(z_1(u)-z_2(u) + z_2(s)\big) - (\cT b_1)_{u,t}\big(z_2(s)\big)\Big)\\
& + (\cT b_1)_{u,t}\big(z_1(u)-z_2(u) + z_2(s)\big) - (\cT b_1)_{u,t}\big(z_1(s)\big) \\
& + \cT (b_1 - b_2)_{u,t}\big(z_2(u)\big)-\cT (b_1 - b_2)_{u,t}\big(z_2(s)\big).
\end{align*}
We deduce that
\begin{align*}
&A_{s,u} + A_{u,t} - A_{s,t}\\
&=
\int_0^1 \bigg( D (\cT b_1)_{u,t}\Big(\lambda\big(z_1(u) - z_2(u)\big) + z_2(u) \Big) -D (\cT b_1)_{u,t}\Big(\lambda\big(z_1(u) - z_2(u)\big) + z_2(s)\Big)\bigg) \\ &\quad\quad\quad\quad\quad\quad \times\big(z_1(u) - z_2(u)\big) \dd \lambda \\
&\quad+ 
\int_0^1 \bigg( D (\cT b_1)_{u,t}\Big(\lambda\big((z_1-z_2)(u)-(z_1-z_2)(s)\big) + z_2(s)\Big)\bigg)\big((z_1-z_2)(u)-(z_1-z_2)(s)\big) \dd \lambda \\
&\quad+
\cT (b_1 - b_2)_{u,t}\big(z_2(u)\big)-\cT (b_1 - b_2)_{u,t}\big(z_2(s)\big).
\end{align*}
It follows that
\begin{align*}
|A_{s,u}+A_{u,t} - A_{s,t}| 
\lesssim &
\|D(\cT b_1)\|_{\cC^\nu_T \cC^{\kappa'-1}_{w_0}}
w_0(\|z_1\|_{\infty,[0,T]}+\|z_2\|_{\infty,[0,T]}) \llbracket z_2 \rrbracket^{\kappa' - 1}_{\cC^1_T} \\ & \hspace{15em} \times \|z_1 - z_2\|_{\infty,[s,t]} |t-s|^{\nu + \kappa' - 1} \\
& + \|D(\cT b_1)\|_{\cC^\nu_T \cC^{\kappa'-1}_{w_0}} w_0(\|z_1\|_{\infty,[0,T]} + \|z_2\|_{\infty,[0,T]}) \llbracket z_1 - z_2 \rrbracket_{\cC^\nu_T} |t-s|^{2\nu}\\
& + 
\|\cT (b_1-b_2)\|_{\cC^\nu_T \cC^{1}_{w_0}} w_0(\|z_2\|_{\infty,[0,T]})\llbracket z_2 \rrbracket_{\cC^\nu_T} |t-s|^{2\nu}.
\end{align*}
Since the previous computation is symmetric in $z_1$ and $z_2$ by using the Sewing Lemma \ref{lemma:sewing}, we obtain the desired result.
\end{proof}

\begin{remark} \label{rem:valueK0}
Thanks to the previous proof, one can choose $K_0$ as follow :  
\[K_0(x) = c\big(1+\|\cT b_1\|_{\cC^\nu_T\cC^{\kappa'}_{w_0}} + \|\cT b_2\|_{\cC^\nu_T\cC^{\kappa'}_{w_0}}\big)(1+x){w}_{0}\big(x\big),\]
for some constant $c=c(a)>0$.
\end{remark}

\begin{proof}[Proof of Theorem \ref{theorem:main2}]
Let us consider directly that we have $b_1$ and $b_2$ and $x_1$ and $x_2$ as in the theorem. 
Let us remark that thanks to Peano's existence theorem, there exists $z_1$ and $z_2$ solutions of the corresponding integral equations. Furthermore, we have 
\[z_1(t) = x_1 + \int_s^t \big(D\varphi_{0,r}(z_1(r))\big)^{-1} b_1\big(\varphi_{0,r}(z_1(r))\big) \dd r.\]
Since $b_1$  and $D\varphi$  are bounded, there exists a constant $C_1 = C_1(\bw,b_1,\sigma,T)$ such that
\[\|z_1\|_{\cC^1_T} \le C_1( |x_1| + 1),\]
and the same holds for $z_2$ (with the corresponding constant $C_2$). Let us define
\[K = c \Big(\big(1+\|\cT b_1\|_{\cC^\nu_T\cC^{\kappa'}_{w_0}} + \|\cT b_2\|_{\cC^\nu_T\cC^{\kappa'}_{w_0}}\big)\tilde{w}_{0}\big( 2 (C_1 + C_2)(1 + |x_1|+|x_2|)\Big),\]
where $c>0$ is the constant in Remark \ref{rem:valueK0}. 
Hence, for all $(s,t)\in \Delta^2_T$, the following estimate holds
\begin{multline*}
\big|\big(z_1(t) - z_2(t)\big)-\big(z_1(s) - z_2(s)\big)\big|
\le 
K(|z_1(s) - z_2(s)| + \|\cT (b_1-b_2)\|_{\cC^\nu_T L^{\infty}_{w_0}})|t-s|^\nu \\
 + K\big(\|\cT (b_1-b_2)\|_{\cC^\nu_T \cC^1_{w_0}}+\|z_1 - z_2\|_{\infty,[s,t]} + \llbracket z_1 - z_2\rrbracket_{\nu,[s,t]} \big)|t-s|^{a}.
\end{multline*}
Here the second line comes directly from the previous Lemma and Remark \ref{rem:valueK0} whereas the first line comes form the estimate for $(\cT b_1)_{s,t}(z_1(s)) -(\cT b_2)_{s,t}(z_1(s))$.

Let $h>0$ such that $\left(h^\nu+h^{a-\nu} \right) K \le \frac13$ and assume that $t-s \le h$. Then, we have
\begin{multline*}
\llbracket z_1 - z_2 \rrbracket_{\nu,[s,t]} \le \frac{1}{3(t-s)^{\nu}}(\|z_1 - z_2\|_{\infty,[s,t]} + \|\cT (b_1-b_2)\|_{\cC^\nu_T L^{\infty}_{w_0}}) \\+ \frac13\big(\|\cT (b_1-b_2)\|_{\cC^\nu_T \cC^1_{w_0}}+\|z_1 - z_2\|_{\infty,[s,t]} + \llbracket z_1 - z_2\rrbracket_{\nu,[s,t]} \big),
\end{multline*}
as well as
\[\llbracket z_1 - z_2 \rrbracket_{\nu,[s,t]} \le \frac{1}{2}\left(1+\frac{1}{(t-s)^{\nu}}\right)(\|z_1 - z_2\|_{\infty,[s,t]} + \|\cT (b_1-b_2)\|_{\cC^\nu_T \cC^{1}_{w_0}}),\]
and
\[K|t-s|^a\llbracket z_1 - z_2 \rrbracket_{\nu,[s,t]} \le \frac49(\|z_1 - z_2\|_{\infty,[s,t]} + \|\cT (b_1-b_2)\|_{\cC^\nu_T \cC^{1}_{w_0}}).\]
Finally, we have, by injecting this inequality into the previous one
\[
|z_1(t)-z_2(t)| \le \frac43|z_1(s)-z_2(s)| + \frac59 \big( \|\cT (b_1-b_2)\|_{\cC^\nu_T \cC^{1}_{w_0}} + \|z_1 - z_2\|_{\infty,[s,t]}\big). 
\]
This gives
\[\|z_1-z_2\|_{\infty,[s,t]} \le 3\left(|z_1(s)-z_2(s)| +  \|\cT (b_1-b_2)\|_{\cC^\nu_T \cC^{1}_{w_0}}\right). 
\]
One can iterate the previous bound on small intervals to deduce the desired estimate.

Note that  the uniqueness of the solution follows directly by setting $b_1=b_2=b$.

Uniqueness and regularity of the semiflow for equation \eqref{eq:rde_with_drift} are then a direct consequence of Theorems \ref{theorem:main1} and \ref{theorem:driftlessRDE}.
\end{proof}

\begin{remark}

Theorem \ref{theorem:main2} requires that $b\in \cC^{\kappa}$ for some $\kappa>0$. This hypothesis is needed in view of Theorem \eqref{theorem:main1} which allows us to state an equivalent notion of solutions for RDE with drift as well as to make sense of Equation \eqref{eq:ODE_rough} as an actual ODE. 

Nevertheless, for any $b\in\cS'(\RR^d)$ such that $\cT b$ exists and $\cT b\in\cC^\nu_T\cC^1_{w_0}$ for some sublinear weight $w_0$ and $\nu >\frac12$, there exists a solution $(\theta_t)_{t\in[0,T]}\in\cC^\nu_T$ to the non linear Young differential equation
\begin{equation}\label{eq:NLYDE}
\theta_t = \theta_0 + \int_0^t (\cT b)_{\dd r}(\theta_r),\end{equation}
where the integral is constructed using the sewing lemma applied to 
\[A_{s,t} = \cT b_{s,t}(\theta_s).\]
One can consult \cite{catellierAveragingIrregularCurves2016,galeatiNonlinearYoungDifferential2021}
for more details.

Hence, one could use the following definition in order to extend the notion of solution to RDE with drift :
\begin{definition}\label{def:RDE_NLY}
Let $p\ge 2$, $T>0$ and $\sigma \in C^{\lfloor p \rfloor +2}(\RR^d;(\RR^d)^{\otimes 2})$. 
Let 
$\bw$ 
be a 
$\frac1p$-H\"older weakly 
geometric $p$-rough path. 
Let 
$b\in\cS'(\RR^d;\RR^d)$ 
and 
$w_0$ 
be a sublinear weight such that 
$\cT b \in \cC^\nu_T\cC^1_{w_0}$ 
for some $\nu>\frac12$. A path $(x_t)_{t\in[0,T]}$ is a solution to Equation \eqref{eq:rde_with_drift} if $(x_t)_{t\in[0,T]}=\big(\varphi_{0,t}(\theta_t)\big)_{t\in[0,T]}$, where $\varphi$ is the flow generated by the driftless RDE \eqref{eq:rde_basic}, and $(\theta_t)_{t\in[0,T]}$ is a solution to the non-linear Young differential equation \eqref{eq:NLYDE} with $\theta_0 = x_0$.
\end{definition}

An obvious remark is that whenever $b\in\cC^\kappa$ for some $\kappa>0$ and when $\sigma\in C^{2(\lfloor p \rfloor +1)}$, thank  to Lemma \ref{lemma:riemann} and Theorem \ref{theorem:main1}, this definition is equivalent to Definition \ref{def:rde_with_drift}. The huge difference when $b$ is not bounded, is the lack of a priori estimates for the solution $\theta$. In the setting, one could prove the following theorem :

\begin{theorem}\label{theorem:main2bis}
Let $p\ge 2$, $T>0$, $b\in \cS'(\RR^d;\RR^d)$ and $\sigma\in C^{\lfloor p \rfloor + 2}_b(\RR^d;(\RR^d)^{\otimes 2})$. Let $\bw$ be a $\frac{1}{p}$-H\"older  weakly geometric $p$-rough path. 

Let us suppose that there exists $1\ge\nu>\frac12$, $\kappa>2$ and a sublinear weight $w_0$ (see Definitions \ref{def:weight} and \ref{def:weighted_holder}) such that 
\[\|\cT b\|_{\cC^\nu_T \cC^{\kappa}_{w_0}} < +\infty.\]
Then there is a unique solution to Equation \eqref{eq:NLYDE}.

In that setting, there is a unique solution, in the sense of Definition \ref{def:RDE_NLY} to Equation \eqref{eq:rde_with_drift}.
\end{theorem}

\begin{proof}
The proof is a direct adaptation of the ideas of \cite{catellierAveragingIrregularCurves2016}.
\end{proof}

The only difference comparing to Theorem \ref{theorem:main2} is the requirement that $\cT b\in\cC^\nu\cC^\kappa_{w_0}$ with $\kappa>2$. In our context of application, when $\bw$ is a Gaussian rough with $p<4$  which enjoys some non-determinism properties (see Section \ref{section:Malliavin}), one cannot expect to have a better regularizing effect than $\frac{p}{2}-\eps$ for some small $\eps>0$ (see Section \ref{section:kolmo} and Theorem \ref{thm:MAIN000}). Namely, if $b\in\cC^{\alpha}:=B^{\alpha}_{\infty,\infty}$ for some $\alpha\in \RR$, one may  expect that $ \cT b \in \cC^\nu_T \cC^{\frac{p}{2} + \alpha - \eps}_{w_0}$ for some small $\eps>0$ and some sublinear weight $w_0$. When applying Theorem \ref{theorem:main2bis}, one should ask that $\frac{p}2+\alpha >2$, but since $p<4$, this gives necessarily $\alpha>0$, and we can apply  Theorem \ref{theorem:main2} to obtain a better result.

Nevertheless, Definition \ref{def:RDE_NLY} and Theorem \ref{theorem:main2bis} could be useful when $b\in L^\infty \cap B^{-\alpha}_{\infty,\infty}$ or when $\bw$ enjoys better regularizing properties. We leave those investigations for  future works.

Finally, in view of \cite{davieIndividualPathUniqueness2011a}, one could expect that in the setting of fractional Brownian motion, the only requirement should be $\cT b \in \cC^{\nu}\cC^{1}_{w_0}$ (this is the case when $\bw$ is the Stratonovitch Brownian rough path). This kind of result would require a use of Girsanov transform as in \cite{davieDifferentialEquationsDriven2007,catellierAveragingIrregularCurves2016,davieIndividualPathUniqueness2011a}. The bounds of Section \ref{section:Malliavin} are not good enough, neither the Kolmogorov estimates of Section \ref{section:kolmo}. Again, we leave this for future investigations, \blue{see also \cite{dareiotisPathbypathRegularisationMultiplicative2022}.}
\end{remark}

\section{Kolmogorov type theorem for averaged fields}\label{section:kolmo}

In the previous section, and especially in Theorem \ref{theorem:main2}, we  exhibit a general criterion in terms of the averaged field $\cT b$ such that Equation \eqref{eq:rde_with_drift} has a unique solution. Nevertheless, we observe that without using any additional property of the flow $\varphi$, even in the case when $b$ is Lipschitz continuous (and thus Equation \eqref{eq:rde_with_drift} has a unique solution thanks to standard arguments), this does not gives uniqueness of solutions. We will proceed by using fine stochastic properties of the flow, when $\bw$ is a random rough path.

Hence, in this section, we will focus ourselves on the action a the random flow $\varphi:[0,T]\times\RR^d\to\RR^d$ as an averaging operator where $T>0$ is fixed. Thanks to the flow decomposition technique when $\varphi$ is a $C^1$ flow of diffeomorphism, we will investigate the mixed space time regularity of the following averaged field :

\[(\cT^{(D\varphi)^{-1},\varphi}b)_t(x) = \int_0^t D\varphi_u(x)^{-1}b(\varphi_u(x))\dd u.\]
In order to be as general as possible, we want to take $b\in B^{\alpha}_{p,r}$ as generic as possible. In order to do so, we rely on estimates of the type
\begin{equation*}
\|\cT^{(D\varphi)^{-1},\varphi} f\|_{L^q(\Omega;\cC^{\gamma}([0,T];B^{\alpha}_{p,r}))} \leq C\|f\|_{B^{\alpha'}_{\ell,m}}
\end{equation*}
for some constant $C>0$ and parameters $q\geq 2$, $\gamma>0$, $\alpha,\alpha'\in \RR$, $p,r,\ell,m\in[1,+\infty)$ and any $f\in\cS$. In the case $\ell = m = +\infty$, weighted estimates are needed for the left hand side, whereas for the right hand side we must use the fact that for any $f\in B^{\alpha'}_{\infty,\infty}$ and any $\eps>0$, $f$ is the limit of Schwartz functions in $B^{\alpha'-\eps}_{\infty,\infty}$ (see Lemma \ref{lem:A3} and Theorem \ref{theorem:kolmo} for details). Then, since $\cT^{(D\varphi)^{-1},\varphi}$ is a linear operator (on $\cS$), we consider its closure from $B^{\alpha'}_{\ell,m}$ to $L^q(\Omega;\cC^{\gamma}([0,T];B^{\alpha}_{p,r}))$ (which is still denoted $\cT^{(D\varphi)^{-1},\varphi}$).

In the following, we will consider a general form of the averaging operator: for any $\phi : [0,T]\times\RR^d \to (\RR^d)^{\otimes 2}$ and any $\varphi:[0,T]\times\RR^d\to\RR^d$, we denote, for any $0\leq s<t\leq T$, $x\in\RR^d$ and $f\in\cS$,
\begin{equation}\label{eq:avgfield}
\cT^{\phi,\varphi} f_t(x) = \int_0^t \phi(u,x) f(\varphi(u,x)) du.
\end{equation}

\begin{remark}We obtain a first, rather straightforward, bound, for any $\phi \in L^{q}(\Omega; L^{\infty}([0,T]\times(\RR^d)^{\otimes 2}))$ and $\varphi \in L^{q}(\Omega; L^{\infty}_{loc}([0,T]\times\RR^d))$,
\begin{equation*}
\|\cT^{\phi,\varphi}f\|_{L^q(\Omega;L^{\infty}([0,T]\times\RR^d)} \leq T \|\phi\|_{L^q(\Omega;L^{\infty}([0,T]\times\RR^d)} \|f\|_{L^{\infty}(\RR^d)}.
\end{equation*}
\end{remark}

Before stating our main result, we need the following assumptions.
\begin{assumption}\label{asm:phipsi}
Let $p\in[2,+\infty]$, $q \ge 2$ and let $(\phi,\varphi)$ be a couple of random functions such that $\phi:[0,T]\times\RR^d\to(\RR^d)^{\otimes 2}$ and $\varphi:[0,T]\times\RR^d\to\RR^d$. 
Define \[\tilde{q} =\begin{cases} 2q & \text{if }\; p<+\infty \\ q & \text{if }\; p=+\infty\end{cases}.\]

We assume furthermore that:
\begin{enumerate}
\item
for all $x\in \RR^d$, $(\phi(t,x))_{t\in[0,T]}$ and $(\varphi(t,x))_{t\in[0,T]}$ are adapted,
\item we have
\[\phi\in L^{\tilde{q}}\Big(\Omega; L^\infty\big([0,T]; L^{\infty}(\RR^d;(\RR^d)^{\otimes 2})\big)\Big) \quad \text{and} \quad \varphi \in L^{q} \Big(\Omega; L^\infty\big([0,T];L^{\infty}_{loc}(\RR^d;\RR^d)\big)\Big),\]
\item if $p<+\infty$, for almost all $t\in[0,T]$, $\varphi^{-1}(t,\cdot)$ exists almost surely and $\det(J_{\varphi^{-1}}) \in L^{\tilde{q}}(\Omega;L^\infty([0,T];L^\infty(\RR^d;\RR)))$.
\end{enumerate}
Moreover, for any $s\in[0,T]$, we assume that there exists $G_{p,s}$, a positive $\cF_s$-measurable random variable such that
\[\sup_{s\in[0,T]} \EE[G_{p,s}^{\tilde{q}}] < + \infty,\] and that there exists $H\in(0,1)$ such that, for any $f\in \cS$, for any $x\in\RR^d$, for any multi-index $\beta \in \NN^d$ and any $0 \le s \le r \le T$, we have
\begin{equation}\label{eq:condition_kolmogorov}
\bigg|
\EE\Big[\phi(r,y)\partial^\beta f\big(\varphi(r,x)\big)\Big| \cF_s \Big]
 \bigg| \lesssim 
 \begin{cases}
 |r-s|^{-|\beta| H}  G_{p,s} \EE\Big[|f(\varphi(r,x))|^p\Big| \cF_s \Big]^{\frac1p} & \text{if } 2\le p<+\infty, \\
 |r-s|^{-|\beta| H}  G_{\infty,s} \|f\|_\infty & \text{if } p=+\infty.
 \end{cases}
\end{equation}
\end{assumption}

We can now proceed to state our main result.

\begin{theorem}\label{theorem:kolmo}
Let $T>0$,  $p,q,(\phi,\varphi)$ that satisfy Assumption \ref{asm:phipsi}. Then, for all $\varepsilon\in(0,1]$, $0<\varepsilon'<\eps<\eps''$, $f\in B^{-\frac{1-\eps}{2H}}_{p,r}$, $\eta > d/q$, $\nu = \frac{1+\eps'}{2}-\frac{1}{q}$, the following estimate holds
\begin{equation*}
\|\cT^{\phi,\varphi} f\|_{L^q(\Omega; \cC^{\nu}([0,T];E_p))} \lesssim \|f\|_{B^{-\frac{1-\eps''}{2H}}_{p,r}},\end{equation*}
with $E_p = L^p(\RR^d)$ if $p<+\infty$ and $E_\infty = L^{\infty}_w(\RR^d)$ if $p = +\infty$, for some weight $w$ with a $\eta$-polynomial growth.
When $r=1$ and $p<+\infty$, one can take $\eps''=\eps$.
\end{theorem}

We begin by proving a Lemma which allow us  to transform the condition \eqref{eq:condition_kolmogorov} into a useful one concerning Paley-Littlewood blocks. To do so, we use a trick which can be found in \cite{bahouriFourierAnalysisNonlinear2011} Lemma 2.1.

\begin{lemma}\label{ref:lemma_interpolation}
Let $f\in \cS$  and $p,q,(\phi,\varphi)$ that satisfy Assumption \ref{asm:phipsi}. Then, for any $p\in [2,+\infty]$, for all $j\ge 0$ and all $\eta\in[0,1]$, we have 
\begin{multline}\label{eq:condition_kolmogorov_paley-littlewood}
 \bigg\|
 \EE\Big[
 \phi(r,\cdot) \Delta_j f\big(\varphi(r,\cdot)\big)\Big| \cF_s 
 \Big] 
 \bigg\|_{L^p(\RR^d;\RR)} 
 \lesssim |r-s|^{-(1-\eta)} 2^{-j \frac{1-\eta}{H} } \|\Delta_j f\|_{L^{p}(\RR^d;\RR)} 
 \\ 
 \times
 G_{p,s}\left(1+ \EE\left[\|\det(J_{\varphi^{-1}})\|_{L^\infty([0,T]\times\RR^d)} | \cF_s\right]^{\frac1p}\ind_{\{p<+\infty\}}\right).
 \end{multline}
\end{lemma}

\begin{proof}
We remark that, for all $\xi=(\xi_1,\ldots,\xi_d)\in\RR^d$ and all $k\in \NN$,
\begin{equation}\label{eq:relAalpha}
|\xi|^{2k} =\left(\sum_{j = 1}^d \xi_j^2 \right)^k =  \sum_{1_\le j_1,\ldots,j_k \le d } \xi_{j_1}^{2}\ldots \xi_{j_k}^{2} = \sum_{\substack{\beta\in\NN^d\\ |\beta| = k}} A_\beta (-i \xi)^\beta (i\xi)^\beta,
\end{equation}
where $A_\beta$ are non negative constants and for $\xi = (\xi_1,\ldots,\xi_d)\in\RR^d$ and a multi-index $\beta=(\beta,\ldots,\beta_d) \in \NN^d$, $\xi^\beta = \xi_1^{\beta_1}\ldots \xi_d^{\beta_d}$. Let $\varrho\in\cC^{\infty}_0(\RR^d)$ such that $\mathrm{supp}(\varrho)\subset\mathscr{A}$ and $\varrho_{|\mathscr{A}} \equiv 1$ where $\mathscr{A} = \{ \xi\in\RR^d; 3/4\leq |\xi|\leq 8/3\}$ is an annulus in $\RR^d$. For any $\beta\in\NN^d$, such that $|\beta| = k$, and $j\geq 0$, we define the function $\mathsf{m}_{k,\beta,j}$ given by
\[\mathsf{m}_{k,\beta,j}=  \mathfrak{F}^{-1}(A_{\beta} (-i\xi)^{\beta} |\xi|^{-2k} \varrho(2^{-j}\xi)),\]
where $\mathfrak{F}^{-1}$ denotes the inverse Fourier transform.
It follows by \eqref{eq:relAalpha} that, for any $k\in\NN$,
\begin{equation}\label{eq:trick_differentiability}
\Delta_j f = \sum_{|\beta| = k}  \mathsf{m}_{k,\beta,j} \star \partial^\beta\Delta_j  f,
\end{equation}
where, by Lemma \eqref{lem:FourierMult}, we have
\begin{equation*}
\| \mathsf{m}_{k,\beta,j}\|_{L^1(\RR^d)} \lesssim 2^{-jk}.
\end{equation*}
We have, for $p<+\infty$
\begin{multline*}
\left\|\EE\left[\phi(r,\cdot)\Delta_j f (\varphi(r,\cdot))\big| \cF_s \right]\right\|_{L^p(\RR^d)}
	\le
\sum_{|\alpha| = k} \left\|\EE\left[\phi(r,\cdot)(\mathsf{m}_{k,\alpha,j}\star\partial^\alpha \Delta_j f) (\varphi(r,\cdot))\big| \cF_s \right]\right\|_{L^p(\RR^d)}
	\\ =
\sum_{|\alpha| = k} \left(\int_{\RR ^d}\left| \int_{\RR^d} \mathsf{m}_{k,\alpha,j}(y) \EE\left[\phi(r,x)(\partial^\alpha \Delta_j f) (\varphi(r,x)-y) \big| \cF_s \right] \dd y\right|^{p} \dd x \right)^{\frac1p},
\end{multline*}
and, for $p = +\infty$,
\begin{multline*}
\left\|\EE\left[\phi(r,\cdot)\Delta_j f (\varphi(r,\cdot))\big| \cF_s \right]\right\|_{L^p(\RR^d)}\le
\sum_{|\beta|  = k} \sup_{x\in\RR^d}\bigg| \int_{\RR^d} \mathsf{m}_{k,\beta,j}(y) \\ \times \EE\Big[\phi(r,x)(\partial^\beta \Delta_j f) (\varphi(r,x)-y) \big| \cF_s \Big] \dd y\bigg|.
\end{multline*}
We want to apply Young's inequality for kernel operators (Theorem \ref{thm:ineq_young}) with  
\[K(x,y) = \EE\left[\phi(r,x)(\partial^\beta \Delta_j f) (\varphi(r,x)-y) \big| \cF_s \right].\]
Note that, thanks to Equation \eqref{eq:condition_kolmogorov}, we know that 
\[|K(x,y)| \lesssim  \begin{cases}
 |r-s|^{-|\beta| H} G_{p,s} \EE[|\Delta_jf(\varphi(r,x)-y)|^p | \cF_s ]^{\frac{1}{p}} & \text{if } 2\le p<+\infty, \\
 |r-s|^{-|\beta| H}  G_{\infty,s} \|\Delta_j f\|_\infty & \text{if } p=+\infty.
 \end{cases}\]
Hence, for all $p\in[2,+\infty]$,
\[\sup_{x\in\RR^d} \|K(x,\cdot)\|_{L^p(\RR^d)}  \lesssim |r-s|^{-|\beta|H} G_{p,s} \|\Delta_jf\|_{L^p(\RR^d)}.\]
Furthermore, we note that, if $p<+\infty$,
\begin{align*}
\int_{\RR^d} \EE\left[|\Delta_jf(\varphi(r,x)-y)|^p\Big| \cF_s\right] \dd x  
= &
\EE\left[\int_{\RR^d} |\Delta_j f(z)|^p \EE \left[|\det(J_{\varphi^{-1}}(r,z+y))| \big| \cF_s  \right] \dd z \right] \\
\lesssim & \|\Delta_j f\|_{L^p(\RR^d)}^p \EE[\|\det(J_{\varphi^{-1}})\|_{L^\infty([0,T]\times\RR^d)}|\cF_s]
\end{align*}
and we obtain
\begin{align*}
&\sup_{y\in\RR^d} \|K(\cdot,y)\|_{L^p(\RR^d)}
\\ &\hspace{2em}\lesssim |r-s|^{-|\beta|H} G_{p,s} \|\Delta_j f\|_{L^p(\RR^d)} \times\begin{cases}
 \EE[\|\det(J_{\varphi^{-1}})\|_{L^\infty([0,T]\times\RR^d)}|\cF_s]^{1/p} & \text{if } 1<p<+\infty, \\
 1 & \text{if } p=+\infty.
 \end{cases} 
\end{align*}
Applying Young's inequality (Theorem \ref{thm:ineq_young}) and denoting 
\begin{equation*}
    P_{p,s} = G_{p,s}\left( 1+\EE[\|\det(J_{\varphi^{-1}})\|_{L^\infty([0,T]\times\RR^d)}|\cF_s]^{\frac1p}\ind_{\{p<+\infty\}}\right),
\end{equation*} we deduce that
\begin{align*}
\left\|\EE\left[\phi(r,\cdot)(\mathsf{m}_{k,\beta,j}\star\partial^\beta \Delta_j f) (\varphi(r,\cdot))\big| \cF_s \right]\right\|_{L^p(\RR^d)}\lesssim 2^{-jk}  |r-s|^{-|\beta|H}  \|\Delta_j f\|_{L^p(\RR^d)}P_{p,s}.
\end{align*}
It follows that, for any $k\in\NN$,
\begin{align*}
\left\|\EE\left[\phi(r,\cdot)\Delta_j f (\varphi(r,\cdot))\big| \cF_s \right]\right\|_{L^p(\RR^d)} &\lesssim  2^{-jk}\sum_{|\beta| = k}|t-s|^{-|\beta| H} \|\Delta_j f\|_{L^p(\RR^d)}P_{p,s}
\\ & \lesssim 
2^{-jk}|t-s|^{-k H} \|\Delta_j f\|_{L^p(\RR^d)}P_{p,s}.
\end{align*}
By interpolating the previous inequalities, we obtain the estimate, for any $\nu\in\RR^+$,
\begin{align*}
\left\|\EE\left[\phi(r,\cdot)\Delta_j f (\varphi(r,\cdot))\big| \cF_s \right]\right\|_{L^p(\RR^d)}\lesssim  2^{-j\nu}|t-s|^{-\nu H} \|\Delta_j f\|_{L^p(\RR^d)}P_{s,p},
\end{align*}
which gives the desired result by taking $\nu = \frac{1-\eta}H$.\end{proof}

\begin{proof}[Proof of Theorem \ref{theorem:kolmo}]

We split the proof in two parts.

\paragraph{The $p<+\infty$ case.} Let $0\leq s<t\leq T$ be fixed. We denote $\mathfrak{j} = \min\{j\in\NN;\;2^{-\frac jH} \leq (t-s)\} $. First, remark that 

\begin{align}
\| (\cT^{\phi,\varphi} \Delta_j f)_{s,t} \|_{L^p(\RR^d)}& \le \int_s^t \|\phi(r,\cdot)\Delta_j f(r,\varphi(r,\cdot)) \|_{L^p(\RR^d)} \dd r \nonumber
\\ &\le (t-s) \|\phi\|_{L^\infty([0,T]\times\RR^d)}\|\det(J_{\varphi^{-1}})\|_{L^\infty([0,T]\times\RR^d)} \|\Delta_j f\|_{L^p(\RR^d)}.\label{eq:basic_bound}
\end{align}
For any $-1\leq j < \mathfrak{j}$ and $\varepsilon\in[0,1]$, this yields, since $(t-s)\leq 2^{-\frac jH}$,
\begin{equation}
\| (\cT^{\phi,\varphi} \Delta_j f)_{s,t} \|_{L^p(\RR^d)} \le (t-s)^{\frac{1+\varepsilon}2} 2^{-\frac{1-\varepsilon}{2H}j} \|\phi\|_{L^\infty([0,T]\times\RR^d)}\|\det(J_{\varphi^{-1}})\|_{L^\infty([0,T]\times\RR^d)} \|\Delta_j f\|_{L^p(\RR^d)}.\label{eq:basic_bound_bis}
\end{equation}
Furthermore, thanks to Lemma \ref{ref:lemma_interpolation}, we also deduce that, for any $\eta\in[0,1]$,
\begin{align}
\left\|\EE[(\cT^{\phi,\varphi}\Delta_j f)_{s,t}| \cF_s ]\right\|_{L^p(\RR^d)} \hspace{5em}&\nonumber
\\\lesssim (t-s)^\eta \|\Delta_j f\|_{L^p(\RR^d)} 2^{-\frac{1-\eta}{H}j}   G_{p,s}&\left(1+ \EE\left[\|\det(J_{\varphi^{-1}})\|_{{L^\infty([0,T]\times\RR^d)}} | \cF_s\right]^{\frac1p}\right).\label{eq:improved_bound}
\end{align}
We now assume that $j\geq \mathfrak{j}$. Let $n_j\ge 1$ and let us define, for any $k\in\{0,\ldots,n_j\}$, the following quantities :
\[t_k^j : = \frac{k}{n_j}(t-s) + s\quad\mathrm{and}\quad M^j_k : = \EE[(\cT^{\phi,\varphi}\Delta_j f)_{s,t} |\cF_{t_k^j}].\]
Hence, $(M^j_k)_{0\leq k \leq n_j}$ is a martingale in $L^p(\RR^d)$ with respect to the filtration $(\cF_{t^j_k})_{0\leq k\leq n_j}$. Furthermore, it follows from \eqref{eq:improved_bound} that
\[\|M^j_0\|_{L^p(\RR^d)} \lesssim (t-s)^\eta \|\Delta_j f\|_{L^p(\RR^d)} 2^{-\frac{1-\eta}{H}j}   G_{p,s}\left(1+ \EE\left[\|\det(J_{\varphi^{-1}})\|_{L^\infty([0,T]\times\RR^d)} | \cF_s\right]^{\frac1p}\right),\]
and, in particular, by setting $\eta = (1+\varepsilon)/2$, where $\varepsilon\in[0,1]$, and for $q\ge 1$, we have
\begin{equation}\label{eq:estmM0}
\EE[\|M^j_0\|_{L^p(\RR^d)}^q]^{\frac1q} \lesssim (t-s)^{\frac{1+\eps}{2}} \|\Delta_j f\|_{L^p(\RR^d)} 2^{-\frac{1-\eps}{2H}j}.
\end{equation}

Since we are working in $L^p(\RR^d)$, $p\in[2,+\infty)$, which is a UMD space of type $2$ (see for example \cite{hytonenAnalysisBanachSpaces2016, hytonenAnalysisBanachSpaces2018} and Appendix \ref{appendix:UMD}), we know that the Burkolder-Davis-Gundy inequality holds, and we have,
\[\EE[\|M^j_{n_j} - M^j_{0}\|_{L^p(\RR^d)}^q] \lesssim \EE\left[\left(\sum_{k=0}^{n_j-1}\|M^j_{k+1}-M^j_k\|^2_{L^p(\RR^d)} \right)^{\frac{q}{2}}\right].\]
Furthermore, we obtain that, for all $0\le k\le n_{j}-1$, 
\[M^j_{k+1}-M^j_k = (\cT^{\phi,\varphi}\Delta_j f)_{t^j_k,t^j_{k+1}} + \EE\left[(\cT^{\phi,\varphi}\Delta_j f)_{t^j_{k+1},t}\Big| \cF_{t^j_{k+1}}\right] - \EE\left[(\cT^{\phi,\varphi}\Delta_j f)_{t^j_{k},t}\Big| \cF_{t^j_{k}}\right].\]
Using again \eqref{eq:basic_bound} to estimate the first term on the right-hand-side and  \eqref{eq:improved_bound} for the second and third terms, we deduce, for any $\varepsilon\in[0,1]$,
\begin{equation*}
\|M^{j}_{k+1}-M^{j}_{k}\|_{L^p(\RR^d)} 
	\lesssim 
\|\Delta_j f\|_{L^p(\RR^d)} (t-s)^\varepsilon \left( \frac{(t-s)^{1-\varepsilon}}{n_j}  + 2^{-\frac{1-\varepsilon}{H}j}\right) \kappa^j_{k},
\end{equation*}
where
\begin{align*}
\kappa^j_{k} &: =\|\phi\|_{L^\infty([0,T]\times\RR^d)}\|\det(J_{\varphi^{-1}})\|_{L^\infty([0,T]\times\RR^d)} 
  \\ &\hspace{1em}+ G_{p,t^j_k}\left(1+ \EE\left[\|\det(J_{\varphi^{-1}})\|_{L^\infty([0,T]\times\RR^d)} | \cF_{t^j_k}\right]^{\frac1p}\right)
 \\ &\hspace{1em}+ G_{p,t^j_{k+1}}\left(1+ \EE\left[\|\det(J_{\varphi^{-1}})\|_{L^\infty([0,T]\times\RR^d)} | \cF_{t^j_{k+1}}\right]^{\frac1p}\right). 
\end{align*}
Note that thanks to the hypothesis and Cauchy-Schwarz inequality, 
\[\sup_{j\ge 0} \sup_{k\in\{0,\ldots,n_j\}} \EE[(\kappa^j_k)^q] <+\infty.\]
Hence, we have,
\begin{align*}
\EE[\|M^j_{n_j} - M^j_{0}\|_{L^p(\RR^d)}^q] 
	&\lesssim 
	 	\EE\left[\left(\sum_{k=0}^{n_j-1}\|M^j_{k+1}-M^j_k\|^2_{L^p(\RR^d)} \right)^{\frac{q}{2}}\right] \\
	&\lesssim
	(t-s)^{\varepsilon q} \|\Delta_j f\|^q_{L^{p}(\RR^d) }\left( \frac{(t-s)^{1-\varepsilon}}{n_j}  + 2^{-\frac{1-\varepsilon}{H}j}\right)^q \EE\left[\left( \sum_{k=0}^{n_j-1} (\kappa_k^j)^2 \right)^{\frac{q}{2}} \right]
	\\ &\lesssim
	(t-s)^{\varepsilon q} \|\Delta_j f\|^q_{L^{p}(\RR^d)}\left( \frac{(t-s)^{1-\varepsilon}}{n_j}  + 2^{-\frac{1-\varepsilon}{H}j}\right)^q  \left(\sum_{k=0}^{n_j-1} \EE[(\kappa^j_k)^q]^{\frac{2}{q}}\right)^{\frac{q}{2}} 
	\\ &\lesssim
(t-s)^{\varepsilon q} \|\Delta_j f\|^q_{L^{p}(\RR^d)}\left( \frac{(t-s)^{1-\varepsilon}}{\sqrt{n_j}}  + \sqrt{n_j}2^{-\frac{1-\varepsilon}{H}j}\right)^q.
\end{align*}
We now optimize the estimate on $n_j\geq 1$. We choose
\[ n_j = \frac{(t-s)^{1-\varepsilon}}{2^{-\frac{1-\varepsilon}{H}j}} + \iota,\]
for a certain $\iota\in[0,1]$. This yields, since $2^{-\frac jH}\leq (t-s)$,
\begin{equation*}
\sqrt{n_j}2^{-\frac{1-\varepsilon}{H}j}  = (t-s)^{\frac{1-\varepsilon}2}2^{-\frac{1-\varepsilon}{2H}j}\sqrt{1 + \iota \frac{2^{-\frac{1-\varepsilon}{H}j}}{(t-s)^{1-\varepsilon}}} \leq \sqrt{2} (t-s)^{\frac{1-\varepsilon}2}2^{-\frac{1-\varepsilon}{2H}j},
\end{equation*}
and, also,
\begin{equation*}
 \frac{(t-s)^{1-\varepsilon}}{\sqrt{n_j}} \leq (t-s)^{\frac{1-\varepsilon}2}2^{-\frac{1-\varepsilon}{2H}j}.
\end{equation*}
Thus, we deduce that
\[\EE\left[\|M^j_{n_j} - M^j_{0}\|_{L^p(\RR^d)}^q\right]^{\frac1q} \lesssim (t-s)^{\frac{1+\eps}{2}} \| \Delta_j f \|_{L^p(\RR^d)} 2^{-\frac{(1-\eps)}{2H} j}.\]
Using the fact that $(\cT^{\phi,\varphi}\Delta_j f)_{s,t} = M^j_{n_j}$ as well as the estimate \eqref{eq:estmM0}, we obtain for all $j\ge \mathfrak{j}$,
\begin{equation}\label{eq:estmThm31}
\EE\left[\left\|(\cT^{\phi,\varphi}\Delta_j f)_{s,t}\right\|_{L^p(\RR^d)}^q\right]^{\frac1q} \lesssim (t-s)^{\frac{1+\eps}{2}} \|\Delta_j f\|_{L^p(\RR^d)} 2^{-j\frac{1-\eps}{2H}}.
\end{equation}
It follows from \eqref{eq:basic_bound_bis} that the previous estimate holds for all $\varepsilon\in[0,1]$ and any $j\geq -1$. Thus, we deduce that for $\varepsilon\in[0,1]$,
\begin{align*}
\EE\left[\left\|(\cT^{\phi,\varphi}f)_{s,t}\right\|_{L^p(\RR^d)}^q\right]^{\frac1q} &\leq \sum_{j\geq -1} \EE\left[\left\|(\cT^{\phi,\varphi}\Delta_j f)_{s,t}\right\|_{L^p(\RR^d)}^q\right]^{\frac1q}
\\ &\lesssim
(t-s)^{\frac{1+\eps}{2}} \sum_{j=-1}^{+\infty} 2^{-j\frac{1-\eps}{2H}} \|\Delta_j f\|_{L^p(\RR^d)}  = 
(t-s)^{\frac{1+\eps}{2}}  \|f\|_{B^{-\frac{1-\eps}{2H}}_{p,1}}.
\end{align*}
It follows from Theorem \ref{thm:GRRstd} that, for all $0<\eps'<\eps$ and $f\in B^{-\frac{1-\eps}{2H}}_{p,1}$, we have
\begin{align*}
\|\cT^{\phi,\varphi}f\|_{L^q(\Omega;C^{\nu}([0,T];L^p(\RR^d)))} &\lesssim \EE\left[ \int_{[0,T]^2} \frac{\|(\cT^{\phi,\varphi}f)_{s,t}\|_{L^p(\RR^d)}^q}{|t-s|^{q\frac{1+\eps'}2+1}} \dd s \dd t \right]^{1/q}
\\ &\lesssim \|f\|_{B^{-\frac{1-\eps}{2H}}_{p,1}} \int_{[0,T]^2}|t-s|^{q\frac{\varepsilon-\varepsilon'}2-1}\dd s \dd t \lesssim \|f\|_{B^{-\frac{1-\eps}{2H}}_{p,1}},
\end{align*}
with $\nu = \frac{1+\eps'}{2}-\frac{1}{q}>0$, which is exactly the theorem in the case $p<+\infty$ when we remind that  whenever $f\in B^{\alpha}_{p,p}$, one has for $\eps''>0$, thanks to H\"older inequality,
\begin{multline*}
\|f\|_{B^{\alpha-\eps''}_{p,1}} = \sum_{j\ge 1} 2^{j(\alpha-\eps'')j} \|\Delta_j f\|_{L^p(\RR^d)} \le  \left(\sum_{j\ge -1}2^{-\eps'' \frac{r}{r-1}j}\right)^{1-\frac1r}\left(\sum_{j\ge -1 } 2^{\alpha r j} \|\Delta_j f\|^r_{L^p(\RR^d)}\right)^{\frac1r}  \\ \lesssim_{\eps''} \|f\|_{B^\alpha_{p,r}}.
\end{multline*}
\paragraph{The $p=+\infty$ case.}
The proof is pretty much the same as the proof in the previous case. We provide the main arguments for completeness. Let $0\leq s<t\leq T$ be fixed. We denote $\mathfrak{j} = \min\{j\in\NN;\;2^{-\frac jH} \leq (t-s)\} $. First, we can see that, for any $-1\leq j < \mathfrak{j}$ and $\varepsilon\in[0,1]$, this yields, since $(t-s)\leq 2^{-\frac jH}$, for any $x\in\RR^d$,
\begin{equation}
\left| (\cT^{\phi,\varphi} \Delta_j f)_{s,t}(x) \right| \le (t-s)^{\frac{1+\varepsilon}2} 2^{-\frac{1-\varepsilon}{2H}j} \|\phi\|_{L^\infty([0,T]\times\RR^d)} \|\Delta_j f\|_{L^{\infty}(\RR^d)}.\label{eq:basic_bound_inf}
\end{equation}
Moreover, it follows from Lemma \ref{ref:lemma_interpolation} that, for any $\eta\in[0,1]$,
\begin{align}
\left|\EE[(\cT^{\phi,\varphi}\Delta_j f)_{s,t}(x)| \cF_s ]\right|\lesssim (t-s)^\eta \|\Delta_j f\|_{L^{\infty}(\RR^d)} 2^{-\frac{1-\eta}{H}j}   G_{\infty,s}.\label{eq:improved_bound_inf}
\end{align}
We now assume that $j\geq \mathfrak{j}$. Let $x\in\RR^d$, $n_j\ge 1$ and let us define, for any $k\in\{0,\ldots,n_j\}$, the following quantities :
\[t_k^j : = \frac{k}{n_j}(t-s) + s\quad\mathrm{and}\quad M^j_k(x) : = \EE[(\cT^{\phi,\varphi}\Delta_j f)_{s,t}(x) |\cF_{t_k^j}].\]
We can see that, $(M^j_k(x))_{0\leq k \leq n_j}$ is a martingale with respect to the filtration $(\cF_{t^j_k})_{0\leq k\leq n_j}$. Thanks to \eqref{eq:improved_bound_inf}, we obtain
\[|M^j_0(x)| \lesssim (t-s)^\eta \|\Delta_j f\|_{L^{\infty}(\RR^d)} 2^{-\frac{1-\eta}{H}j}   G_{\infty,s},\]
which yields, by setting $\eta = (1+\varepsilon)/2$, where $\varepsilon\in[0,1]$,
\begin{equation}\label{eq:estmM0_inf}
\EE[|M^j_0(x)|^q] \lesssim (t-s)^{\frac{1+\eps}{2}} \|\Delta_j f\|_{L^{\infty}(\RR^d)}^q 2^{-\frac{1-\eps}{2H}j}.
\end{equation}
It follows from \eqref{eq:basic_bound_inf} and \eqref{eq:improved_bound_inf}  that, for any $\varepsilon\in[0,1]$ and $1\leq k\leq n_j$,
\begin{equation*}
|M^{j}_{k+1}(x)-M^{j}_{k}(x)| 
	\lesssim 
\|\Delta_j f\|_{L^{\infty}(\RR^d)} (t-s)^\varepsilon \left( \frac{(t-s)^{1-\varepsilon}}{n_j}  + 2^{-\frac{1-\varepsilon}{H}j}\right) \kappa^j_{k},
\end{equation*}
where
\begin{equation*}
\kappa^j_{k} : =\|\phi\|_{L^\infty([0,T]\times\RR^d)}+ G_{\infty,t^j_k}+ G_{\infty,t^j_{k+1}}.
\end{equation*}
The assumptions gives the bound 
\[\sup_{j\ge 0} \sup_{k\in\{0,\ldots,n_j\}} \EE[(\kappa^j_k)^q] <+\infty.\]
From here, by following the same arguments as in the $p<+\infty$ case, we deduce, thanks to the BDG inequality and some optimization on $n_j$, that for any $\varepsilon\in[0,1]$ and for all $j\ge -1$,
\begin{equation}\label{eq:estmThm31_inf}
\EE\left[|(\cT^{\phi,\varphi}\Delta_j f)_{s,t}(x)|^q\right]^{\frac1q} \lesssim (t-s)^{\frac{1+\eps}{2}} \|\Delta_j f\|_{L^{\infty}(\RR^d)} 2^{-j\frac{1-\eps}{2H}},
\end{equation}
and we obtain the estimate, for any $x\in\RR^d$,
\[\EE[|(\cT^{\phi,\varphi} f)_{s,t}(x)|^q]^{\frac1q} \lesssim |t-s|^{\frac{1+\eps}{2}} \|f\|_{B^{-\frac{1-\eps}{2H}}_{\infty,1}}.\] Finally, Theorem \ref{thm:Kolmogo_std} gives the desired result.
\end{proof}
\begin{remark}
The dependence on $\phi$, $\varphi$ and $(G_{p,s})_s$ in the bound of Theorem \ref{theorem:kolmo} is not given explicitly but can be tracked thanks to the variables $\kappa_k^j$ used in the proof.
\end{remark}

We can extend the previous result with the help of the following Lemma.

\begin{lemma}\label{lemma:derivation}
Let $\beta\in\NN^d$, $\phi,\varphi \in L^q(\Omega; L^{\infty}([0,T];C^{|\beta|}(\RR^d)))$ and $f\in\cS$. Then, we have, for any $(t,x)\in[0,T]\times\RR^d$,
\[\partial^\beta (\cT^{\phi,\varphi})_t f(x) 
= 
\sum_{\gamma\leq\beta} \sum_{\substack{1\leq |\mu|\leq |\gamma|\\ \upsilon\in\mathcal{N}_{\gamma,\mu}}} C_{\beta,\gamma,\mu,\upsilon} \left(\cT^{\Lambda_{\gamma,\upsilon}(\varphi)\partial^{\beta-\gamma}\phi, \varphi} \big(\partial^{\mu} f\big)\right)_t(x),\]
where $\Lambda_{\gamma,\upsilon}(\varphi) : = \prod_{\substack{1\leq |\delta|\leq |\gamma|\\1\leq j \leq d}} (\partial^{\delta} \varphi_{j}(r,x))^{\upsilon_{\delta_j}}$ and $C_{\beta,\gamma,\mu,\upsilon}>0$ and
\begin{equation*}
\mathcal{N}_{\gamma,\mu} : = \left\{ \upsilon\in\NN^d:\;\sum_{1\leq |\delta|\leq |\gamma|} \upsilon_{\delta_j} = \mu_j,\;\textrm{for any}\; 1\leq j\leq d,\;\mathrm{and}\;\sum_{\substack{1\leq |\delta|\leq |\gamma|\\1\leq j \leq d}} \delta \upsilon_{\delta_j} = \gamma \right\},
\end{equation*}
\end{lemma}

\begin{proof}
The proof of the lemma is a straightforward consequence of Leibniz's rule of differentiation, which yields
\begin{equation*}
\partial^{\beta} \phi(r,x) b(\varphi(r,x)) = \sum_{\gamma\leq\beta} C^1_{\beta,\gamma} \partial^{\beta-\gamma}\phi(r,x) \partial^{\gamma} (b(\varphi(r,x))),
\end{equation*}
for some constants $\{C^1_{\beta,\gamma}\}_{\gamma\leq \beta}\subset\RR^+$ and Fa\`a di Bruno's formula, which gives
\begin{equation*}
 \partial^{\gamma} (b(\varphi(r,x))) = \sum_{\substack{1\leq|\mu|\leq |\gamma|\\ \upsilon\in \mathcal{N}_{\gamma,\mu}}} C^2_{\mu,\upsilon} \partial^{\mu} b (\varphi(r,x)) \prod_{\substack{1\leq |\delta|\leq |\gamma|\\1\leq j \leq d}} (\partial^{\delta} \varphi_{j}(r,x))^{\upsilon_{\delta_j}},
\end{equation*}
for some constants $\{C^2_{\upsilon,\mu}\}_{\substack{1\leq|\mu|\leq |\gamma|\\ \upsilon\in \mathcal{N}_{\gamma,\mu}}}\subset\RR^+$.
\end{proof}

We can now state a by-product of Theorem \ref{theorem:kolmo}.

\begin{corollary}\label{corollary:averaged_holder}
Let $r\in[1,+\infty]$, $n\geq 1$ and $p\in[2,+\infty]$, $q\ge2$ and $(\phi,\varphi)$ such that, for any $\beta,\gamma,\mu,\upsilon \in\NN^d$ that verify $|\beta|\leq n$, $\gamma\leq \beta$, $1\leq |\mu|\leq |\gamma|$ and $\upsilon\in\cN_{\gamma,\mu}$, $p,q$, and $(\Lambda_{\gamma,\upsilon}(\varphi)\partial^{\beta-\gamma}\phi,\varphi)$ satisfies Assumption \ref{asm:phipsi}. 
Then, for all $\varepsilon\in(0,1]$, $\varepsilon'<\varepsilon<\eps''$, $\kappa\in(0,n]$, $f\in B^{\kappa-\frac{1-\eps}{2H}}_{p,r}$, $\eta > d/q$ and  such that, $\nu = \frac{1+\eps'}{2}-\frac{1}{q}$, the following estimate holds
\begin{equation*}
\|\cT^{\phi,\varphi} f\|_{L^q(\Omega; \cC^{\nu}([0,T];E_{p,\kappa}))} \lesssim \|f\|_{B^{\kappa-\frac{1-\eps''}{2H}}_{p,r}},
\end{equation*}
with $E_{p,\kappa} = B^{\kappa}_{p,p}(\RR^d)$ if $p<+\infty$ and $E_{\infty,\kappa}=\cC^{\kappa}_w(\RR^d)$ if $p = +\infty$, where $w$ is a weight with $\eta$-polynomial growth. When $p<+\infty$ and $r=1$, one can take $\eps''=\eps$.
\end{corollary}

\begin{proof} As usual, one needs to distinguish between $p=+\infty$ and $2\le p<+\infty$. 

\paragraph{The $p<+\infty$ case.} Thanks to Lemma \ref{lemma:derivation} and \eqref{eq:estmThm31}, we obtain that, for any $\beta\in\NN^d$, $\varepsilon\in[0,1]$, $q\geq 2$ and $j\geq-1$,
\begin{align*}
\EE\left[\left\|\partial^{\beta}(\cT^{\phi,\varphi}\Delta_j f)_{s,t}\right\|_{L^p(\RR^d)}^q\right]^{\frac1q}&\leq \sum_{\gamma\leq\beta} \sum_{\substack{1\leq |\mu|\leq |\gamma|\\ \upsilon\in\mathcal{N}_{\gamma,\mu}}} C_{\beta,\gamma,\mu,\upsilon} 
\EE\left[\left\|\cT^{\Lambda_{\gamma,\upsilon}(\varphi)\partial^{\beta-\gamma}\phi , \varphi} \big(\partial^{\mu} \Delta_j f\big))\right\|_{L^p(\RR^d)}^q\right]^{\frac1q}
\\ &\lesssim (t-s)^{\frac{1+\eps}{2}} 2^{-j\frac{1-\eps}{2H}} \sum_{\gamma\leq \beta} \|\partial^{\gamma}\Delta_j f\|_{L^p(\RR^d)}
\\ &\lesssim (t-s)^{\frac{1+\eps}{2}} 2^{-j\frac{1-\eps}{2H}+|\beta|j}  \|\Delta_j f\|_{L^p(\RR^d)}.
\end{align*}
In particular, we deduce that, for any $k\in\NN$, we have
\begin{align*}
\EE\left[\left\|(\cT^{\phi,\varphi}\Delta_j f)_{s,t}\right\|_{B_{p,p}^{2k}}^q\right]^{\frac1q} &\simeq \EE\left[\left\|(1-\Delta)^{k}(\cT^{\phi,\varphi}\Delta_j f)_{s,t}\right\|_{L^{p}(\RR^d)}^q\right]^{\frac1q}
\\ &\lesssim (t-s)^{\frac{1+\eps}{2}} 2^{-j\frac{1-\eps}{2H}+2kj}  \|\Delta_j f\|_{L^p(\RR^d)},
\end{align*}
and, thus, by interpolation, we deduce that, for any $\varrho>0$,
\begin{equation*}
\EE\left[\left\|(\cT^{\phi,\varphi}\Delta_j f)_{s,t}\right\|_{B_{p,p}^{\varrho}}^q\right]^{\frac1q}\lesssim (t-s)^{\frac{1+\eps}{2}} 2^{-j\frac{1-\eps}{2H}+\varrho j}  \|\Delta_j f\|_{L^p(\RR^d)}.
\end{equation*}
In particular, this yields
\begin{equation*}
\EE\left[\left\|(\cT^{\phi,\varphi}f)_{s,t}\right\|_{B^\varrho_{p,p}(\RR^d)}^q\right]^{\frac1q}\lesssim (t-s)^{\frac{1+\eps}{2}} \|f\|_{B^{-\frac{1-\eps}{2H}+\varrho}_{p,1}},
\end{equation*}
which gives the desired result by using Kolmogorov's continuity theorem.

\paragraph{The $p=+\infty$ case.} For any $\beta\in\NN^d$, any $x,y\in\RR^d$, we have, for any $\theta\in[0,1]$ and $j\geq -1$,
\begin{align*}
&\left|\partial^{\beta}(\cT^{\phi,\varphi}\Delta_j f)_{s,t}(x) - \partial^{\beta}(\cT^{\phi,\varphi}\Delta_jf)_{s,t}(y)\right|
\\ &\hspace{3em} \leq |\nabla \partial^{\beta}(\cT^{\phi,\varphi}\Delta_j f)_{s,t}(z)|^{\theta} \left(\left|\partial^{\beta}(\cT^{\phi,\varphi}\Delta_j f)_{s,t}(x)\right| + \left|\partial^{\beta}(\cT^{\phi,\varphi}\Delta_j f)_{s,t}(y)\right|\right)^{1-\theta} |x-y|^{\theta},
\end{align*}
for some $z\in \{(1-\zeta)x + \zeta y:\zeta\in[0,1]\}$. It follows from Lemma \ref{lemma:derivation} and \eqref{eq:estmThm31_inf} that, for any $q\geq 2$,
\begin{align*}
\EE\left[\left|\partial^{\beta}(\cT^{\phi,\varphi}\Delta_j f)_{s,t}(x) - \partial^{\beta}(\cT^{\phi,\varphi} \Delta_jf)_{s,t}(y)\right|^q\right]^{1/q} \leq |x-y|^{\theta} \EE\left[\left|\nabla \partial^{\beta}(\cT^{\phi,\varphi}\Delta_j f)_{s,t}(z)\right|^{2\theta q}\right]^{1/(2q)}\\ \times\left(\EE\left[\left| \partial^{\beta}(\cT^{\phi,\varphi}\Delta_j f)_{s,t}(x)\right|^{2(1-\theta)q}\right]^{1/(2(1-\theta)q)}+\EE\left[\left| \partial^{\beta}(\cT^{\phi,\varphi}\Delta_j f)_{s,t}(y)\right|^{2(1-\theta)q}\right]^{1/(2(1-\theta)q)} \right)^{1-\theta}
\\ \lesssim |x-y|^{\theta} (t-s)^{\frac{1+\varepsilon}2}2^{-j\frac{1-\varepsilon}{2H}}\left(\sum_{\gamma\leq \beta} \|\nabla \partial^{\gamma}\Delta_j f\|_{L^{\infty}(\RR^d)}\right)^{1-\theta} \left(\sum_{\gamma\leq \beta} \|\partial^{\gamma}\Delta_j f\|_{L^{\infty}(\RR^d)}\right)^{1-\theta}.
\\ \lesssim |x-y|^{\theta} (t-s)^{\frac{1+\varepsilon}2}2^{-j\frac{1-\varepsilon}{2H}}2^{(|\beta| + \theta)j}\|\Delta_j f\|_{L^{\infty}(\RR^d)}.
\end{align*}
We also have, by Lemma \ref{lemma:derivation} and \eqref{eq:estmThm31_inf},
\begin{equation*}
\EE\left[\left|\partial^{\beta}(\cT^{\phi,\varphi}\Delta_j f)_{s,t}(x)\right|^q\right]^{1/q} \lesssim (t-s)^{\frac{1+\varepsilon}2}2^{-j\frac{1-\varepsilon}{2H}}2^{(|\beta| + \theta)j}\|\Delta_j f\|_{L^{\infty}(\RR^d)}.
\end{equation*}
From here, we use Theorem \ref{thm:Kolmogo} and Theorem \ref{thm:Kolmogo_std} to deduce the desired result.
\end{proof}

\section{Malliavin Calculus}
In order to be self-contained, and in the spirit of Section \ref{section:rough_paths} we recall in this section some facts about Malliavin calculus. Most of them are well-known facts, but Subsection \ref{subsec:conditional} yields some non-standard estimates that we will use in the following. In particular, we will derived a (not so surprising) conditional integration by part formula, which will allow us to check whenever a flow generated by a RDE driven by a Gaussian rough path satisfy the conditions of Theorem \ref{theorem:kolmo} or Corollary \ref{corollary:averaged_holder}.
One may consult 
 \cite{nualart2006malliavin,nualart2009lectures} for more details. \label{section:Malliavin}

\subsection{Isonormal Gaussian processes}
\begin{definition}(Isonormal Gaussian process) An Isonormal Gaussian process is the set of:
\begin{enumerate}
\item a real and separable Hilbert space $\cH$ whose scalar product is denoted as $\langle\cdot,\cdot\rangle_{\cH}$ and its norm as $\|\cdot\|_\cH$,
\item a complete probability space $(\Omega,\cF,\PP)$,
\item a real-valued Gaussian process $W:h\in \cH\to W(h)$, \textit{i.e.} $(W(h))_{h\in \cH}$ is a family of centered Gaussian random variables such that $\EE[W(h)W(g)] = \langle h,g\rangle_\cH$, for any $h,g\in \cH$.
\end{enumerate}
\end{definition}
\begin{remark}\leavevmode
\begin{enumerate}
\item By Kolmogorov's theorem, given only $\cH$, we can construct $(\Omega,\cF,\PP)$ and $W$ satisfying the above conditions.
\item The mapping $h\to W(h)$ is linear.
\end{enumerate}
\end{remark}

The two following examples will be of interest below. The first one is a classical construction linked to the standard Brownian motion.
\begin{example}(Isonormal Gaussian process associated to the Brownian motion) Let $(B_t)_{t\geq 0}$ be a $d$-dimensional Brownian motion defined on its canonical probability space $(\Omega,\cF,\PP)$. Let $\cH = L^2(\RR^+;\RR^d)$ and, for any $h\in \cH$, let us define $W(h)$ as the Wiener integral
\begin{equation*}
W(h) = \sum_{k = 1}^d \int_0^{+\infty} h_k(s) \dd B^{k}_s.
\end{equation*}
\end{example}
The second example deals with a generic construction of an isonormal Gaussian process associated to a generic Gaussian process. This example will be of primer importance in the following, as we will always work in this specific setting. 
\begin{example}
Let us consider any vector-valued continuous centered Gaussian process $W = (W^1,\cdots,W^d)$ with iid components and let us construct a corresponding (vector-valued) isonormal process. Fix $j\in\{1,\ldots,d\}$, $d\geq 1$. We denote the covariance function of $W$ as
\begin{equation*}
R_W(t,s) : = \EE\left[W_t^{j}W_s^{j} \right].
\end{equation*}
Fix $T>0$. We now consider $\cE$ the set of step functions on $[0,T]$
\begin{equation*}
\cE : = \left\{f = \sum_{k = 1}^{m} a_k\ind_{[t_{k},t_{k+1}]}:\; m\in\NN^*,\, (a_k)_{1\leq k\leq m}\in\RR^{m-1},\, (t_k)_{1\leq k\leq m+1}\in \Pi([0,T]) \right\}.
\end{equation*}
We define $\cH$ as the closure of $\cE$ with respect to the scalar product
\begin{equation*}
\langle \ind_{[0,t]},\ind_{[0,s]}\rangle_\cH = R_W(t,s),
\end{equation*}
which leads to the norm defined on $\cE$ by 
\begin{equation*}
\|f\|_{\cH}^2 = \sum_{k,\ell = 1}^m a_k a_{\ell} \langle \ind_{[0,t_k]}, \ind_{[0,t_{\ell}]}\rangle_{\cH} = \sum_{k,\ell = 1}^m a_k a_{\ell}R_W(t_{k},t_{\ell}).
\end{equation*}
We can extend the linear  mapping $\cK:\ind_{[0,t]}\in\cE\to W^{j}(t)\in L^2(\Omega)$ to an isometry $\cK:h\in\cH\to W^{j}(h)$ since
\begin{equation*}
\EE\left[ \cK(\ind_{[0,t]})  \cK(\ind_{[0,s]})\right]= \EE\left[W^{j}_t  W^{j}_s\right] =  R_W(t,s) =  \langle \ind_{[0,t]}, \ind_{[0,s]}\rangle_{\cH}.
\end{equation*}
Then, $\cK:h\in\cH\to W^{j}(h)$ is an isonormal Gaussian process. Let us construct the vector-valued isonormal Gaussian process $\cK:h\in\cH^{\oplus d} \to W(h) = (W^{1}(h_1),\cdots,W^{d}(h_d))\in L^2(\Omega;\RR^d)$. In the following, we will treat the case $n = 1$ without loss of generality.
\end{example}

\begin{remark}\leavevmode
\begin{enumerate}
\item An interesting instance is, for example, the case where $W = B^H$ the fractional Brownian motion with Hurst parameter $H\in(0,1)$ and covariance
\begin{equation*}
R_{B^H}(t,s) = \EE[B^H_tB^H_s] = \frac12\left( t^{2H}+s^{2H} - |t-s|^{2H}\right).
\end{equation*}
\item For $W_0 =0$ (and $R_W(0,0) = 0$), we have, by \cite[Proposition 4]{cass2009non},
\begin{equation}\label{eq:calHscalarprod}
\langle h_1,h_2\rangle_{\cH} = \int_0^1\int_0^1 h_1(t)h_2(s)\dd R_W(t,s),
\end{equation}
whenever the $2$d Young integral on the right-hand-side is well-defined.
\end{enumerate}
\end{remark}

\textcolor{black}{Fix $[a,b]\subset [0,T]$ for the rest of the section. }Since $\cH$ is the closure of indicator functions, for any $[a,b]\subset [0,T]$, we can also define $\cH([a,b])$ the restriction of $\cH = \cH([0,T])$ to $[a,b]$. Thus, for almost every $s\in[a,b]^c$, $h(s) = 0$ for $h\in \cH([a,b])$. Furthermore, we directly deduce that, for any \textcolor{black}{$h,g\in\cH([a,b])$},
\begin{equation*}
\langle h, g\rangle_{\cH([a,b])}  = \langle h\ind_{[a,b]}, g\ind_{[a,b]}\rangle_{\cH}.
\end{equation*}
\textcolor{black}{In particular, if $h\in\cH([a,b])$ then $h = h\ind_{[a,b]}\in\cH$. For any set $B$ which is a finite union of intervals of $[0,T]$, we denote $\cF_{B}$ the $\sigma$-algebra generated by $\{W(\ind_{[u,v]}),\, [u,v]\subset B\}$.}

\subsection{The Malliavin derivative}

We consider the set of smooth cylindrical fields $S$ given by
\textcolor{black}{\begin{equation*}
 S = \{f(W(\ind_{[u_1,v_1]}),\ldots,W(\ind_{[u_n,v_n]})):\;f\in C^{\infty}_p(\RR^n),\; [u_k,v_k]\subset [0,1],\; \forall k\in\{1,\ldots,n\},\; n\geq 0\},
\end{equation*}}
where $\cC^{\infty}_p(\RR^n)$ denotes the set of functions from $\RR^n$ to $\RR$ that are infinitely differentiable and with all their partial derivatives having polynomial growth. Furthermore, we denote $C^{\infty}_b(\RR^n)$ the subset of $C^{\infty}_p(\RR^n)$ where all the partial derivatives are bounded. The smooth random variables associated to this space $C^\infty_b$ is denoted by $ S_b$. \textcolor{black}{For any set $B$ which is a finite union of intervals of $[0,T]$, we also denote $S_b(B)$ (resp. $S(B)$)  the natural restriction of $S_b = S_b(B)$ (resp. $S = S(B)$) to $B$.} We note that \textcolor{black}{$S_b(B)$ (and $S(B)$) is dense in \textcolor{black}{$L^2(\Omega,\mathcal{F}_{B},\mathbb{P})$}}.

\textcolor{black}{\begin{definition}
For any $F=f\big(W(\ind_{[u_1,v_1]})\cdots W(\ind_{[u_n,v_n]})\big)\in S$ with $f\in C^\infty_p$, we define the Malliavin derivative of $F$ restricted on $[a,b]\subset[0,1]$ as
\begin{equation*}
D_{[a,b]}F = \sum_{k = 1}^n \partial_k f(W(\ind_{[u_1,v_1]}),\ldots,W(\ind_{[u_n,v_n]})) \ind_{[u_k,v_k]}\ind_{[a,b]},
\end{equation*}
which takes values in $\cH([a,b])$.
\end{definition}}

\textcolor{black}{We can directly check that, for any $F,G\in S$, we have the "chain rule" relation
\begin{equation*}
D_{[a,b]}(FG) = G D_{[a,b]}F + F D_{[a,b]}G.
\end{equation*}}

\textcolor{black}{Let us also introduce} the so-called \textit{Cameron-Martin space} denoted $\cH_1$ which is defined as the completion of
\begin{equation*}
\tilde{\cE} = \left\{ t\to \sum_{k = 1}^n a_kR_W(t_k,t):\; n\in\NN,\, (a_k)_{1\leq k\leq n},\, (t_k)_{1\leq k\leq n}\subset[0,T] \right\},
\end{equation*}
with respect to the scalar product
\begin{equation*}
\langle R_W(t,\cdot), R_W(s,\cdot)\rangle_{\cH_1} = R_W(t,s).
\end{equation*}
Then, the mapping defined by 
\begin{equation}\label{eq:Rmapp}
\cR : \ind_{[0,t]}\in\cE \to R_W(t,\cdot) \in\tilde{\cE}
\end{equation}
can be extended to an isometry $\cR:\cH\to\cH_1$ since
\begin{equation*}
\langle \cR(\ind_{[0,t]}),\cR(\ind_{[0,s]}) \rangle_{\cH_1} = R_W(t,s) = \langle\ind_{[0,t]},\ind_{[0,s]} \rangle_{\cH}.
\end{equation*}
\textcolor{black}{In the same fashion as for $\mathcal{H}$, we define $\cH_1([a,b])$ the restriction of $\cH_1 = \cH_1([0,T])$ to $[a,b]$.}

We have the following integration by parts formula.
\textcolor{black}{\begin{lemma}
For any $F,G\in S$ and $h\in\cH([a,b])$, we have
\begin{equation*}
\EE[G \langle D_{[a,b]}F,h\rangle_{\cH([a,b])}] = -\EE[F \langle D_{[a,b]}G,h\rangle_{\cH([a,b])}]  + \EE[F GW(h)].
\end{equation*}
\end{lemma}}

Furthermore, we have the following result.
\textcolor{black}{\begin{proposition}\label{prop:closeD}
For any $p\geq 1$, the linear operator $D_{[a,b]}$ is closable from $L^p(\Omega,\mathcal{F}_{[0,1]},\mathbb{P})$ to $L^p(\Omega,\mathcal{F}_{[0,1]},\mathbb{P}; \cH([a,b]))$.
\end{proposition}}

We can iteratively define, for any $m\geq 1$,
\textcolor{black}{\begin{equation*}
D_{[a,b]}^m F = \sum_{k_1,k_2,\ldots,k_m = 1}^n \partial_{k_1,k_2,\ldots,k_m}f(W(\ind_{[u_1,v_1]}),\ldots,W(\ind_{[u_n,v_n]})) \left(\otimes_{\ell = 1}^m(\ind_{[u_{k_\ell},v_{k_\ell}]}\ind_{[a,b]})\right)
\end{equation*}}
which takes values in \textcolor{black}{$\cH([a,b])^{\otimes m}$}. Moreover, following Proposition \ref{prop:closeD}, we can close \textcolor{black}{$D^m_{[a,b]}$}.
\begin{proposition}
For any $m\geq 1$ and $p\geq 1$, the linear operator \textcolor{black}{$D^m_{[a,b]}$} is closable from $ S$ to \textcolor{black}{$L^p(\Omega;\cH([a,b])^{\otimes m})$}.
\end{proposition}

By using the same notation for its extension, the domain of the operator \textcolor{black}{$D^m_{[a,b]}$ is the space $\DD^{m,p}_{[a,b]}$ which is the completion of $ S$ with respect to the norm
\begin{equation*}
\|F\|_{\DD^{m,p}_{[a,b]}} : = \left( \EE[|F|^p] + \sum_{k = 1}^m \EE[\|D^k_{[a,b]} F\|_{\cH([a,b])^{\otimes k}}^p]\right)^{1/p}.
\end{equation*}}
In the previous norm, in the case of multidimensional processes, we remark that \textcolor{black}{$\|\cdot\|_{\cH([a,b])^{\otimes k}}$} is not an operator norm but the Hilbert-Schmidt norm. \textcolor{black}{For $E$ a Banach space,} we also denote by \textcolor{black}{$\DD_{[a,b]}^{m,p}(E)$ the $E$-valued random variables that belong in $\DD_{[a,b]}^{m,p}$}. Furthermore, for any $k,p\geq1$, we have H\"older's inequality
\textcolor{black}{\begin{equation*}
\|FG\|_{\DD_{[a,b]}^{k,p}}\leq \|F\|_{\DD_{[a,b]}^{k,r}}\|G\|_{\DD_{[a,b]}^{k,q}},
\end{equation*}}
for any $r,q\geq 1$ such that $1/r+1/q = 1/p$.
We now state a chain rule with respect to the Malliavin derivative.

\begin{proposition}\label{prop:chainrule}
Let $g\in C^1(\RR^d;\RR)$ with bounded derivatives, $p\geq 1$ and $F = (F^{1},\ldots,F^{d})$ be a random vector such that \textcolor{black}{$F^k\in \DD_{[a,b]}^{1,p}$} for any $k\in\{1,\ldots,d\}$. Then, \textcolor{black}{$g(F)\in\DD_{[a,b]}^{1,p}$} and
\textcolor{black}{\begin{equation*}
D_{[a,b]}(g(F)) = \sum_{k = 1}^d \partial_k g(F) D_{[a,b]}F^{k}.
\end{equation*}}
\end{proposition}

\subsection{The divergence operator}

The divergence operator \textcolor{black}{$\delta_{[a,b]}$ is the adjoint of the derivative $D_{[a,b]}$}. In fact, \textcolor{black}{$\delta_{[a,b]}$} is an unbounded operator from \textcolor{black}{$L^2(\Omega;\cH([a,b])) = \DD^{0,2}_{[a,b]}(\cH([a,b]))$ to $L^2(\Omega)$} such that:
\begin{enumerate}
\item its domain \textcolor{black}{$\textrm{Dom}(\delta_{[a,b]})$ is the set of random variables $u\in L^2(\Omega;\cH([a,b]))$ such that, for all $F\in\DD_{[a,b]}^{1,2}$,
\begin{equation*}
\left|\EE\left[\langle D_{[a,b]}F, u \rangle_{\cH([a,b])}\right] \right| \leq c_u\|F\|_{L^2(\Omega)},
\end{equation*}}
\item for any \textcolor{black}{$u\in\textrm{Dom}(\delta_{[a,b]})$, we have $\delta_{[a,b]}(u)\in L^2(\Omega)$ and the following duality relation holds, for all $F\in\DD_{[a,b]}^{1,2}$,
\begin{equation*}
\EE\left[\langle D_{[a,b]}F, u \rangle_{\cH([a,b])}\right]  = \EE\left[F \delta_{[a,b]}(u)\right].
\end{equation*}}
\end{enumerate}

We have the following result concerning the continuity of $\delta$.
\begin{theorem}\label{thm:condelta}
For any $p>1$ and $m\geq 1$, the operator \textcolor{black}{$\delta_{[a,b]}$ is continuous from $\DD_{[a,b]}^{m-1,p}(\cH([a,b]))$ to $\DD^{m,p}_{[a,b]}$. That is, the following inequality holds for any $u\in\DD_{[a,b]}^{m-1,p}(\cH([a,b]))$,
\begin{equation*}
\|\delta(u)\|_{\DD^{m-1,p}_{[a,b]}} \leq c_{m,p} \|u\|_{\DD^{m,p}_{[a,b]}(\cH([a,b]))}.
\end{equation*}}
\end{theorem}
\noindent This yields in particular that \textcolor{black}{$\DD^{m,p}_{[a,b]}(\cH([a,b]))\subset \textrm{Dom}(\delta_{[a,b]})$} for any $p>1$ and $m\geq 1$. 

For any random variable \textcolor{black}{$F\in\DD^{1,2}_{[a,b]}$, we know that $D_{[a,b]}F$ is a stochastic process in $\cH([a,b])$ that we can denote $(D_{[a,b],t}F)_{t\in[a,b]}$} which is defined almost surely with respect to the measure $\lambda\times\PP$ (where $\lambda$ is the usual Lebesgue measure). With this, we can deduce a \textit{local property} of the derivative operator.

\begin{lemma}\label{lem:localD}
Let \textcolor{black}{$[u,v]\subset[0,T]$ and $F \in\DD_{[a,b]}^{1,2}\cap L^2(\Omega,\cF_{[u,v]},\PP)$}. Then we have
\textcolor{black}{\begin{equation*}
D_{[a,b],s} F(\omega) = 0,
\end{equation*}}
for $(\lambda\times\PP)$-almost every \textcolor{black}{$(s,\omega)\in ([a,b]\backslash[u,v])\times \Omega$}.
\end{lemma}
\begin{proof}
\textcolor{black}{We have, for any $F\in  S_b([u,v])$,
\begin{equation*}
D_{[a,b],s}F = \sum_{k = 1}^n \partial_k f(W(\ind_{[u_1,v_1]}),\ldots,W(\ind_{[u_n,v_n]}))\ind_{[u_k,v_k]}(s)\ind_{[a,b]}(s),
\end{equation*}
which yields the desired result since, for $1\leq k\leq n$, $[u_k,v_k]\subset[u,v]$.}
\end{proof}

\subsection{Conditional Integration by parts formula}\label{subsec:conditional}

\textcolor{black}{In the previous section, we have seen that $\delta_{[a,b]}$ is the adjoint operator of $D_{[a,b]}$ and, as such, it satisfies a duality relation. It turns out that this duality relation also holds under a conditional expectation. That is,} for any \textcolor{black}{$F\in S$ and $G_{[a,b]^c}\in S([a,b]^c)$, with $[a,b]^c = [0,T]\backslash[a,b]$}, we can see that, thanks to Lemma \ref{lem:localD},
\textcolor{black}{\begin{align*}
\EE\left[  \EE\left[\langle D_{[a,b]}F, u \rangle_{\cH([a,b])} \middle|\cF_{[a,b]^c}\right]G_{[a,b]^c}\right] & = 
 \EE\left[\langle D_{[a,b]}F,u\rangle_{\cH([a,b])} G_{[a,b]^c}\right] 
 \\&= -\EE\left[F \langle DG_{[a,b]^c},u\rangle_{\cH([a,b])} \right] + \EE\left[\langle D_{[a,b]}(FG_{[a,b]^c}),u\rangle_{\cH([a,b])}\right]
\\&= \EE\left[FG_{[a,b]^c}\delta_{[a,b]}(u)\right] = \EE\left[\EE\left[F\delta_{[a,b]}(u)\middle|\cF_{[a,b]^c}\right]G_{[a,b]^c}\right].
\end{align*}}
This yields in fact the following duality relation under a conditional expectation since \textcolor{black}{$S$ (resp. $S([a,b]^c)$) is dense in $L^2(\Omega)$ (resp. $L^2(\Omega,\mathcal{F}_{[a,b]^c},\mathbb{P})$)}
\begin{equation}\label{eq:Dualcond}
\EE\left[\langle DF, u \rangle_{\cH([a,b])} \middle|\cF_{[a,b]^c}\right] = \EE\left[F\delta_{[a,b]}(u)\middle|\cF_{[a,b]^c}\right].
\end{equation}
We can now proceed to prove a conditional integration by parts formula.
\begin{proposition}
Let \textcolor{black}{$F\in\DD^{1,2}_{[a,b]}$}, $G$ be a random variable and $u$ be an $\cH([a,b])$-valued random variable such that
\textcolor{black}{\begin{equation*}
\langle D_{[a,b]}F,u\rangle_{\cH([a,b])} \neq 0\quad\textrm{and}\quad G \frac{u}{ \langle D_{[a,b]}F,u\rangle_{\cH([a,b])}} \in\textrm{Dom}(\delta_{[a,b]}).
\end{equation*}}
Then, for any function $f\in C^1$ with bounded derivatives, we have that
\begin{equation*}
\EE\left[f'(F)G \middle| \cF_{[a,b]^c}\right] = \EE\left[f(F)H_{[a,b]}(F,G)\middle| \cF_{[a,b]^c}\right],
\end{equation*}
where
\begin{equation*}
H_{[a,b]}(F,G) : = \delta_{[a,b]}\left(G \frac{u}{ \langle D_{[a,b]}F,u\rangle_{\cH([a,b])}}\right).
\end{equation*}
\end{proposition}
\begin{proof}
We apply Proposition \ref{prop:chainrule} to deduce that
\begin{equation*}
f'(F) = \left\langle D_{[a,b]}(f(F)) ,\frac{u}{ \langle D_{[a,b]}F , u \rangle_{\cH([a,b])} }\right\rangle_{\cH([a,b])}.
\end{equation*}
Then, it follows from \eqref{eq:Dualcond} that
\begin{align*}
\EE\left[ f'(F) G\middle|\cF_{[a,b]^c}\right] &= \EE\left[ \langle D_{[a,b]}(f(F)),u\rangle_{\cH([a,b])} (\langle D_{[a,b]} F,u\rangle_{\cH([a,b])})^{-1}G\middle|\cF_{[a,b]^c}\right]
\\ &= \EE\left[ \langle D_{[a,b]}(f(F)),Gu(\langle D_{[a,b]} F,u\rangle_{\cH([a,b])})^{-1} \rangle_{\cH([a,b])} \middle|\cF_{[a,b]^c}\right] 
\\ &= \EE\left[ f(F)\delta_{[a,b]}\left(Gu(\langle D_{[a,b]} F,u\rangle_{\cH([a,b])})^{-1} \right)\middle|\cF_{[a,b]^c}\right],
\end{align*}
which is the desired result.
\end{proof}

We may now introduce the Malliavin matrix associated to a random vector.
\begin{definition}\label{defi:covmat}
Let $p>1$ and $F = (F^1,\ldots,F^d)$ be a random vector such that $F^k\in\DD_{[a,b]}^{1,p}$ for all $k\in\{1,\ldots,d\}$. The covariance matrix $\gamma_F$ associated to $F$ is defined as
\begin{equation*}
\gamma_{F,[a,b]} : = (\langle D_{[a,b]}F^i,D_{[a,b]}F^j \rangle_{\cH([a,b])})_{(i,j)\in\{1,\ldots,d\}^2}.
\end{equation*}
\end{definition}

We then have the following Lemma \cite[Lemma 7.2.3]{nualartIntroductionMalliavinCalculus2018} :

\begin{lemma}\label{lem:diffgamma}
Let $\gamma$ be a $d\times d$ random matrix such that $\textrm{det}(\gamma)>0$ almost surely and $\textrm{det}(\gamma)^{-1}\in \cap_{p\geq 1}L^p(\Omega)$. We assume that $(\gamma)_{i,j}\in\DD_{[a,b]}^{\infty} : = \cap_{p\geq 1}\cap_{k\geq 1} \DD_{[a,b]}^{k,p}$ for any $(i,j)\in\{1,\ldots,d\}^2$. Then, $(\gamma^{-1})_{i,j}\in\DD_{[a,b]}^{\infty}$, for any $(i,j)\in\{1,\ldots,d\}^2$, and
\begin{equation*}
D_{[a,b]}(\gamma^{-1})_{i,j} = -\sum_{k,\ell = 1}^m (\gamma^{-1})_{i,k}(\gamma^{-1})_{\ell,j} D_{[a,b]}\gamma_{k,\ell}.
\end{equation*}
\end{lemma}

Thanks to the Malliavin matrix, we define a \textit{nondegeneracy condition} for random vectors.
\begin{definition}
We say that a random vector $F = (F^1,\ldots,F^d)$, whose components are in $\DD^{\infty}_{[a,b]}$, is nondegenerate if the Malliavin matrix $\gamma_{F,[a,b]}$ is invertible almost surely and
\begin{equation*}
\textrm{det}(\gamma_{F,[a,b]})^{-1}\in \cap_{p\geq 1}L^p(\Omega).
\end{equation*}
\end{definition}

We now state a more general conditional integration by parts result for nondegenerate random vectors.

\begin{proposition}\label{prop:integration_by_part}
Let $F = (F^1,\ldots,F^d)$ be a nondegenerate  random vector, $G\in\DD_{[a,b]}^{\infty}$ and $g\in C^{\infty}_p(\RR^d)$. Then, for any multiindex $\alpha \in \{1,\ldots,d\}^k$, $k\geq 1$, there exists an element $H_{\alpha,[a,b]}(F,G)\in\DD^{\infty}_{[a,b]}$ such that
\begin{equation*}
\EE\left[ \partial^{\alpha}g(F) G\middle|\cF_{[a,b]^c}\right] = \EE\left[ g(F) H_{\alpha,[a,b]}(F,G)\middle|\cF_{[a,b]^c}\right].
\end{equation*}
The variable $H_{\alpha,[a,b]}(F,G)$ is given recursively by
\begin{equation*}
H_{\alpha,[a,b]}(F,G) = H_{\alpha_k,[a,b]}(F,H_{(\alpha_1,\ldots,\alpha_{k-1}),[a,b]}(F,G))
\end{equation*}
where, for any $i\in\{1,\ldots,d\}$,
\begin{equation*}
H_{i,[a,b]}(F,G) = \sum_{j = 1}^d\delta_{[a,b]}\left(G (\gamma_{F,[a,b]}^{-1})_{i,j} DF^j \right).
\end{equation*}
\end{proposition}

\begin{proof}
It follows from the chain rule that, for any $j\in\{1,\ldots,m\}$,
\begin{equation*}
\langle D_{[a,b]}(g(F)),D_{[a,b]}F^j \rangle_{\cH([a,b])} = \sum_{i = 1}^d \partial_i g(F) \langle D_{[a,b]}F^i,D_{[a,b]}F^j\rangle_{\cH([a,b])} = \sum_{i = 1}^d \partial_i g(F) (\gamma_{F,[a,b]})_{i,j}.
\end{equation*}
Thus, we obtain that
\begin{equation*}
\partial_i g(F) = \sum_{j = 1}^d (\gamma_{F,[a,b]}^{-1})_{i,j} \langle D_{[a,b]}(g(F)),DF^j \rangle_{\cH([a,b])}.
\end{equation*}
By exploiting the duality relation \eqref{eq:Dualcond}, this gives
\begin{align*}
\EE\left[ \partial_ig(F) G\middle|\cF_{[a,b]^c}\right] &= \EE\left[   \left\langle D_{[a,b]}(g(F)),\sum_{j = 1}^d G(\gamma_{F,[a,b]}^{-1})_{i,j} D_{[a,b]}F^j \right\rangle_{\cH([a,b])}\middle|\cF_{[a,b]^c}\right] 
\\&= \EE\left[  g(F) H_{i,[a,b]}(F,G)\middle|\cF_{[a,b]^c}\right].
\end{align*}
Thanks to Lemma \ref{lem:diffgamma}, we can deduce that $\gamma_{F,[a,b]}\in\DD^{\infty}_{[a,b]}$. Furthermore, by Theorem \ref{thm:condelta}, this yields that $H_{i,[a,b]}(F,G)\in\DD^{\infty}_{[a,b]}$ if $G\in\DD^{\infty}_{[a,b]}$. From here, we can proceed by induction and deduce the result.
\end{proof}

We finally give an estimate for the random variables $H_{\alpha,[a,b]}$.
\begin{lemma}\label{lemma:boundeH}
Let $k\geq 1$ and $p>1$. For any $\alpha\in\{1,\ldots,d\}^k$, we have
\begin{equation*}
\|H_{\alpha,[a,b]}(F,G)\|_{L^p(\Omega)} \lesssim_{\alpha,d,p} \|G\|_{\DD^{k,q}_{[a,b]}} \prod_{\ell = 1}^{k} \left\|\sum_{j =1}^d (\gamma_{F,[a,b]}^{-1})_{\alpha_\ell,j}D_{[a,b]}F^j \right\|_{\DD_{[a,b]}^{k-\ell+1,r_\ell}}
\end{equation*}
where $1/p = 1/q + \sum_{\ell = 1}^d 1/r_\ell$.
\end{lemma}
\begin{proof}
By Theorem \ref{thm:condelta} and H\"older's inequality, we have
\begin{align*}
\|H_{\alpha,[a,b]}(F,G)\|_{L^p(\Omega)} &\leq c_{p,d}\left\|H_{(\alpha_1,\ldots,\alpha_{k-1}),[a,b]}(F,G) \sum_{j =1}^d (\gamma_{F,[a,b]}^{-1})_{\alpha_k,j}D_{[a,b]}F^j \right\|_{\DD^{1,p}_{[a,b]}}
\\ &\leq c_{p,d}\left\|H_{(\alpha_1,\ldots,\alpha_{k-1}),[a,b]}(F,G)\right\|_{\DD^{1,q}_{[a,b]}}\left\|\sum_{j =1}^d (\gamma_{F,[a,b]}^{-1})_{\alpha_k,j}D_{[a,b]}F^j \right\|_{\DD^{1,r}_{[a,b]}},
\end{align*}
where $1/p = 1/q + 1/r$, which yields the desired result by induction.
\end{proof}

\section{Malliavin calculus and rough paths}\label{section:Malliavin2}

The idea of this Section is to use Malliavin calculus (and especially the conditional bounds proved in the previous subsection)  to derive some nice bounds on flow generated by rough differential equations. In order to do so, we will need some properties of the Cameron-Martin spaces $\cH_1$.

\subsection{Embedding of the spaces \texorpdfstring{$\cH$}{H} and \texorpdfstring{$\cH_1$}{H1}}

In this section, we give some embedding of $\cH$ and $\cH_1$ into the space of continuous functions with finite $p$-variation, with $p>0$.

\begin{definition}
Let $f:[0,T]\to \RR^d$ be a continuous function. For any $p>0$, we define its $p$-variation on $[a,b]\subset[0,T]$ as
\begin{equation*}
\|f\|_{p-var ;[a,b]} = \sup_{\pi\in\Pi([a,b])}\left( \sum_{[u,v]\in\pi} |f(v) - f(u)|^p\right)^{1/p},
\end{equation*}
where we remind that $\Pi([a,b])$ is the set of all subdivisions of $[a,b]$. The set of continuous functions with finite $p$-variation on $[a,b]$ is denoted $\cC^{p-var }([a,b];\RR^d) = \cC^{p-var }([a,b])$.
\end{definition}

A similar, mixed $(p,q)$-variation notion also exists for continuous functions on $[0,T]^2$.

\begin{definition}
Let $R:[0,T]^2\to \RR^d$ be a continuous function. For any $p,q>0$, we define its $(p,q)$-variation on the square $[a,b]\times[c,d]$ as
\begin{equation*}
\|R\|_{(p,q)-var ;[a,b]\times[c,d]} = \sup_{\substack{\pi_1\in\Pi([a,b])\\\pi_2\in\Pi([c,d])}}\left( \sum_{[u_2,v_2]\in\pi_1} \left(\sum_{[u_1,v_1]\in\pi_2}|\square_{[u_1,v_1]\times[u_2,v_2]} R|^p\right)^{q/p}\right)^{1/q},
\end{equation*}
where $\square_{[u_1,v_1]\times[u_2,v_2]} R = R(v_1,v_2) - R(v_1,u_2) - R(u_1,v_2) + R(u_1,u_2)$. The set of continuous functions with finite $(p,q)$-variation on $[a,b]\times[c,d]$ is denoted $\cC^{(p,q)-var }([a,b]\times[c,d];\RR^d) = \cC^{(p,q)-var }([a,b]\times[c,d])$.
\end{definition}

We also denote $\|R\|_{p-var ;[a,b]^2} = \|R\|_{(p,p)-var ;[a,b]^2}$ as well as $ \cC^{p-var }([a,b]^2) : =  \cC^{(p,p)-var }([a,b]^2)$. 

\begin{remark}\label{rem:bndmix}
We can see that
\begin{equation*}
\|R\|_{p\vee q-var ;[a,b]^2} \leq \|R\|_{(p,q)-var ;[a,b]^2} \leq \|R\|_{p\wedge q-var ;[a,b]^2}.
\end{equation*}
\end{remark}

Furthermore, we have the following definition.

\begin{definition}
Let $\rho\in[1,2)$ and $[a,b]\subset[0,T]$. We say that $R$ has finite H\"older-controlled $\rho$-variation on $[a,b]$ if $\|R\|_{\rho-var ;[a,b]^2}<+\infty$ and if the following estimate holds, for any $[s,t]\subset[a,b]$,
\begin{equation*}
\|R\|_{\rho-var ;[s,t]^2} \lesssim (t-s)^{1/\rho}.
\end{equation*}
\end{definition}

We are now in position to state a first assumption on the process $W$ that we will consider on the rest of the section.

\begin{assumption}\label{asm:1rhovar}
The process $W$ is a $\RR^d$-valued continuous centered Gaussian process starting at $0$ with iid components and covariance $R_W$ which belongs in $\cC^{(1,\rho)-var }([0,T]^2)$ for some $\rho\in[1,2)$. Furthermore, the following estimate holds, for any $[a,b]\subset[0,T]$,
\begin{equation}\label{eq:1rhovar}
\|R_W\|_{(1,\rho)-var ;[a,b]^2} \lesssim (b-a)^{1/\rho}.
\end{equation}
\end{assumption}

In view of Remark \ref{rem:bndmix}, we immediately see that \eqref{eq:1rhovar} implies that the covariance has finite H\"older-controlled $\rho$-variation on $[0,T]$. 



In the following, we denote
\begin{equation*}
\kappa_{a,b} : = \|R_W\|_{(1,\rho)-var ;[a,b]^2}^{1/2}
\end{equation*}
and
\begin{equation*}
\sigma_{a,b} : = \EE\left[\left(W^{j}_b - W^{j}_a \right)^2 \right]^{1/2} = (\square_{[a,b]^2}R_W)^{1/2},
\end{equation*}
for any $j\in\{1,\ldots,d\}$.

We can now proceed to state our first embedding :

\begin{theorem}[\cite{friz2016jain}, Theorem 1.1]
Let $W$ be a centered Gaussian process satisfying Assumption \ref{asm:1rhovar}. Then, for any $h\in\cH_1([a,b])$, $[a,b]\subset[0,T]$, we have
\begin{equation*}
\|h\|_{q-var ;[a,b]} \leq \kappa_{a,b}\|h\|_{\cH_1([a,b])}, 
\end{equation*}
where  $q = 1/(1/2\rho+1/2)<2$. In particular, the embedding $\cH_1([a,b])\hookrightarrow \cC^{q-var }([a,b])$ is continuous.
\end{theorem}

Furthermore, we also have the following results \cite[Remark 2.16]{gess2020density}.
\begin{proposition}
Let $W$ be a centered Gaussian process satisfying Assumption \ref{asm:1rhovar} and $[a,b]\subset[0,T]$.
\begin{enumerate}
\item Let $f\in\cC^{p-var }([a,b])$ with $1/p+1/\rho>1$. Then, $f\in\cH([a,b])$ and
\begin{equation*}
\|f\|_{\cH([a,b])}^2 = \int_a^b\int_a^b f(s)f(t) \dd R_W(s,t),
\end{equation*}
where the right-hand side is well defined as a $2$d Young integral. In particular, we have the continuous embedding $\cC^{p-var }([a,b])\hookrightarrow \cH([a,b])$ since
\begin{equation*}
\|f\|_{\cH([a,b])}^2 \lesssim \|f\|_{p-var ;[a,b]}^2 \|R_W\|_{\rho-var ;[a,b]^2}.
\end{equation*}
\item Let $f_1\in\cC^{p-var }([a,b])$ with $1/p+1/\rho>1$ and $f_2\in\cH([a,b])$. Then,
\begin{equation*}
\langle f_1,f_2 \rangle_{\cH([a,b])} = \int_a^b f_1\dd\cR f_2,
\end{equation*}
where the right-hand side is well defined as a Young integral and $\cR : \cH([a,b])\to\cH_1([a,b])$ is the isomorphism defined by \eqref{eq:Rmapp}.
\end{enumerate}
\end{proposition}

We now make an additional assumption on the process $W$.
\begin{assumption}\label{asm:correW}
Let $W$ be an $\RR^d$-valued Gaussian process with iid coordinates and covariance function $R_W$ such that
\begin{enumerate}
\item it has non-positively correlated increments, that is, for all $(t_1,t_2,t_3,t_4)\in [0,T]^4$ with $t_1<t_2<t_3<t_4$, we have
\begin{equation*}
\square_{[t_1,t_2]\times[t_3,t_4]}R_W \leq 0,
\end{equation*}
\item its covariance function is diagonally dominant, that is, for all $(t_1,t_2,t_3,t_4)\in [0,T]^4$ with $t_1<t_2<t_3<t_4$, we have
\begin{equation*}
\square_{[t_2,t_3]\times[t_1,t_4]}R_W \geq 0.
\end{equation*}
\end{enumerate}
\end{assumption}

We deduce the following inequalities \cite[Proposition 2.18]{gess2020density}.

\begin{proposition}
Let $W$ be a Gaussian process satisfying Assumption \ref{asm:1rhovar}, $p\geq 1$ such that $1/p+1/\rho>1$.
\begin{enumerate}
\item For every $f\in\cC^{p-var }([a,b])$, we have
\begin{equation*}
\|f\|_{\cH([a,b])}^2\lesssim \kappa_{a,b}^2\left(\|f\|_{p-var ;[a,b]}^2 + \|f\|_{\infty;[a,b]}^2 \right)
\end{equation*}
\item If $W$ satisfies Assumption \ref{asm:correW}, we have, for any $f\in\cC^{\gamma}([a,b])$ with $1/\rho+\gamma>1$,
\begin{equation*}
\|f\|_{\cH([a,b])}^2\geq \sigma_{a,b}^2 \min_{t\in[a,b]} |f(t)|.
\end{equation*}
\end{enumerate}
\end{proposition}

To continue, we need some non-degeneracy concerning $W$.

\begin{assumption}\label{asm:nondet}
Let $W$ be a centered continuous $\RR^d$-valued Gaussian process. We assume that there exists an $\alpha>0$ such that
\begin{equation*}
\inf_{0\leq s\leq t\leq T} (t-s)^{-\alpha} \var\left[W_t-W_s\middle| \cF_{[0,s]}\vee\cF_{[t,1]}\right] = c_W >0.
\end{equation*}
The smallest $\alpha$ that satisfies the above condition is called the index of non-determinism of $W$.
\end{assumption}

Under the previous assumption, we deduce our last inequality which is taken from \cite[Corollary 6.10]{cass2015smoothness}. 

\begin{proposition}\label{prop:interpIneq}
Let $W$ be a continuous Gaussian process satisfying Assumptions \ref{asm:1rhovar}, \ref{asm:correW} and \ref{asm:nondet}. Then, for any $f\in\cC^{\gamma}([a,b])$ with $\gamma+1/\rho>1$, we have
\begin{equation*}
\|f\|_{\infty;[a,b]} \leq 2\max\left(\sigma_{a,b}^{-1}\|f\|_{\cH([a,b])}, \sqrt{c_W}^{-1}\|f\|_{\cH([a,b])}^{2\gamma/(2\gamma + \alpha)}\|f\|_{\cC^{\gamma}([a,b])}^{\alpha/(2\gamma+\alpha)} \right).
\end{equation*}
\blue{One has the same inequality in the $p$-variation norm, for $p>1$ with $\frac{1}{p}+\frac{1}{\rho}>1$, and for $f\in \cC^{p-var}([a,b])$, : 
\begin{equation*}
\|f\|_{\infty;[a,b]} \leq 2\max\left(\sigma_{a,b}^{-1}\|f\|_{\cH([a,b])}, \sqrt{c_W}^{-1}\|f\|_{\cH([a,b])}^{2/(2 + p \alpha)}\|f\|_{p-var;[a,b]}^{p\alpha/(2+p\alpha)} \right).
\end{equation*}}
\end{proposition}

As a direct consequence, we have the following result.
\begin{corollary}\label{cor:interpIneq}
Let $W$ be a continuous Gaussian process satisfying Assumptions \ref{asm:1rhovar}, \ref{asm:correW} and \ref{asm:nondet}. Then, for any $f\in\cC^{\gamma}([a,b])$ with $\gamma+1/\rho>1$, we have
\begin{equation*}
\|f\|_{\cH([a,b])} \geq \frac{\sigma_{a,b} \|f\|_{\infty;[a,b]}}2\min\left(1,\frac{2(c_W/2)^{(2\gamma+\alpha)/4\gamma}\|f\|_{\infty;[a,b]}^{\alpha/2\gamma}}{\sigma_{a,b}|f|_{\cC^{\gamma}([a,b])}^{\alpha/2\gamma}} \right)
\end{equation*}
\blue{and
for $p>1$ with $\frac{1}{p}+\frac{1}{\rho}>1$, and for $f\in \cC^{p-var}([a,b])$,
\begin{equation*}
\|f\|_{\cH([a,b])} \geq \frac{\sigma_{a,b} \|f\|_{\infty;[a,b]}}2\min\left(1,\frac{2(c_W/2)^{(2\gamma+\alpha)/4\gamma}\|f\|_{\infty;[a,b]}^{\alpha/2\gamma}}{\sigma_{a,b}|f|_{\cC^{p-var}([a,b])}^{\alpha/2\gamma}} \right)
\end{equation*}}
\end{corollary}

\subsection{Gaussian rough paths}

Let us now focus ourselves on rough differential equations driven by Gaussian rough paths Let us consider $W$ a continuous centered Gaussian process satisfying Assumption \ref{asm:1rhovar}. It can therefore be lifted to a geometric rough path (see \cite{friz2010multidimensional}) that we denote $\mathbf{W}$ (we remind that all the needed definition are given in Section \ref{section:rough_paths}). 

\begin{proposition}\label{prop:Wlift}
Let $W$ be a continuous Gaussian process satisfying 
Assumption \ref{asm:1rhovar}. Then, almost surely, $W$ can be lifted to a geometric $p$-rough path $\bW$ with $p>2\rho$ and verifies
\begin{equation*}
\EE\left[\exp\left(\eta \|\bW\|_{p-var ;[0,T]}^2\right)\right], \quad\textrm{for some}\;\eta>0.
\end{equation*}

\end{proposition}

We also have the following useful result (which is a direct consequence of  Proposition \ref{prop:Wlift} and \cite[Theorem 15.33]{friz2010multidimensional}).

\begin{proposition}\label{prop:GaussFern}
Let $W$ be a Gaussian process satisfying Assumption \ref{asm:1rhovar}. Then, $W$ can be lifted to a geometric $p$-rough path $\bW$ with $p>2\rho$ which has $1/p$-H\"older sample paths and verifies
\begin{equation*}
\EE\left[\exp\left(\eta \|\bW\|_{1/p\textrm{-H\"ol};[0,T]}^2\right)\right], \quad\textrm{for some}\;\eta>0.
\end{equation*}
\end{proposition}

We will need the following more precise statement of Theorem \eqref{theorem:driftlessRDE} in this setting : 

\begin{theorem}\label{thm:RDEex}
Let $p\geq 1$, $\bw$ be a weakly geometric $p$-rough path, $\sigma \in\cC^{\gamma}_b(\RR^d;(\RR^d)^{\otimes 2})$ for some $\gamma>p$. Let $a\in[0,T)$ and $x\in\RR^d$. 

There exists a unique weakly geometric $p$-rough path $\bx \in \cC^{p-var }([a,T]; G^{\lfloor p \rfloor}(\RR^d))$ such that 
\[\bx_{a,\cdot}^{1} = \varphi_{a,\cdot}(x)\]
where $\varphi$ is the flow constructed in Theorem \ref{theorem:driftlessRDE}. 

Furthermore, if
$(w^{\varepsilon})_{\varepsilon>0}$ is a family of smooth functions on $[0,T]$ such that 
\begin{equation*}
S_{\lfloor p\rfloor}(w^{\varepsilon})\to \mathbf{w},
\end{equation*} in $p$-variation. 
Then, for any $0\leq a\leq t\leq T$, the solutions $(x^{\varepsilon})_{\varepsilon>0}\subset\cC^{1}([s,T])$ of
\begin{equation*}
x_t^{x,a,\varepsilon} = x + \sum_{k = 1}^d \int_a^t \sigma_k(x_r^{x,a,\varepsilon}) \dd w^{k,\varepsilon}_r,
\end{equation*}
are such that
\begin{equation*}
S_{\lfloor p\rfloor}(x^{\varepsilon}) \to \bx,
\end{equation*}
in $p$-variation. Moreover, for any $[s,t]\subset[a,T]$, the following estimates hold
\begin{equation}\label{eq:EstRDE}
\|\bx\|_{p-var ;[s,t]} \lesssim_{p,\gamma} \left(\|\sigma\|_ {\cC^{\gamma-1}_b}\|\bw\|_{p-var ;[s,t]} \vee\|\sigma\|_ {\cC^{\gamma-1}_b}^p\|\bw\|_{p-var ;[s,t]}^p\right).
\end{equation}
\end{theorem}

As a direct consequence of the previous result, we deduce the following theorem :

\begin{corollary}
Let $W$ be a Gaussian process satisfying Assumption \ref{asm:1rhovar} and let $p>2\rho$. Let $\bX$ be the unique solution of equation 
\[\dd X_t = \sigma(X_t) \dd \bW_t,\quad X_a = x,\quad t\in[a,T], \]
in the sense of Theorem \ref{thm:RDEex}. Let us denote $X^{x,a}_{t}= \bX^{1}_{a,t}$.

Then, for any $[s,t]\subset[a,T]$,
\begin{equation}\label{eq:EstRDEHol}
\sup_{x\in\RR^d}|X^{x,a}-x|_{\cC^{\frac1p}([s,t])} \lesssim_{\alpha,\gamma} \left(\|\sigma\|_ {\cC^{\gamma-1}_b}\|\bW\|_{p-var ;[s,t]} \vee\|\sigma\|_ {\cC^{\gamma-1}_b}^p\|\bW\|_{p-var ;[s,t]}^p\right).
\end{equation}
\end{corollary}	

\subsection{Malliavin calculus on rough differential equations}

The solution $X$ of a RDE driven by a Gaussian process $W$ is a random variable inheriting its randomness from $W$. 
One can thus try to differentiate $X$ in the Malliavin sense. The Malliavin derivative of $X$ can indeed be simply expressed thanks to the Jacobian $\mathbf{J}$ of the solution of equation \eqref{eq:rde_basic}. It is given by $\left(\mathbf{J}^{x,a}_{t}\right)_{(i,j)\in\{1,\ldots,d\}^2} = \partial_{x_j} (X_t^{x,a})_{i}$ and solution the following linear RDE
\begin{equation}\label{eq:Jacobian}
\mathbf{J}_t^{x,a} = \mathrm{Id} + \sum_{k = 1}^d \int_a^t \nabla \sigma_k(X_s^{x,a}) \mathbf{J}_s^{x,a} \dd\bW_s^k,
\end{equation}
where $\sigma_k$ is the $k$-th column of $\sigma$.
We notice that the inverse Jacobian $\mathbf{J}^{-1}$ solves
\begin{equation}\label{eq:invJacobian}
(\mathbf{J}^{x,a}_t)^{-1} =  \mathrm{Id} - \sum_{k = 1}^d \int_a^t (\mathbf{J}^{x,a}_s)^{-1} \nabla\sigma_k(X_s^{x,a}) \dd\bW_s^k,
\end{equation}
and that, for any $s\leq r\leq t$, we have, by the flow property of the Jacobian,
\begin{equation}\label{eq:flowJacobian}
\mathbf{J}_t^{X^{x,s}_r,r} = \mathbf{J}_t^{x,s}(\mathbf{J}^{s}_r)^{-1}.
\end{equation}

We have the following results concernant $\mathbf{J}$ as well as the Malliavin derivative of $X$ \cite{cass2010densities,cass2013integrability,inahama2014malliavin}.
\begin{proposition}
Let $W$ be a continuous centered Gaussian process satisfying Assumption \ref{asm:1rhovar} and $\sigma \in\cC^{\infty}_b(\RR^d;(\RR^d)^{\otimes 2})$.
\begin{enumerate}
\item For any $q\geq 1$, there exists a constant $c_{q,a}$ such that the Jacobian $\mathbf{J}$ defined by \eqref{eq:Jacobian} satisfies
\begin{equation}\label{eq:bndJacMom}
\EE\left[ \sup_{x\in\RR^d}\|\mathbf{J}^{x,a}\|_{p-var ;[0,T]}^q\right] = c_{q,a}.
\end{equation}
\item For any $x\in\RR^d$, $a<b\leq T$ and $t>0$, we have that $X_t\in\DD_{[a,b]}^{\infty}$ and, furthermore, the Malliavin derivative $D_{[a,T],s}X_t^{x,a}$ is such that
\begin{equation}\label{eq:MalliavinD2Jacobian}
D_{[a,b],s}X_t^{x,a} = \mathbf{J}_{t}^{x,s}\sigma(X_s^{x,a}),
\end{equation}
for any $a\leq s\leq t$, and $D_{[a,b],s}X_t^{x,a} = 0$ for all $s>t$. 
\end{enumerate}
\end{proposition}

We now state some estimates on the Malliavin derivative of $X$ as well as the associated covariance matrix $\gamma_X$. We follow the approach of \cite{inahama2014malliavin} which uses some functionals $\Xi_m$ that are equal to $D^m_h X_t$ along a single path $h\in \cH([a,b])$ (the identification extends to any direction $(h_1,h_2,\ldots,h_m)$ since these are bounded symmetric $m$-multilinear maps). For any $m\geq 2$, the Malliavin derivative of $X$ along $h^{\otimes m}\in\cH^{\otimes m}([a,b])$ is given by
\begin{align*}
D^m_h X_t^{x,a} &: = \langle D^m_{[a,b]} X_t^{x,a}, h^{\otimes m} \rangle_{\cH^{\otimes m}([a,b])} 
\\ & = 
\sum_{k = 1}^d  
\sum_{\ell = 2}^m 
\sum_{\substack{\mathbf{i}\in K_{\ell}\\ |\mathbf{i}| = m}} C_{1,\mathbf{i}}
\int_a^t  \mathbf{J}_{t}^{x,s} \nabla_{D^{i_1}_hX_s^{x,a}, D^{i_2}_h X_s^{x,a}, \ldots  ,D^{i_\ell}_h X_s^{x,a}}^{\ell} \sigma_k(X_s^{x,a}) 
\dd \bW^k_s
\\ &\hspace{1em} + \sum_{k = 1}^d\sum_{\ell = 1}^{m-1} \sum_{\substack{\mathbf{i}\in K_{\ell} \\ |\mathbf{i}| = m}} C_{2,\mathbf{i}} 
\int_a^t  \mathbf{J}_{t}^{x,s} \nabla_{D^{i_1}_hX_s^{x,a}, D^{i_2}_h X_s^{x,a}, \ldots  ,D^{i_\ell}_h X_s^{x,a}}^{\ell} \sigma_k(X_s^{x,a}) \dd h_s^k
\end{align*}
where $K_{\ell} = \{\mathbf{j}\in\NN^{\ell};\; 0<j_1\leq j_2\leq \ldots\leq j_{\ell}\}$, for some constants $\{C_{1,\mathbf{i}}\}_{\mathbf{i}\in\mathbf{K}_{\ell},2\leq \ell\leq m}$ and $\{C_{2,\mathbf{i}}\}_{\mathbf{i}\in\mathbf{K}_{\ell},1\leq \ell\leq m-1}$. Furthermore, by Equation \eqref{eq:invJacobian}, Leibniz's differentiation rule and Fa\`a di Bruno's formula, the Malliavin derivative of $(\mathbf{J})^{-1}$ along $h^{\otimes m}\in\cH^{\otimes m}([a,b])$ is such that
\begin{align*}
&D_h^m(\mathbf{J}^{x,a}_t)^{-1} =
\\ &\hspace{0.5em} -\sum_{k = 1}^d\sum_{\ell = 1}^{m} \sum_{o = 1}^{\ell}\sum_{\substack{\mathbf{i}\in K_{o}\\ |\mathbf{i}| = \ell}} C_{3,m,\ell,\mathbf{i}} 
\int_a^t(\mathbf{J}^{x,s}_t)^{-1} D_h^{m-\ell}(\mathbf{J}^{x,a}_s)^{-1} \nabla^{o}_{D^{i_1}_hX_s^{x,a},  \ldots  ,D^{i_o}_h X_s^{x,a}}(\nabla\sigma_k(X_s^{x,a})) \dd\bW_s^k
\\ &\hspace{.5em}-\sum_{k = 1}^d\sum_{\ell = 1}^{m-1} \sum_{o = 1}^{\ell}  \sum_{\substack{\mathbf{i}\in K_{o}\\ |\mathbf{i}| = \ell}} C_{4,m,\ell,\mathbf{i}}
\int_a^t (\mathbf{J}^{x,s}_t)^{-1} D_h^{m-1-\ell}(\mathbf{J}^{x,a}_s)^{-1} \nabla^{o}_{D^{i_1}_hX_s^{x,a}, \ldots  ,D^{i_o}_h X_s^{x,a}}(\nabla\sigma_k(X_s^{x,a})) \dd h_s^k
\\ &\hspace{.5em}-m\sum_{k = 1}^d
\int_a^t(\mathbf{J}^{x,s}_t)^{-1} D_h^{m-1}(\mathbf{J}^{x,a}_s)^{-1} \nabla\sigma_k(X_s^{x,a}) \dd h_s^k,
\end{align*}
for some constants $\{C_{3,m,\ell,\mathbf{i}}\}_{\substack{\mathbf{i}\in K_{o},1\leq o\leq \ell \\ 1\leq \ell\leq m}}$ and $\{C_{4,m,\ell,\mathbf{i}}\}_{\substack{\mathbf{i}\in K_{o},1\leq o\leq \ell\\ 1\leq \ell\leq m}}$.

\begin{proposition} Let $B$ be an independent copy of $W$ and  $\bB$ be the corresponding rough path above $B$.
For any $t\in[a,b]$, we denote by
\begin{equation}\label{eq:XIdef1}
\mathcal{X}_{1,t}(\bW,\mathbf{B}) = \sum_{k = 1}^d 
\int_a^t \mathbf{J}_{t}^{x,s}\sigma_k(X_s^{x,a}) \dd\mathbf{B}^k_s,
\end{equation}
\begin{multline}\label{eq:YIdef1}
\mathcal{Y}_{1,t}(\bW,\mathbf{B}) \\ = - \sum_{k = 1}^d \left(\int_a^t (\mathbf{J}_{t}^{x,a})^{-1}\nabla_{\mathcal{X}_{1,s}(\bW,\mathbf{B})}(\nabla\sigma_k(X_s^{x,a}))\dd\bW^k_s - \int_a^t (\mathbf{J}_{t}^{x,a})^{-1}\nabla\sigma_k(X_s^{x,a})\dd\mathbf{B}^k_s\right),
\end{multline}
and, for any $m\geq 2$,
\begin{align}
\mathcal{X}_{m,t}(\bW,\mathbf{B}) &=  \sum_{k = 1}^d\left(  \sum_{\ell = 2}^m \sum_{\substack{\mathbf{i}\in \mathbf{K}_{\ell} \\ |\mathbf{i}| = m}} C_{1,\mathbf{i}} \int_a^t  \mathbf{J}_{t}^{x,s} \nabla_{\mathcal{X}_{i_1,s}(\bW,\mathbf{b}),  \ldots  ,\mathcal{X}_{i_\ell,s}(\bW,\mathbf{b})}^{\ell} \sigma_k(X_s^{x,a}) \dd\bW^k_s\right.\nonumber
\\ &\hspace{1em} +\left.\sum_{\ell = 1}^{m-1} \sum_{\substack{\mathbf{i}\in \mathbf{K}_{\ell} \\ |\mathbf{i}| = m}} C_{2,\mathbf{i}} \int_a^t  \mathbf{J}_{t}^{x,s} \nabla_{\mathcal{X}_{i_1,s}(\bW,\mathbf{b}), \ldots  ,\mathcal{X}_{i_\ell,s}(\bW,\mathbf{b})}^{\ell} \sigma_k(X_s^{x,a}) \dd\mathbf{B}_s^k\right)\label{eq:XIdef2},
\end{align}
\begin{align}
&\mathcal{Y}_{m,t}(\bW,\mathbf{B}) =\nonumber
\\ & -\sum_{k = 1}^d\left(  \sum_{\ell = 1}^{m} \sum_{o = 1}^{\ell}\sum_{\substack{\mathbf{i}\in\mathbf{K}_{o}\\ |\mathbf{i}| = \ell}} C_{3,m,\ell,\mathbf{i}} \int_a^t(\mathbf{J}^{x,s}_t)^{-1} \mathcal{Y}_{m-\ell,s}(\bW,\mathbf{B}) \nabla^{o}_{\mathcal{X}_{i_1,s}(\bW,\mathbf{B}),  \ldots  ,\mathcal{X}_{i_o,s}(\bW,\mathbf{B})}(\nabla\sigma_k(X_s^{x,a})) \dd\bW_s^k\right.\nonumber
\\ &\hspace{1em}+\sum_{\ell = 1}^{m-1} \sum_{o = 1}^{\ell}  \sum_{\substack{\mathbf{i}\in\mathbf{K}_{o}\\ |\mathbf{i}| = \ell}} C_{4,m,\ell,\mathbf{i}}\int_a^t (\mathbf{J}^{x,s}_t)^{-1} \mathcal{Y}_{m-1-\ell,s}(\bW,\mathbf{B}) \nabla_{\mathcal{X}_{i_1,s}(\bW,\mathbf{B}),  \ldots  ,\mathcal{X}_{i_o,s}(\bW,\mathbf{B})}^{o}(\nabla\sigma_k(X_s^{x,a})) \dd\mathbf{B}_s^k \nonumber
\\ &\hspace{1em} +\left.m\int_a^t(\mathbf{J}^{x,s}_t)^{-1} \mathcal{Y}_{m-1,s}(\bW,\mathbf{B}) \nabla\sigma_k(X_s^{x,a}) \dd\mathbf{B}_s^k\right)\label{eq:YIdef2},
\end{align}

Then, for any $m\geq 1$ and $h\in\cH([a,b])$,

\begin{equation*}
\tilde{D}^m_h\mathcal{X}_{m,t}(\bW,\cdot) = m! D^m_h X_t^{x,a}\quad\mbox{and}\quad \tilde{D}^m_h\mathcal{Y}_{m,t}(\bW,\cdot) = m! D^m_h (\mathbf{J}_t^{x,a})^{-1},
\end{equation*}
where $\tilde{D}$ is the Malliavin derivative with respect to $B$ and the left-hand side does not depend on $\mathbf{B}$, and we have the estimates, for any $r\geq 2$,
\begin{align}
\EE\left[ \sup_{x\in\mathbb{R}^d}\left\| D^m_{[a,b]} X_t^{x,a}\right\|_{\cH^{\otimes m}([a,b])}^r \right] \lesssim_{m,r} \EE\left[\sup_{x\in\mathbb{R}^d}\left| \mathcal{X}_{m,t}(\bW,\mathbf{B})\right|^r \right]\label{eq:MalliavinXIest}
\\\mbox{and}\quad\EE\left[\sup_{x\in\mathbb{R}^d}\left\| D^m_{[a,b]} (\mathbf{J}_t^{x,a})^{-1}\right\|_{\cH^{\otimes m}([a,b])}^r \right] \lesssim_{m,r} \EE\left[\sup_{x\in\mathbb{R}^d}\left| \mathcal{Y}_{m,t}(\bW,\mathbf{B})\right|^r \right]\label{eq:MalliavinYIest}.
\end{align}
\end{proposition}
\begin{proof}
We only give the main arguments of the proofs (see \cite[Proposition 3.3]{inahama2014malliavin}). The fact that the Malliavin derivative with respect to $B$ of $\mathcal{X}_m$ (resp. $\mathcal{Y}$) is related to the Malliavin derivative of $X$ (resp. $\mathbf{J}^{-1}$) is a straightforward consequence of their expressions. Concerning the inequality, we remark that
\begin{equation*}
\left\| D^m_{[a,b]} X_t^{x,a}\right\|_{\cH^{\otimes m}([a,b])} = \frac{1}{m!} \left\|\tilde{D}^m_{[a,b]}\mathcal{X}_{m,t}(\bW,\cdot)\right\|_{\cH^{\otimes m}([a,b])} = \frac{1}{m!}\tilde{\EE}\left[ \left\|\tilde{D}^m_{[a,b]}\mathcal{X}_{m,t}(\bW,\cdot)\right\|_{\cH^{\otimes m}([a,b])}^2 \right]^{\frac12}
\end{equation*}
where $\tilde{\EE}$ is the expectation with respect to $B$, as well as
\begin{align*}
\left\| D^m_{[a,b]} (\mathbf{J}_t^{x,a})^{-1}\right\|_{\cH^{\otimes m}([a,b])} &= \frac{1}{m!} \left\|\tilde{D}_{[a,b]}^m\mathcal{Y}_{m,t}(\bW,\cdot)\right\|_{\cH^{\otimes m}([a,b])}
\\ &= \frac{1}{m!}\tilde{\EE}\left[ \left\|\tilde{D}_{[a,b]}^m\mathcal{Y}_{m,t}(\bW,\cdot)\right\|_{\cH^{\otimes m}([a,b])}^2 \right]^{\frac{1}{2}}.
\end{align*}
 Since $\mathcal{X}_m$ (resp. $\mathcal{Y}_m$) belongs to the $m$-th order inhomogeneous Wiener chaos generated by $B$, we know that all the $\DD^{2,m}_{[a,b]}$-norms are equivalent for the right-hand-side term. Hence, we obtain
\begin{align*}
\EE\left[ \sup_{x\in\mathbb{R}^d} \left\| D^m_{[a,b]} X_t^{x,a}\right\|_{\cH^{\otimes m}([a,b])}^r \right]&\lesssim_{m,r}  \EE\left[\sup_{x\in\mathbb{R}^d}\tilde{\EE}\left[ \left\|\tilde{D}^m_{[a,b]}\mathcal{X}_{m,t}(\bW,\cdot)\right\|_{\cH^{\otimes m}([a,b])}^2 \right]^{r/2} \right]
\\ &\lesssim_{m,r}   \EE\left[\tilde{\EE}\left[ \sup_{x\in\mathbb{R}^d}\left|\mathcal{X}_{m,t}(\bW,\mathbf{B})\right|^2 \right]^{r/2} \right],
\end{align*}
and, similarly,
\begin{equation*}
\EE\left[ \sup_{ x\in\mathbb{R}^d }\left\| D^m_{[a,b]} (\mathbf{J}_t^{x,a})^{-1}\right\|_{\cH^{\otimes m}([a,b])}^r \right]\lesssim_{m,r}   \EE\left[\tilde{\EE}\left[ \sup_{ x\in\mathbb{R}^d }\left|\mathcal{Y}_{m,t}(\bW,\mathbf{B})\right|^2 \right]^{r/2} \right],
\end{equation*}
for which gives the desired results.
\end{proof}

In order to push further estimates \eqref{eq:MalliavinXIest} and \eqref{eq:MalliavinYIest}, we need to estimate the rough integrals from \eqref{eq:XIdef1}, \eqref{eq:XIdef2}, \eqref{eq:YIdef1} and \eqref{eq:YIdef2}. This is done by considering, for any $m\geq1$, the function $(X,\mathbf{J},\mathbf{J}^{-1}, \mathcal{X}_1,\mathcal{X}_2,\ldots,\mathcal{X}_m, \mathcal{Y}_1,\mathcal{Y}_2,\ldots,\mathcal{Y}_m)$ 
as a solution of a (system of) RDE given by \eqref{eq:rde_basic}-\eqref{eq:Jacobian}-\eqref{eq:invJacobian}-\eqref{eq:XIdef1}-\eqref{eq:XIdef2}-\eqref{eq:YIdef1}-\eqref{eq:YIdef2} driven by $\mathbf{Z}$ which is the lifted Gaussian process $Z = (W,B)$ (that satisfies Assumption \ref{asm:1rhovar}) in the geometric $p$-rough paths. It turns out that, for any $q>2$ large enough, we have estimate \cite[Theorem 35-(i) and Corollary 66]{friz2010differential}
\begin{equation*}
\EE\left[\|\mathbf{Z}\|_{p-var ;[a,t]}^q\right] \lesssim_q \kappa_{a,t}.
\end{equation*}
We remark that we can not immediately apply the estimate \eqref{eq:EstRDE} since the vector-fields in \eqref{eq:Jacobian}-\eqref{eq:invJacobian}-\eqref{eq:XIdef1}-\eqref{eq:XIdef2}-\eqref{eq:YIdef1}-\eqref{eq:YIdef2} are not bounded with respect to $(\mathbf{J},\mathbf{J}^{-1}, \mathcal{X}_1,\mathcal{X}_2,\ldots,\mathcal{X}_m,$ $ \mathcal{Y}_1,\mathcal{Y}_2,\ldots,\mathcal{Y}_m)$ but have (as well as their derivatives) polynomial growth. 

To proceed, as in \cite{gess2020density}, we use an induction argument. We first consider $V = (X,Y)$ where, for any $1\leq k\leq d$,
\begin{equation}\label{eq:RDEY}
(Y_t)^k = \int_a^t  \nabla \sigma_k(X_s^{x,a})d\bW_s^k.
\end{equation}
Then, $V$ is the solution of \eqref{eq:rde_basic}-\eqref{eq:RDEY} driven by $\bZ$ and, thus, is a geometric $p$-rough path (denoted $\mathbf{V}$). Since the vector-fields are smooth and bounded, we have, by Theorem \ref{thm:RDEex},
\begin{equation}\label{eq:estV}
\|\mathbf{V}\|_{p-var ;[a,t]} \lesssim_{p,\gamma,\sigma}\|\mathbf{Z}\|_{p-var ;[a,t]}.
\end{equation}

\begin{remark}
Note here that when $x$ is smooth, the definition of $S_{\lfloor p \rfloor}(x)$ only involved increments of $x$. Remark also that when $W^\eps$ is a sequence of smooth paths such that $S_{\lfloor p \rfloor}(W^\eps)$ converge to $\bW$ in the rough path topology, then by Inequality \eqref{eq:EstRDE},
\[\sup_{ x\in\mathbb{R}^d } \| S_{\lfloor p \rfloor}(X^{\eps,s,x})\|_{p-var ;[s,t]}^p \lesssim   \left(\|\sigma\|_ {\cC^{\gamma-1}_b}\|S_{\lfloor p \rfloor}(W^\eps)\|_{p-var ;[s,t]} \vee\|\sigma\|_ {\cC^{\gamma-1}_b}^p\|S_{\lfloor p \rfloor}(W^\eps)\|_{p-var ;[s,t]}^p\right)\]
and one has, by letting $\eps \to 0$,
\[\sup_{ x\in\mathbb{R}^d } 
\| \bX^{s,x}
\|_{p-var ;[s,t]}^p 
\lesssim   
\left(
	\|\sigma\|_ {\cC^{\gamma-1}_b}
	\|\bW\|_{p-var ;[s,t]} 
	\vee
	\|\sigma\|_ {\cC^{\gamma-1}_b}^p
	\|\bW\|_{p-var ;[s,t]}^p
\right)\]
\end{remark}

\blue{Now, let $V_1 = (X,\mathbf{J},\mathbf{J}^{-1})$ be a solution of the RDE
\eqref{eq:rde_basic}-\eqref{eq:Jacobian}-\eqref{eq:invJacobian} driven by $\bV$. This RDE can be solved and is a geometric $p$-rough path (denoted $\mathbf{V}_1$) that satisfies \cite{cass2013integrability,bailleulNonexplosionCriteriaRough2020}
\begin{equation}\label{eq:integrability-condition-jacobian}
\|\mathbf{V}_1\|_{p-var ;[s,t]} \leq \kappa_1 \|\mathbf{V}\|_{p-var ;[s,t]}e^{\kappa_2 N_{\alpha,[s,t],p}(\mathbf{V})},
\end{equation}
where $\kappa_1,\kappa_2>0$ depend on $\sigma$ and $N_{\alpha,[s,t],p}(\mathbf{V})$ as finite moments of any order (see \cite{friz2013integrability,cass2013integrability} for details).} 

\blue{
\begin{remark}
By using the same trick adding $(\partial_i \partial_j \bJ^{-1})_{1\le i,j,\le d}$, one also have the of bound Equation \eqref{eq:integrability-condition-jacobian} for the second derivative of the flow.
Using the semiflow property, one gets for a given $a\le s \le t$,
\[\bJ^{x,s}_t = \bJ^{X^{x,a}_s,a}_t \bJ^{x,a}_s \]
This leads to the following estimates : 
\begin{equation}\label{eq:moments-initial-condition-flow}\|\bJ^{x,\cdot}_t\|_{p-var;[a,t]} \le \kappa_1 \|\bV\|_{p-var;[a,t]}e^{\kappa_2 N_{\alpha,[a,t]}(\bV)},
\end{equation}
leading to finite moment estimates from the norm of the Jacobian in the initial time. 
\end{remark}
}
Now, we can see that the integrand in \eqref{eq:XIdef1} has a polynomial growth (as well as its derivatives) with respect to $\bV_1$. Thus, we have the following estimate from \cite[Equation (52)]{gess2020density} (see also \cite{bailleulNonexplosionCriteriaRough2020} Subsection 4.3)
\begin{equation*}
\left| \mathcal{X}_{1,t}(\bW,\mathbf{B})\right|\land \left\| \mathcal{X}_{1}(\bW,\mathbf{B})\right\|_{p-var ;[a,t]}\lesssim \left( 1+\|\mathbf{V}_1\|_{p-var ;[a,t]}\right)^r\|\mathbf{V}_1\|_{p-var ;[a,t]}.
\end{equation*}
for some $r>0$. This yields, in the end,
\begin{equation*}
\left| \mathcal{X}_{1,t}(\bW,\mathbf{B})\right|\lesssim \left( 1+\|\mathbf{Z}\|_{p-var ;[a,t]}e^{\kappa_2 N_{\alpha,[a,t],p}(\mathbf{V})}\right)^r\|\mathbf{Z}\|_{p-var ;[a,t]}e^{\kappa_2 N_{\alpha,[a,t],p}(\mathbf{V})},
\end{equation*}
which leads to the estimate
\begin{equation*}
\EE\left[\sup_{ x\in\mathbb{R}^d }\left| \mathcal{X}_{1,t}(\bW,\mathbf{B})\right|^q\right] \lesssim \kappa_{a,t}.
\end{equation*}
Finally, if we assume that we have, for any $1\leq\ell\leq m-1$,
\begin{equation*}
\left\| \mathcal{X}_{\ell}(\bW,\mathbf{B})\right\|_{p-var ;[a,t]} \leq \left( 1+\|\mathbf{Z}\|_{p-var ;[a,t]}e^{\kappa_2 N_{\alpha,[a,t],p}(\mathbf{V})}\right)^r\|\mathbf{Z}\|_{p-var ;[a,t]}e^{\kappa_2 N_{\alpha,[a,t],p}(\mathbf{V})},
\end{equation*}
for some $r>0$, then one can consider the function $V_{m-1} = (X,\mathbf{J},\mathbf{J}^{-1},\mathcal{X}_1,\ldots,\mathcal{X}_{m-1})$ as a geometric $p$-rough path driven by $V$ denoted by $\mathbf{V}_{m-1}$ and, since the integrand in \eqref{eq:XIdef2} as a polynomial growth (as well as its derivatives), we deduce the estimate
\begin{equation*}
\left|\mathcal{X}_{m,t}(\bW,\mathbf{B})\right|\land \left\| \mathcal{X}_{m}(\bW,\mathbf{B})\right\|_{p-var ;[a,t]}\lesssim \left( 1+\|\mathbf{V}_{m-1}\|_{p-var ;[a,t]}\right)^r\|\mathbf{V}_{m-1}\|_{p-var ;[a,t]}.
\end{equation*}
In particular, following the same arguments as for $\mathcal{X}_1$, we obtain that
\begin{equation*}
\EE\left[\sup_{ x\in\mathbb{R}^d }\left| \mathcal{X}_{m,t}(\bW,\mathbf{B})\right|^q\right] \lesssim \kappa_{a,t}.
\end{equation*}
We proceed then to obtain estimates on $(\mathcal{Y}_{\ell})_{1\leq \ell\leq m}$ by considering iteratively $V_{\ell+m} = (X,\mathbf{J},\mathbf{J}^{-1},$ $\mathcal{X}_1,\ldots,\mathcal{X}_{m},\mathcal{Y}_{1},\ldots,\mathcal{Y}_{\ell})$ taken as a geometric $p$-rough path driven by $V$ and denoted $\mathbf{V}_{\ell+m}$. By considering \eqref{eq:YIdef1}-\eqref{eq:YIdef2}, we deduce that, for any $0\leq\ell\leq m-1$,
\begin{equation*}
\left|\mathcal{Y}_{\ell+1,t}(\bW,\mathbf{B})\right|\land \left\| \mathcal{Y}_{\ell+1}(\bW,\mathbf{B})\right\|_{p-var ;[a,t]}\lesssim \left( 1+\|\mathbf{V}_{\ell+m}\|_{p-var ;[a,t]}\right)^r\|\mathbf{V}_{\ell+m}\|_{p-var ;[a,t]},
\end{equation*}
which leads to the estimate
\begin{equation*}
\EE\left[\sup_{ x\in\mathbb{R}^d }\left| \mathcal{Y}_{m,t}(\bW,\mathbf{B})\right|^q\right] \lesssim \kappa_{a,t}.
\end{equation*}

For the following result, we need a notation. Let $F:x\mapsto F(x) \in \DD^{m,p}$ Let us define 
\[\|F\|_{\star,\DD^{m,p}}= \left( \EE[\sup_{x\in \RR^d}|F(x)|^p] + \sum_{k = 1}^m \EE[\sup_{x\in\RR^d}\|D^k F(x)\|_{\cH^{\otimes k}}^p]\right)^{1/p}.\]
What we just have proved is that 

\begin{proposition}\label{prop:XMalliavinEst}
We have the estimates, for any $m\geq 1$ and $p\geq 2$,
\begin{equation*}
\|D_{[a,b]}X_t^{x,a}\|_{\star,\DD^{m,p}(\cH([a,b]))}\lesssim \kappa_{a,t}\quad\mbox{and}\quad\|(\mathbf{J}_t^{x,a})^{-1}\|_{\star,\DD^{m,p}(\cH([a,b]))}\lesssim \kappa_{a,t}.
\end{equation*}
\end{proposition}

\begin{remark}\label{rem:malliavin_higher}
Let us remark that if  we denote  $\bJ^0 = \bJ$ and $\bJ^m = D_{x} \bJ^{m-1}$, where $D_x$ stands for the derivative with respect to $x$ then we have the  following expression~:
\[ \bJ^{m}_{t} =  \sum_{\substack{\ell \in\{ 2,\cdots,m\}\\ \alpha \in \{1,\cdots,m-1\}^{\ell} \\ |\alpha| = m}} c_{\ell,\alpha}\sum_{k=1}^d D_x^{\ell}\sigma_k(X^{a,x}_r) \bJ^{\alpha_1}_r \cdots \bJ^{\alpha_\ell}_r  \dd \bW_r + \sum_{k=1}^{d} \int_a^t \nabla\sigma_k(X^{a,x}_r) \bJ^{m}_{r} \dd \bW_r,\]
for some constants $c_{\ell,\alpha}>0$.
By using exactly the same strategy as before, one can also prove that 
\[\|\bJ^m_t\|_{\star,\DD^{k,q}(\cH([a,b]))} \lesssim \kappa_{a,t},\]
by using again Equation \eqref{eq:YIdef2} in this situation.
\end{remark}

We now turn to the covariance matrix $\gamma_{X,[a,b]}$ which is expressed, following Definition \ref{defi:covmat}, as
\begin{equation}\label{eq:covXexp}
\gamma_{X_t^{x,a},[a,b]} = \left( \langle D_{[a,b]}X^{x,a,i}_t,D_{[a,b]}X^{x,a,j}_t \rangle_{\cH([a,b])}\right)_{(i,j)\in\{1,\ldots,d\}^2}.
\end{equation}
We remark that $\gamma_{X_t,[a,b]}$ is a symmetric definite positive matrix. Furthermore, thanks to \eqref{eq:calHscalarprod}, \eqref{eq:MalliavinD2Jacobian} and \eqref{eq:flowJacobian} , we deduce the expression
\begin{align}
\gamma_{X_t^{x,a},[a,b]} &= \langle D_{[a,b]}X^{x,a}_t,(D_{[a,b]}X^{x,a}_t)^* \rangle_{\cH([a,b])} = \langle D_{[a,b]}X^{x,a}_t,(D_{[a,b]}X^{x,a}_t)^* \rangle_{\cH([a,t])} 
\\ &=  \int_a^t \int_a^t \mathbf{J}_{t}^{x,s_1}\sigma(X_{s_1}^{x,a})\sigma(X_{s_2}^{x,a})^* \left(\mathbf{J}_{t}^{x,s_2}\right)^* \dd R_W(s_1,s_2).\label{eq:expreCovMat}
\end{align}

\begin{proposition}\label{prop:malliavin_matrix}
Under the uniform ellipticity condition
\begin{equation*}
\|\sigma(x)z\|^2\geq \varsigma \|z\|^2,\quad\forall z,x\in\RR^d,
\end{equation*}
for some constant $\varsigma>0$, we have the estimate, for any $m\geq 1$ and $p\geq 2$,
\begin{equation*}
\left\|\gamma_{X_t^{x,a},[a,b]}^{-1}  \right\|_{\star,\DD^{m,p}(\cH([a,b]))}\lesssim \frac{\max(1,\varrho_{a,t})^m}{\sigma_{a,t}^2},
\end{equation*}
where $\varrho_{a,t} : = \kappa_{a,t}/\sigma_{a,t}$.
\end{proposition}

\begin{proof}
We follow the arguments from \cite[Proposition 3.9]{gess2020density}. We start with the estimate in the case $m=0$. Since $\gamma_{X_t^{x,a},[a,b]}$ is symmetric definite positive and all the matrix norms are equivalent, an upper bound on $\|\gamma_{X_t^{x,a},[a,b]}^{-1} \|$ can be deduced by estimating its lowest eigenvalue. In order to do so,  we will identify a positive random variable $\varrho_{a,t}$ admitting negative moments of any order such that, for any $z\in\RR^d\backslash\{0\}$,
\begin{equation*}
z^*\, \gamma_{X_t^{x,a},[a,b]}\,z \geq \varrho_{a,t} \|z\|^2.
\end{equation*}
Thanks to \eqref{eq:expreCovMat} and by denoting $\mathsf{f}_s^{z,x,a,t} = \sigma(X_{s}^{x,a})^* \left(\mathbf{J}_{t}^{x,s}\right)^* z $, we obtain
\begin{align*}
z^*\, \gamma_{X_t^{x,a},[a,b]}\,z &= \int_a^t \int_a^t z^* \mathbf{J}_{t}^{x,s_1}\sigma(X_{s_1}^{x,a})\sigma(X_{s_2}^{x,a})^* \left(\mathbf{J}_{t}^{x,s_2}\right)^*z \dd R_W(s_1,s_2)
\\ &=\int_a^t \int_a^t\left(\mathsf{f}_{s_1}^{z,x,a,t}\right)^* \mathsf{f}_{s_2}^{z,x,a,t} \dd R_W(s_1,s_2) = \|\mathsf{f}^{z,x,a,t}\|^2_{\cH([a,t])}.
\end{align*}
By Corollary \ref{cor:interpIneq}, we furthermore deduce that \blue{for $q>1$ with $\frac1q + \frac1\rho >1$}
\blue{\begin{equation*}
\|\mathsf{f}^{z,x,a,t}\|^2_{\cH([a,t])} 
\geq 
\frac{\sigma_{a,t}^2 \|\mathsf{f}^{z,x,a,t}\|_{\infty;[a,t]}^2}4
\min
	\left(
		1,
		\frac
			{4(c_W/2)^{(2+q \alpha)/2}\|\mathsf{f}^{z,x,a,t}\|_{\infty;[a,t]}^{q\alpha}}
			{\sigma_{a,t}^2|\mathsf{f}^{z,x,a,t}|_{\cC^{q-var}([a,t])}^{q\alpha}} 
	\right).
\end{equation*}}
Also, thanks to the uniform ellipticity condition, we obtain
\begin{equation*}
\|\mathsf{f}^{z,x,a,t}_s\|^2  \geq \varsigma \|\left(\mathbf{J}_{t}^{x,s}\right)^{-1}z\|^2\geq \varsigma \|\mathbf{J}_{t}^{x,s}\|^{-2} \|z\|^2,
\end{equation*}
which, since $\sup_{s\in[a,t]} \|\mathbf{J}_{t}^{x,s}\|^{-1}\geq \|\mathbf{J}_{t}^{x,t}\|^{-1} = \|\mathrm{Id} \|^{-1} = 1$, writes as
\begin{equation*}
\|\mathsf{f}^{z,x,a,t}\|_{\infty;[a,t]}  \geq \sqrt{\varsigma} \|z\|.
\end{equation*}
Moreover, we have
\blue{
\begin{equation*}
|\mathsf{f}^{z,x,a,t}|_{\cC^{q-var}([a,t])} \leq | \sigma(X_{\cdot}^{x,a})^* \left(\mathbf{J}_{t}^{x,\cdot}\right)^*|_{\cC^{q-var}([a,t])} \|z\| = | \mathbf{J}_{t}^{x,\cdot}\sigma(X_{\cdot}^{x,a})|_{\cC^{q-var}([a,t])} \|z\|.
\end{equation*}
}
This leads to the estimate
\blue{
\begin{equation*}
z^*\, \gamma_{X_t^{x,a},[a,b]}\,z \geq \frac{\sigma_{a,t}^2 \varsigma }4\min\left(1,\frac{4\varsigma^{\alpha/\gamma} (c_W/2)^{(2\gamma+\alpha)/2\gamma}}{\sigma_{a,t}^2|\mathbf{J}_{t}^{x,\cdot}\sigma(X_{\cdot}^{x,a})|_{\cC^{q-var}([a,t])}^{q\alpha}} \right) \|z\|^2,
\end{equation*}
and we identify
\begin{equation*}
 \varrho_{a,t} =  \frac{\sigma_{a,t}^2 \varsigma }4\inf_{x\in\RR^d}\left(\min\left(1,\frac{4\varsigma^{\alpha/\gamma} (c_W/2)^{(2\gamma+\alpha)/2\gamma}}{\sigma_{a,t}^2|\mathbf{J}_{t}^{x,\cdot}\sigma(X_{\cdot}^{x,a})|_{\cC^{q(var}([a,t])}^{q\alpha}} \right)\right).
\end{equation*}
We now have to prove that $\EE[ \sup_{x\in\RR^d}\varrho_{a,t}^{-p}]<+\infty$. We can see that
\begin{equation*}
 \varrho_{a,t}^{-1} \leq \frac4{\sigma_{a,t}^2 \varsigma }\sup_{x\in\RR^d}\left(\max\left(1,\frac{\sigma_{a,t}^2|\mathbf{J}_{t}^{x,\cdot}\sigma(X_{\cdot}^{x,a})|_{\cC^{q-var}([a,t])}^{q \alpha}}{4\varsigma^{q\alpha} (c_W/2)^{(2+q\alpha)/2}} \right)\right),
\end{equation*}
and, thus, we essentially have to see that $\sup_{x\in\RR^d}|\mathbf{J}_{t}^{x,\cdot}\sigma(X_{\cdot}^{x,a})|_{\cC^{q-var}([a,t])}$ admits moments of any order. We have
\begin{equation*}
\sup_{x\in\RR^d}|\mathbf{J}_{t}^{x,\cdot}\sigma(X_{\cdot}^{x,a})|_{\cC^{q-var}([a,t])} \leq \|\sigma\|_{\cC^0_b}\sup_{x\in\RR^d}|\mathbf{J}_{t}^{x,\cdot}|_{\cC^{q-var}([a,t])} + \sup_{s\in[a,t],\,x\in\RR^d}|\mathbf{J}_{t}^{x,s}| \sup_{x\in\RR^d}|\sigma(X_{\cdot}^{x,a})|_{\cC^{q-var}([a,t])}.
\end{equation*}
From here and Equation \eqref{eq:moments-initial-condition-flow}, it is rather direct to prove that, for any $p\in [2,+\infty)$,
\begin{equation*}
\EE
	\left[
		\left(
			\sup_{x\in\RR^d}
				|\sigma(X_{\cdot}^{x,a})|_{\cC^{q-var}([a,t])}\right)^p\right] \leq \|\sigma\|_{\cC^1_b}\EE\left[\left( \sup_{x\in\RR^d}|X_{\cdot}^{x,a}|_{\cC^{q-var}([a,t])}
		\right)^p
	\right]
<+\infty,
\end{equation*}}
thanks to \eqref{eq:EstRDEHol} and Proposition \ref{prop:GaussFern}. We use Equation \eqref{eq:bndJacMom} to handle the Jacobian. In the end, we obtain the bound
\begin{equation*}
\EE\left[ \varrho_{a,t}^{-p}\right] \lesssim \sigma_{a,t}^{-2p},
\end{equation*}
which provides the desired result in the case $m = 0$. In the case $m\geq 1$, we use Lemma \ref{lem:diffgamma} to deduce that, for any $(i,j)\in\{1,\ldots,d\}^2$,
\begin{equation*}
D_{[a,b]}(\gamma_{X_t^{x,a},[a,b]}^{-1})_{i,j} = -\sum_{\ell_1,\ell_2 = 1}^d (\gamma^{-1}_{X_t^{x,a},[a,b]})_{i,\ell_1}\,D_{[a,b]}(\gamma_{X_t^{x,a},[a,b]})_{\ell_1,\ell_2}\,(\gamma^{-1}_{X_t^{x,a},[a,b]})_{\ell_2,j},
\end{equation*}
which yields, thanks to Leibniz's rule, for any $m\geq1$,
\begin{multline}\label{eq:LeibnGamma}
D^m_{[a,b]}(\gamma^{-1}_{X_t^{x,a},[a,b]})_{i,j}
  =-\sum_{\ell_1,\ell_2 = 1}^d\sum_{\substack{k_1,k_2,k_3\in\NN\\k_1+k_2+k_3 = m-1}}\begin{pmatrix}m-1\\ k_1,k_2,k_3\end{pmatrix} \\ D_{[a,b]}^{k_1}(\gamma^{-1}_{X_t^{x,a},[a,b]})_{i,\ell_1}\,D_{[a,b]}^{k_2+1}(\gamma_{X_t^{x,a},[a,b]})_{\ell_1,\ell_2}\,D_{[a,b]}^{k_3}(\gamma^{-1}_{X_t^{x,a},[a,b]})_{\ell_2,j}.
\end{multline}
From here, we can see that, thanks to \eqref{eq:covXexp},
\begin{equation*}
D_{[a,b]}^{k_2+1}(\gamma_{X_t^{x,a},[a,b]})_{\ell_1,\ell_2} = \sum_{k = 0}^{k_2+1}\begin{pmatrix}k_2+1\\ k\end{pmatrix}\langle D_{[a,b]}^{k+1}X^{x,a}_t,(D_{[a,b]}^{k_2+1-k}X^{x,a}_t)^* \rangle_{\cH([a,b])},
\end{equation*}
which provides the estimate, thanks to Proposition \ref{prop:XMalliavinEst} and H\"older's inequality, for any $p\in[2,+\infty)$,
\begin{multline}
\EE\left[\sup_{x\in\RR^d}\|D_{[a,b]}^{k_2+1}(\gamma_{X_t^{x,a},[a,b]})_{\ell_1,\ell_2}\|_{\cH([a,b])^{\otimes(k_2+1)}}^p \right]
\\ \leq \sum_{k = 0}^{k_2+1} \begin{pmatrix}k_2+1\\ k\end{pmatrix} \EE\left[\sup_{x\in\RR^d}\|D_{[a,b]}^{k+1}X^{x,a}_t\|_{\cH([a,b])^{\otimes(k+1)}}^{p_1}\right]^{1/p_1}  \\ \times \EE\left[\sup_{x\in\RR^d}\|D_{[a,b]}^{k_2+1-k}X^{x,a}_t\|_{\cH([a,b])^{\otimes(k_2+1-k)}}^{p_2}\right]^{1/p_2}
 \lesssim \kappa_{a,t}^2,\label{eq:estDGamma}
\end{multline}
where $1/p = 1/p_1+1/p_2$.
Proceeding by induction, we assume that, for any $p\in[2,+\infty)$,
\begin{equation*}
\EE\left[\sup_{x\in\RR^d}\|D_{[a,b]}^{m-1}(\gamma^{-1}_{X_t^{x,a},[a,b]})\|_{\cH^{\otimes (m-1)}([a,b])}^p\right]^{1/p} \lesssim \frac{\max(1,\varrho_{a,t})^{m-1}}{\sigma_{a,t}^2}.
\end{equation*}
From \eqref{eq:LeibnGamma}, \eqref{eq:estDGamma} and H\"older's inequality, we deduce that, for $1/p = 1/p_1+1/p_2+1/p_3$,
\begin{align*}
&\EE\left[\sup_{x\in\RR^d}\|D_{[a,b]}^{m}(\gamma^{-1}_{X_t^{x,a},[a,b]})\|_{\cH^{\otimes m}([a,b])}^p\right]^{1/p}
\\ &\lesssim \sum_{\substack{k_1,k_2,k_3\in\NN\\k_1+k_2+k_3 = m-1}}\begin{pmatrix}m-1\\ k_1,k_2,k_3\end{pmatrix} \EE\left[\sup_{x\in\RR^d}\|D_{[a,b]}^{k_1}(\gamma^{-1}_{X_t^{x,a},[a,b]})\|_{\cH^{\otimes k_1}([a,b])}^{p_1}\right]^{1/p_1}
\\ & \hspace{2em}\times \EE\left[\sup_{x\in\RR^d}\|D_{[a,b]}^{k_2+1}\gamma_{X_t^{x,a},[a,b]}\|_{\cH^{\otimes (k_2+1)}([a,b])}^{p_2}\right]^{1/p_2}\EE\left[\sup_{x\in\RR^d}\|D_{[a,b]}^{k_3}(\gamma^{-1}_{X_t^{x,a},[a,b]})\|_{\cH^{\otimes k_3}([a,b])}^{p_3}\right]^{1/p_3},
\\ &\lesssim \sum_{\substack{k_1,k_2,k_3\in\NN\\k_1+k_2+k_3 = m-1}}\begin{pmatrix}m-1\\ k_1,k_2,k_3\end{pmatrix}\frac{\max(1,\varrho_{a,t})^{k_1+k_3}}{\sigma_{a,t}^2}\lesssim \frac{\max(1,\varrho_{a,t})^{m}}{\sigma_{a,t}^2},
\end{align*}
which concludes our proof.
\end{proof}

\section{Proof of the main theorem}\label{section:main}

\begin{theorem}\label{thm:MAIN000}
Let $W$ be a Gaussian process which satisfies Assumption \ref{asm:1rhovar}, \ref{asm:correW} and \ref{asm:nondet}. Let $2<p<4$ such that $W$ can be lifted into a geometric $p$-rough path (from assumption \eqref{asm:1rhovar}). Furthermore let $\alpha$ its index of non-determinism from Assumption \ref{asm:nondet}.

Let $\sigma\in C^\infty_{b}(\RR^d;(\RR^d)^{\otimes 2})$ be such there exists $\varsigma>0$ that for all $x,z\in\RR^d$
\[ |\sigma(x) z|^2 \ge \varsigma |z|^2.\]

Let $b\in \cC^{\kappa}(\RR^d)$ with \blue{$\kappa + \frac1\alpha>\frac32$}. Then, almost surely, there exists a solution flow to the equation 
\[\dd X_t = b(X_t) \dd t + \sigma(X_t)  \dd \bW_t, \quad X_0 = x,\quad t\in[0,T].\]
Furthermore, this flow is locally Lipschitz continuous. 
\end{theorem}

\begin{proof}
As seen in Section \ref{section:rough_paths} and especially in Theorem \ref{theorem:main2}, one only has to check that there exists $\nu>\frac12$, $\kappa>2-\nu$ and a weight $w_0$ such that almost surely
\[\cT^{D\varphi,\varphi} b \in \cC^{\nu}_T \cC^{\kappa}_{w_0}.\]
Furthermore, thanks to Section \ref{section:kolmo} and especially Corollary \ref{corollary:averaged_holder}, it is enough to check that for all $i,j\in\{1,\cdots,d\}$, we have the following couples 
\begin{multline*}
    (\bJ^{-1},X),(\partial_i \bJ^{-1} ,X), (\partial_i X \bJ^{-1},X), (\partial_j \partial_i X \bJ^{-1} ,X), (\partial_j \partial_i \bJ^{-1},X),\\ (\partial_j \partial_i X \bJ^{-1},X), (\partial_j X \partial_i X \bJ^{-1},X) \end{multline*}
which satisfy Assumption \ref{asm:phipsi} for some $H\in(0,1)$. 
Remark that thanks to the flow properties, for any $s\in[0,T)$ and $x\in\mathbb{R}^d$, one can always consider that in the expression of $\cT^{\phi,\varphi}$, one can always take
\begin{multline*}
\phi(\cdot,x) \in \{  (\bJ^{x,s}_{\cdot})^{-1}, \partial_i (\bJ^{x,s}_{\cdot})^{-1}, \partial_i X^{s,x}_{\cdot} (\bJ^{x,s}_{\cdot})^{-1},\\ 
\partial_i\partial_j(\bJ^{x,s}_{\cdot})^{-1}, \partial_i X^{x,s}_{\cdot} \partial_j (\bJ^{x,s}_{\cdot})^{-1}, \partial_i \partial_j X^{x,s}_{\cdot}(\bJ^{x,s}_{\cdot})^{-1}, \partial_i X^{s,x}_{\cdot} \partial_j X^{s,x}_{\cdot} (\bJ^{x,s}_{\cdot})^{-1}\}
\end{multline*}
and
\[\varphi(\cdot,x) = X^{s,x}_{\cdot}.\]
Thanks to Proposition \ref{prop:integration_by_part}, for all $t \in [s,T]$, we can also take $F = X_t^{s,x}$ and $G=\phi(t,x)$ as given above. Notice that $F, G \in \cF_t$. Hence, we have, for all $\beta\in\NN^d$ and all $f\in\cS$,
\[\EE[\partial^{\beta}f(F) G |\cF_s] =
\EE[f(F) H_{\beta,[s,t]}(F,G) |\cF_{[0,s]}],\]
and, moreover,
\begin{align*}
|\EE[\partial^{\beta}f(F) G |\cF_s]| 
 =&  
|\EE[f(F) H_{\beta,[s,t]}(F,G) |\cF_{[0,s]}]| 
\le  \EE[f(S)^2|\cF_s]^{\frac12}\EE[|H_{\beta,[s,t]}(F,G)|^2|\cF_s]^{\frac12} \\
\le & \|f\|_\infty \sup_{x\in\RR^d} \EE[|H_{\beta,[s,t]}(F,G)|^2|\cF_s]^{\frac12} 
\end{align*}

Note that for any $q\ge 2$, thanks to Lemma \ref{lemma:boundeH},  we obtain the bound
\begin{align*}
\|\sup_{x\in\RR^d}\EE[H_{\beta,[s,t]}(F,G)^2|\cF_s]^\frac{1}{2}\|_{L^q(\Omega)}
\leq &
\|\sup_{x\in\RR^d}|H_{\beta,[s,t]}(F,G)| \|_{L^{\frac{q}{2}}(\Omega)}^{\frac12}
\\
  \leq  \|G\|_{\star,\DD_{[s,t]}^{|\beta|,r}}^{\frac12}& \left(\prod_{\ell = 1}^{|\beta|} \left\|\sum_{j =1}^d  (\gamma_{X^{s,x}_t,[s,t]}^{-1})_{\beta_\ell,j}D_{[s,t]}(X^{s,x}_t)^j  \right\|_{\star,\DD_{[s,t]}^{|\beta|-\ell+1,r_\ell}}\right)^{\frac12},
\end{align*}
 with $\frac{2}{q}=\frac{1}{r} + \sum\frac{1}{r_i}$.

 Furthermore, thanks to Proposition \ref{prop:XMalliavinEst} and Remark \ref{rem:malliavin_higher}, we know that, for every $\beta\in\NN^d$ and every $r\ge1$,
 \[\||G|\|_{\star,\DD_{[s,t]}^{|\beta|,r}} \lesssim \kappa_{s,t} = \|R_W\|_{(1,\rho)-var ;[a,b]^2}^{1/2} \lesssim_T 1.\]
It follows from Propositions \ref{prop:malliavin_matrix} \ref{prop:XMalliavinEst} and Hölder's inequality that, for all $\ell \le |\beta|$ and all $r\ge 2$,
 \[\|\gamma^{-1}_{X^{x,s}_t,[s,t]} (D_{[s,t]}(X^{s,x}_t)^j \|_{\star,\DD_{[s,t]}^{|\beta|-\ell + 1,r}} \lesssim \frac{1}{\sigma_{s,t}^{2}},\]
 and finally
\[ \left\|\sup_{x\in\RR^d}\EE[H_{\beta,[s,t]}(F,G)^2|\cF_s]^\frac{1}{2}\right\|_{L^q(\Omega)} \lesssim \sigma_{s,t}^{-|\beta|}.\]
Since $W$ is centered and has the local non-determinism property, we know that 
\[\sigma_{s,t}^2 = \EE[(W_t-W_s)^2] = \EE\left[\EE[(W_t-W_s)^2|\cF_{[0,s]}\vee \cF_{{t,T}}]\right] \gtrsim |t-s|^{\alpha},\]
where $\alpha$ is the index of non-determinism. Altogether, this yields
\[|\EE[\partial^{\beta} f(F) G | \cF_s]| \lesssim \|f\|_{\infty}|t-s|^{-\frac{\alpha}{2} |\beta|} G_{\infty,s},\]
where 
\[G_{\infty,s} = (t-s)^{\frac{\alpha}{2} |\beta|} \sup_{x\in\RR^d} \EE[H_{\beta,[s,t]}(F,G)^2|\cF_s]^{\frac12}\]
enjoys finite $L^q(\Omega)$ moments which are bounded uniformly in $(s,t)\in\Delta^2_T$.

Hence, by using Corollary \ref{corollary:averaged_holder}, for all  $\kappa \in \RR $ such that $0<\kappa + \frac1\alpha \le 2$, there exists $\eps>0$ small enough such that for for all \blue{$q\ge \frac8{3\eps} \vee d$} and for $w_0 (x) = (1+|x|)$,
 there exists a positive random variable $K(b) = K(b,\sigma,\bW)$ with $K(b) \in L^q(\Omega)$  such that 
\blue{
\[\|\cT^{D\varphi,\varphi}b\|_{
\cC^{\frac12+\frac{\eps}{8}}_T \cC^{\kappa + \frac1\alpha - \frac{3\eps}8}_{1+|\cdot|}} \le K(b) \|b\|_{\cC^{\kappa}}. \]

Therefore, for $\kappa + \frac{1}{\alpha}-\frac{3\eps}8+ \frac12 + \frac{\eps}{8}>2$, namely for 
\[\kappa + \frac{1}{\alpha} > \frac32 + \frac{\eps}{4},\]
thanks to Theorem \ref{theorem:main2}, there exists a unique solution, locally Lipschitz continuous} in the initial condition to Equation \eqref{eq:rde_with_drift}.

\end{proof}

Finally we give a (standard) example that satisfies the conditions which are required for uniqueness of the solutions. We refer to \cite[Section 4]{cass2015smoothness} and \cite[example 3.8]{friz2016jain} for the proofs of the following propositions.

\begin{proposition}\label{prop:fBm}
Let $H\in(0,1)$. Let $B,\tilde{B}$ be two independent $d$-dimensional Brownian motions. We define the $d$-dimensional fractional Brownian motion of Hurst parameter $H$ as the process defined for all $t\ge 0$ by
\[B^H_t = \frac{1}{c_H}\left(\int_0^t (t-r)^{H-\frac{1}{2}} \dd B_r + \int_0^{+\infty} (t+r)^{H-\frac{1}{2}} - r^{H-\frac{1}{2}} \dd \tilde{B}_r\right),\]
where 
\[c_H = \sqrt{\frac{1}{2H}+\int_0^{+\infty} \big((t+r)^{H-\frac{1}{2}} - r^{H-\frac{1}{2}}\big)^2  \dd r }\]
The fractional Brownian motion is a centered continuous Gaussian process with stationary increments and covariance 
\[R_{B^H}(s,t) = \frac{1}{2}(t^{2H} + s^{2H} - |t-s|^{2H}).\]

For $\frac14<H<\frac12$ it satisfies Assumption \ref{asm:1rhovar} with $\rho = \frac1{2H}$, Assumption \ref{asm:correW} and Assumption \ref{asm:nondet} with $\alpha = 2H$.
\end{proposition}

\begin{corollary}\label{cor:MAIN000}
Let $H\in\left(\frac{1}{4},\frac12\right)$. Let $B^H$ be a $d$-dimensional fractional Brownian motion and let $\bB^H$ be the geometric rough path above $B^H$.

Let $b \in \cC^{\kappa}$ with 
\[\kappa > 0\vee \left(\frac32-\frac1{2H}\right).\]

Let $\sigma \in C^{\infty}_b(\RR^d;(\RR^d)^{\otimes 2})$ which satisfies the strong ellipticity condition, namely there exists $c>0$ such that :
\[ |\sigma(x) z|^2 \ge c|z|^2,\quad x,z\in\RR^d.\]
Tthere exists $\cN = \cN(b) \in \cF$ such that $ \PP(\cN) = 0$ and for all $\omega \notin \cN$ and for all $x\in \RR^d$,
\[\dd x_t = b(x_t) \dd t + \sigma(x_t) \dd \bB^H_r,\quad x_0 = x,\quad t\in[0,T]\]
admits a unique solution. Furthermore, for all $R>0$ and all  all $b,\tilde{b}$ which satisfies the above conditions and all $x_0,\tilde{x}_0 \in \RR^d $ with $|x_0|,|\tilde{x}_0| \le R$, there exists an almost surely positive and finite random variable $K = K(\bB^H,\sigma,b,\tilde{b},R,T)$ such that 
\[\sup_{t\in[0,T]} |x_t - \tilde{x}_t| \le K (|x_0-\tilde{x_0}| - \|b - \tilde b\|_{\cC^{\kappa}}),\] 
where $x$ (respectively $\tilde{x}$) is the solution of the rough differential equation driven by the fractional Brownian motion with initial value $x_0$ (respectively $\tilde{x}_0$) and coefficient $b$ and $\sigma$ (respectively $\tilde b$ and $\sigma$).
\end{corollary}

\begin{proof}
One only has to use Theorem \ref{thm:MAIN000} and Proposition \ref{prop:fBm} and Theorem \ref{theorem:main2} to conclude.
\end{proof}

\begin{remark}
Let us remark that we where not able to prove the usual continuity of the Itô map solution with respect to the driven rough path. Indeed, whenever $B^{\eps,H}$ is a smooth (say piecewise linear) approximation of the fractional Brownian motion, one does not necessarily have Proposition \ref{prop:malliavin_matrix}. Similarly, the continuity with respect to the coefficient $\sigma$ is not clear, since it appears in a deeply non-linear fashion  in the averaged field.
\end{remark}

\appendix

\section{Besov spaces}\label{sec:Besov}

We first recall some of the Littlewood-Paley theory that is used in Besov spaces. There exist \blue{(see \cite[Proposition 2.10]{bahouriFourierAnalysisNonlinear2011})} $\psi,\phi$  two functions valued in $[0,1]$ such that
\begin{itemize}
\item[1)] $\psi,\phi \in \cC^{\infty}_0(\RR^d) $ with $\textrm{supp}(\psi)\subset B(0,4/3)$ and $\textrm{supp}(\phi)\subset \mathscr{A}$, where $\mathscr{A} = \{ \xi\in\RR^d; 3/4\leq |\xi|\leq 8/3\}$ is an annulus in $\RR^d$,
\item[2)] $\forall \xi\in\RR^d, \quad \psi(\xi) + \sum_{j = 0}^{\infty} \phi(2^{-j}\xi) = 1$,
\item[3)] $\textrm{supp}(\psi)\,\cap \, \textrm{supp}(\phi(2^{-j}\cdot) = \emptyset$ for any $j\geq 1$,
\item[4)] $\textrm{supp}(\phi(2^{-j}\cdot) \,\cap \, \textrm{supp}(\phi(2^{-k}\cdot) = \emptyset$ for any $j,k\geq 0$ such that $|j-k|\geq 2$.
\end{itemize}
We define the associated Littlewood-Paley blocks as
\begin{equation}
\Delta_{-1}u = \mathfrak{F}^{-1}\left(\psi\mathfrak{F}(u)\right) \quad \mbox{ and } \quad \Delta_j u = \mathfrak{F}^{-1}\left(\phi(2^{-j}\cdot) \mathfrak{F}(u)\right), \quad \forall j\geq 0,
\end{equation}
for any $u\in \cS'$. We recall the Bernstein Lemmas.
\begin{lemma}[{\cite[Lemma 2.1]{bahouriFourierAnalysisNonlinear2011}}]
Let $k\in\NN$. There exists a constant $C>0$ such that, for all $1\leq p\leq q\leq \infty$ and $f\in L^p(\RR^d)$, we have
\begin{itemize}
\item[1)] $\sup_{|\alpha| = k }\| \partial^{\alpha}\Delta_{j} f \|_{L^q(\RR^d)} \leq C^{k+1} 2^{k (j)_+ + d(1/q-1/p) }\| \Delta_{j} f\|_{L^p(\RR^d)}$, for all $j\in \NN\cup \{-1\}$,
\item[2)] $C^{-k-1} 2^{jk} \|\Delta_j f\|_{L^p(\RR^d)} \leq \sup_{|\alpha| = k }\| \partial^{\alpha}\Delta_j f \|_{L^p(\RR^d)} \leq C^{k+1} 2^{jk} \|\Delta_j f\|_{L^p(\RR^d)}$, for all $j\in\NN$.
\end{itemize}
\end{lemma}
The Besov spaces associated to the Littlewood-Paley blocks are defined as
\begin{align*}
B_{p,r}^{s} : = \left\{ u \in \cS'(\RR^d):\; \|u\|_{B_{p,r}^s} : = \left(\sum_{j = -1}^{+\infty} 2^{rs j}\|\Delta_{j} u\|_{L^p(\RR^d)}^r \right)^{1/r}<\infty\right\},
\end{align*}
where $s\in\RR$ and $p,r\in [1,\infty]$. Those spaces are Banach spaces continuously embedded in $\cS'$ and they are of type $p\wedge r\wedge 2$ (see Corollary \ref{cor:besovtype} below) when $p,r\in(1,+\infty)$.
\begin{remark}[{\cite[Propositions 2.71]{bahouriFourierAnalysisNonlinear2011}, \cite[Chapter 1 Section 1.2 and 2, Proposition 2]{runstSobolevSpacesFractional2011}}] {$ $ }
\begin{itemize}
\item[1)] We have the following continuous embeddings
\begin{itemize}
\item[a)] for $\tilde{s}<s$ or $\tilde{s} = s$ and $\tilde{r}\geq r$
\begin{equation*}
B_{p,r}^s \hookrightarrow B_{p,\tilde{r}}^{\tilde{s}},
\end{equation*}
\item[b)] for $\tilde{p}\geq p$,
\begin{equation*}
B_{p,r}^s \hookrightarrow B_{\tilde{p},r}^{s-d(1/p-1/\tilde{p})},
\end{equation*}
\item[c)] for $p<\infty$,
\begin{equation*}
B_{p,1}^{d/p} \hookrightarrow C^0_0,
\end{equation*}
where $C$ is the space of uniformly continuous bounded functions.
\end{itemize}
\item[2)] For $p=r=\infty$ and $s\in \RR^+/\NN$, we have that
\begin{equation*}
\|u\|_{B_{\infty,\infty}^s} \simeq \sum_{|\alpha|\leq [s]}\left(\|\partial^{\alpha}u\|_{L^{\infty}(\RR^d)} + \sup_{\substack{x,y\in\RR^d\\x\neq y}}\frac{|\partial^{\alpha}u(x)-\partial^{\alpha}u(y)|}{|x-y|^{s-[s]}}\right).
\end{equation*}
Thus, the spaces $B_{\infty,\infty}^s$ consist of $s$-Hölder functions and are called Hölder-Besov spaces.
\item[3)] For $p = r$, and $s>0$ not an integer, we have that (see \cite[Proposition p17]{runstSobolevSpacesFractional2011}
\begin{equation*}
\|u\|_{B_{p,p}^s} \simeq \|u\|_{W_p^s(\RR^d)},
\end{equation*}
where $W^p_s$ is the usual Sobolev space. 
\item[4)] Let $k\in\NN$. It follows from Bernstein's lemma that, for any $\alpha \in\NN^d$ such that $|\alpha|\leq k$,
\begin{equation*}
\|\partial^{\alpha}u\|_{B_{p,r}^s} \lesssim \|u\|_{B_{p,r}^{s+k}}.
\end{equation*}
\end{itemize}
\end{remark}
It turns out that the space $\cC^{\infty}_0(\RR^d)$ is dense in $B_{p,r}^s$ if and only if $p,r<\infty$. If we wish to build approximations of functions in the Hölder-Besov spaces, we can rely on following result.
\begin{lemma}[{\cite[Corollary 2.96]{bahouriFourierAnalysisNonlinear2011}}]\label{lem:A3}
Let $s\in\RR$ and $u\in B_{\infty,\infty}^s$. Then, the sequence $(u_n)_{n\geq0}$ defined by
\begin{equation*}
u_n = \sum_{q = -1}^{n-1} \Delta_q u,
\end{equation*}
converges to $u$ in $B_{\infty,\infty}^{\tilde{s}}$ for any $\tilde{s}<s$. Moreover, we have, $\forall n\in\NN$ and $\forall \ell\in\RR$,
\begin{equation*}
\|u_n\|_{ B_{\infty,\infty}^{\ell}} < +\infty.
\end{equation*}
\end{lemma}



We give a result concerning Fourier multiplier.
\begin{lemma}\label{lem:FourierMult}
Let $m:\RR^d\to\RR$ be a smooth function such that there, $\gamma\in\NN^d$ and $p\geq 0$,
\begin{equation}\label{asm:multiplier}
\partial^{\gamma}m(\xi) \leq |\xi|^{-|\gamma|-p},\; \xi\in\RR^d\backslash\{0\},
\end{equation}
and, for any $j$, define the function $\mathsf{m}_j$ given by
\begin{equation*}
\mathsf{m}_j = \mathfrak{F}^{-1}(m(\xi)\varrho(2^{-j}\xi)),
\end{equation*}
where $\varrho$ is a compactly supported function such that $\mathrm{supp}(\varrho)\subset2\cA$ and $\varrho_{|\cA} \equiv 1$. Then, we have
\begin{equation*}
\|\mathsf{m}_j \|_{L^1(\RR^d)}\lesssim_{\varrho,d} 2^{-jp}.
\end{equation*}
\end{lemma}
\begin{proof}
It turns out that, by a change of variable,
\begin{equation*}
\|\mathsf{m}_j\|_{L^1(\RR^d)} = \|\mathfrak{F}^{-1}(m(2^j\xi)\varrho(\xi))\|_{L^1(\RR^d)}.
\end{equation*}
Furthermore, we have, by integration by parts and Leibniz's formula,
\begin{align*}
\mathfrak{F}^{-1}(m(2^j\xi)\varrho(\xi))(x) &= (1+|x|^2)^{-d}(2\pi)^{-d/2}\int_{\RR^d}e^{i\xi\cdot x} (1-\Delta)^{d}(m(2^j\xi)\varrho(\xi)) \dd\xi
\\ &=(1+|x|^2)^{-d}(2\pi)^{-d/2}\sum_{\substack{\alpha,\beta\in\NN^d\\ \beta\leq \alpha,\; |\alpha|\leq 2d}}C_{\alpha,\beta} 2^{j|\alpha|}\int_{\RR^d}e^{i\xi\cdot x} \partial^{\alpha} m(2^j\xi)\partial^{\alpha-\beta}\varrho(\xi) \dd\xi.
\end{align*}
We can see that, by \eqref{asm:multiplier}
\begin{equation*}
\left|\int_{\RR^d}e^{i\xi\cdot x} \partial^{\alpha}m(2^j\xi)\partial^{\alpha-\beta}\varrho(\xi) \dd\xi\right| \lesssim 2^{-j|\alpha|-jp},
\end{equation*}
which leads to the estimate
\begin{equation*}
\|\mathsf{m}_j\|_{L^1(\RR^d)} \lesssim 2^{-jp} 
\end{equation*}
\end{proof}

\section{Young inequality}
Let us recall the following Young inequality for kernels.
\begin{theorem}[{\cite[Theorem 0.3.1]{soggeFourierIntegralsClassical2017}}]\label{thm:ineq_young}
Let $K:\RR^d\times \RR^d \mapsto \RR$ be a measurable function such that for some $p\in [1,+\infty]$,
\[\sup_{x\in\RR^d} \|K(x,\cdot)\|_{L^p(\RR^d)}  \le C \quad \text{ and } \quad \sup_{y\in\RR^d} \|K(\cdot,y)\|_{L^p(\RR^d)}  \le C.\]
Then $p,q,r\in[1,+\infty]$ with $\frac{1}{p}+\frac{1}{q} = 1+\frac{1}{r}$ and for all $f \in L^{q}(\RR^d)$, we have 
\[\left(\int_{\RR^d} \left|\int_{\RR^d} K(x,y) f(y)\dd y \right|^r \dd x\right)^{\frac1r} \le C \|f\|_{L^q(\RR^d)}.\]
\end{theorem}

\section{Burkolder Davis Gundy inequality in Lebesgue spaces}\label{appen:BDG}
In this section, we recall some standard results about martingales in Banach spaces, and the corresponding inequalities. We refer to Pisier \cite{pisierMartingalesBanachSpaces2016} chapter 4 and to Pinelis \cite{pinelisOptimumBoundsDistributions1994} and Veraar and coauthors \cite{hytonenAnalysisBanachSpaces2016,hytonenAnalysisBanachSpaces2018} for proofs.

The aim of this section is to give an almost self-contain proof of Burkolder-Davis-Gundy inequality in infinite dimensional Banach spaces, especially in $L^p(\RR^d)$ spaces. The idea of the proof is to use the Kahane-Khinchine inequality, and UMD properties via smoothness of spaces. 

\subsection{Discrete Fourier transform and Kahane-Khinchine inequality}

We recall some of standard feature about the Fourier transform on the hypercube. 

Let $\Omega_n = \{-1,1\}^n$. We endowed $\Omega_n$ with the uniform measure $\mu_n$, such that for a function $f : \Omega_n \mapsto \RR$ we have
\[\int_{\Omega_n} f \dd \mu_n = \EE[f(X)] = \frac{1}{2^n} \sum_{x\in \Omega_n} f(x)\]
where $X=(X_1,\cdots,X_n)$ is a vector of i.i.d. random variables such that $\PP(X_1 = 1) = \PP(X_1 = -1) = \frac12$. 

Remark that for any $Y\in L^2(\Omega_n,\cP(\Omega_n),\mu_n)$ a unique function $f : \Omega_n \mapsto \RR$ exists such that $Y = f(X)$. Hence $L^2(\Omega_n,\cP(\Omega_n),\mu_n)$ is a finite dimensional vector space of dimension $2^n$. For $Y_1 = f_1(X),Y_2=f_2(X) \in L^2(\Omega_n,\cP(\Omega_n),\mu_n)$ let us write the standard inner product as
\[
\langle Y_1, Y_2\rangle = \EE[f_1(X) f_2(X)].\] 

Furthermore let us define for $S\subset\{1,\cdots,n\}$, we define 
\[X_S = \prod_{l\in S} X_l,\quad X_{\emptyset} = 1.\]

Remark that $\langle X_S , X_S \rangle = 1$ and for $S'\neq S$ we have $\langle X_S , X_S' \rangle = 0$. Hence $X_S$ is a orthogonal basis of $L^2(\Omega_n,\cP(\Omega_n),\mu_n)$. Furthermore for $Y \in L^2(\Omega_n,\cP(\Omega_n),\mu_n)$ and $S\subset \{1,\cdots,n\}$, we define 
\[\hat{Y}(S) = \langle Y , X_S \rangle = \EE[Y X_S] ,\] 
such that
\[Y= \sum_{S\subset \{1,\cdots,n\}} \hat{Y}(S) X_S,\]
and
\[\EE[YZ]=\sum_{S\subset\{1,\cdots,n\}}\hat{Y}(S)\hat{Z}(S).\]
Note that we have the two following identities :
\[\EE[Y^2]= \sum_{S\subset\{1,\cdots,n\}} \hat{Y}(S)^2\]
and
\[\EE[Y] = \hat{Y}(\emptyset).\]

\begin{theorem}[Khinchine-Kahane Inequality, via Szarek's proof, {\cite[Theorem 5.20]{boucheronConcentrationInequalitiesNonasymptotic2013}}]\label{thm:khinchine}
Let $(B,|\cdot|_B)$ be a normed vector space. Then for any $p\in [1,+\infty]$, there exists a constant $C_p$ such that for any $n\ge 2$ and for any $b_1,\cdots,b_n \in B$, we have

\[\EE\left[\left|\sum_{i=1}^n b_i X_i\right|^p_B\right] \le C_p \EE\left[\left|\sum_{i=1}^n b_i X_i\right|_B\right],\]
where $X=(X_1,\cdots,X_n)$ is an iid random vector with $\PP(X_1 = 1) = \PP(X_1 = -1)= \frac12$.
\end{theorem}
\begin{proof}
 Let $k\ge 1$ such that $k\le p < k+1$. 
\begin{equation*}
\EE\left[\left|\sum_{i=1}^n b_i X_i\right|^p_B\right]^{\frac1p} 
	\le 
\EE\left[\left|\sum_{i=1}^n b_i X_i\right|^{k+1}_B\right]^{\frac1{k+1}}\\
	\le 
C_{k+1} \EE \left[\left|\sum_{i=1}^n b_i X_i\right|_B\right].
\end{equation*}
It is enough to prove the Theorem for $k+1$. 
Let us define 
\[Z = \left|\sum_{j=1}^n b_j X_j\right|_B,\]

Let us define for $i\in\{1,\cdots,n\}$,
\[X^{(i)}=(X^{(i)}_1,\cdots X^{(i)}_{n})=(X_1,\cdots, X_{i-1},-X_i,X_{i+1},\cdots,X_n),\]
such that $X^{(i)} \overset{(d)}{=} X$, and 
\[Z^{(i)} = \left| \sum_{j=1,j\neq i}^n b_j X^{(i)}_j\right|_B\]
and finally
\[\bar{Z} = \sum_{i=1}^n Z^{(i)}.\]

The core of the proof is then the following : we will give some lower and upper bounds for $\EE[Z^k \bar{Z}]$ with respect to $\EE[Z]$, $\EE[Z^{k+1}]$. To do so, we will give a direct lower bound to $\bar{Z}$. For the upper bound, we will dramatically use the Fourier transform.

Let us focus first on the lower bound. We have
\begin{equation*}
\bar{Z} 
= 
	\sum_{i=1}^n 
		\left|
			\sum_{j=1}^n X^{(i)}_j b_j 
		\right|_B \\
 \ge  
 	\left| 
 		\sum_{i=1}^n \sum_{j=1}^n X^{(i)}_j b_j 
 	\right|_B \\
  =  (n - 2) Z.
\end{equation*}
Since $Z^k$ is non negative, we have
\[(n - 2) \EE[Z^{k+1}] \le \EE[Z^k \bar{Z}].\]

Let us now focus on the upper bound. First note that if we define for $S\subset\{1,\cdots,n\}$ and $i\in\{1,\cdots,n\}$,
\[X_S^{(i)} = \prod_{l\in S} X^{(i)}_l,\quad X_{\emptyset}^{(i)} = 1,\]
then $(X_S^{(i)})_{S\subset\{1,\cdots,n\}}$ is a orthonormal basis of $L^2$ and furthermore note that $(Z^{(i)},X^{(i)}_S)$ has the same distribution as $(Z,X_S)$ for all $i$ and all $S$.
Note also that 
\[X_S^{(i)} = 
\begin{cases}
X_S & \text{if } i \notin S\\
-X_S & \text{if } i \in S
\end{cases}.\]
Hence,
\begin{align*}
\hat{\bar{Z}}(S) = & \sum_{i=1}^n \EE[Z^{(i)} X_S] \\
=& \sum_{i\notin S} \EE[Z^{(i)} X^{(i)}_S] - \sum_{i\in S}\EE[Z^{(i)} X^{(i)}_S] \\
=& (n-2|S|) \EE[Z X_S] \\
=& (n-2|S|) \hat{Z}(S).
\end{align*}
Not also that if there exists $m$ such that $|S|=2m+1$, since $X\overset{(d)}{=}-X$ we have
\[\hat{Z}(S)=\EE[Z X_S ] = \left[\left|\sum_{j=1}^{n} -X_j b_j \right|_{B} \prod_{l\in S}(-X_l)\right] = -\hat{Z}(S),\]
and 
\[\hat{Z}(S) = 0.\]
Hence we have
\begin{align*}
\EE[Z^k \bar{Z}] 
	=&
\sum_{S\subset\{1,\cdots,n\}} \widehat{Z^k}(S)\hat{\bar{Z}}(S)\\
	=&
\sum_{S\subset\{1,\cdots,n\}} (n-2|S|)\widehat{Z^{k}}(S)\hat{Z}(S) \\
	=&
n\widehat{Z^k}(\emptyset)\hat{Z}(\emptyset) + \sum_{\substack{S\subset\{1,\cdots,n\}\\|S|=2m\\S\neq \emptyset}} (n-2|S|) \widehat{Z^k}(S)\hat{Z}(S) \\
	\le & 
4\EE[Z^k]\EE[Z] + (n-4) \sum_{S\subset\{1,\cdots,n\}}  \widehat{Z^k}(S) \hat{Z}(S) \\
=&
4\EE[Z^k]\EE[Z] + (n-4) \EE[Z^{k+1}].
\end{align*}

Putting upper and lower bounds together, we get
\[\EE[Z^{k+1}] \le 2 \EE[Z^k]\EE[Z].\]
By a direct induction, we have 
\[\EE[Z^{k}]^{\frac{1}{k}} \le2^{1-\frac{1}{k}} \EE[Z],\]
which is the wanted bound with $C_k = 2^{1-\frac{1}{k}}$
\end{proof}

\subsection{UMD spaces}\label{appendix:UMD}

\begin{definition}[{\cite[Definition 8.7]{pisierMartingalesBanachSpaces2016}, \cite[Definition 4.2.1]{hytonenAnalysisBanachSpaces2016}}]
Let $(B,|\cdot|_B)$ ]
be a Banach space and $p\in(1,+\infty)$. We say that $B$ as the \emph{Unconditionnal Martingale Difference}-$p$ (\UMDp) property if there exists a constant
$K_p(B)>0$ such that for any filtered probability space $(\Omega,\cF, (\cF_n)_{n\ge 0},\PP)$ and any $B$ value martingale $(M_n)_{n\ge 0}$ with $p$ moments, any $N\ge 1$ and any $\eps_1,\cdots,\eps_N \in \{-1,1\}$,
\[\EE\left[\left|\sum_{k=0}^{N-1} \eps_k (M_{k+1}-M_k)\right|^p_B\right] \le K_p(B)^p \EE[|M_N-M_0|_B^p]\]

We say that $B$ as the UMD property if $B$ is \UMDp for all $p\in(0,+\infty)$.
\end{definition}

The following theorem is a fundamental tool in the theory of UMD spaces. We refer to \cite{pisierMartingalesBanachSpaces2016} for a proof of it (and much more).

\begin{theorem}[{\cite[Theorem 8.12]{pisierMartingalesBanachSpaces2016}}]
Let $(B,|\cdot|_B)$ be a Banach space with the \UMDp  property. Then $(B,|\cdot|_B)$ has the UMD property.
\end{theorem}

A direct consequence of the remark after \cite[Definition 8.7]{pisierMartingalesBanachSpaces2016} and the previous theorem is the following fact~:

\begin{proposition}
The finite dimensional vector space $\RR^d$ endowed with its standard euclidean norm has the UMD property.
\end{proposition}

\begin{proof}
The proof relies on the classical discrete Burkolder Davis Gundy (BDG) inequality. Indeed let $(M_n)_n$ be an $L^p$ martingale in $\RR^d$. Let $\eps_1,\cdots,\eps_N \in \{-1,1\}$, $M^\eps_0 = 0$ and for $N\in\{1,\cdots,N\}$
\[M^\eps_n = \sum_{k=0}^{n-1} \eps_k (M_{k+1} - M_k).\]
Then $(M^\eps_n)_{n\in\{0,\cdots,N\}}$ is a martingale, and we have 
\begin{multline*}
\EE\left[\left|M^\eps_N\right|^p\right] 
	\lesssim 
\EE\left[\left(\sum_{k=0}^{N-1}\left|M^\eps_{k+1}-M^\eps_{k+1}\right|^2\right)^{\frac{p}{2}}\right] 
	\\ = 
\EE\left[\left(\sum_{k=0}^{N-1}\left|M_{k+1}-M_{k+1}\right|^2\right)^{\frac{p}{2}}\right] 
	\lesssim
\EE[|M_N - M_0|^p], 
\end{multline*}
were we have applied twice the BDG inequality.

\end{proof}

\begin{theorem}[{\cite[Corollary 8.19]{pisierMartingalesBanachSpaces2016}}]\label{thm:lpumd}
Let $p\in(1,+\infty)$, let $(B,|\cdot|_B)$ be an Banach space with the UMD and let $(A,\cA,\mu)$ be a measured space. The space $L^p(A;B)$ is a UMD space. 
\end{theorem}

\begin{proof}
The proof follows easily  from the previous results. Indeed, it is enough to prove that $L^p(A;B)$ has the \UMDp property. Let $(M_n)_{n}$ be a $L^p$ martingale with value in $L^p(A;B)$. Remark that for $\mu$-almost all $x \in A$, $\big(M_n(x)\big)_n$ is a $B$ value martingale. Furthermore, thanks to the Fubini property, we have 
\[\EE\left[\|M_n\|_{L^p(A;B)}^p\right] = \int_A \EE[|M_n(x)|_B^p] \dd \mu(x).\]
Hence for $\mu$-almost all $x$, $\EE[|M_n(x)|_B^p]<+\infty$ and $\big(M_n(x)\big)_n$ is a $L^p$ martingale with value in $B$. Hence for all $\eps_1,\cdots,\eps_n \in \{-1,1\}$, we have for almost all $x\in A$ and all $N\ge 1$
\[\EE[|M_N^\eps(x)|^p_B] \le K_p(B)^p \EE[|M_N(x) - M_0(x)|^p_B].\]
Hence, we have
\begin{multline*}
\EE[\|M^\eps_N\|_{L^p(A;B)}^p] = \int_A \EE[ |M^\eps_N(x)|^p_B ] \dd \mu(x) \\ \le K_p(B)^p \int_A \EE[ |M_N(x)- M_0(x)|^p_B ] \dd \mu(x) = K_p(B)^p \EE[\|M_N - M_0\|_{L^p(A;B)}^p].  
\end{multline*}
Hence, $L^p(A;B)$ is \UMDp and then UMD.
\end{proof}

\begin{corollary}\label{cor:besovumd}
Let $p,q \in (1,+\infty)$ and let $s \in \RR$. Then $B^s_{p,q}(\RR^d,\RR)$ is UMD.
\end{corollary}

\begin{proof}
Is previously it is enough to prove that $B^s_{p,q}$ is $\text{UMD}_q$. Furthermore, since $L^p(\RR^d;\RR)$ is UMD we have  for al $N\in \NN$ and all $\eps_1,\cdots,\eps_N \in \{-1,1\}$,
\begin{equation*}
\EE[\|M^\eps_N\|_{B^s_{p,q}}^q 
= 
\sum_{j\ge -1} 2^{s j q} \EE[ \|\Delta_j M^\eps_N\|^q_{L^p(\RR^d;\RR)}] 
= K_q(L^p(\RR^d;\RR))^q \EE[\|M_N - M_0\|_{B^s_{p,q}}^q],
 \end{equation*}
 and the result follows easily.
\end{proof}

\subsection{Banach space, smoothness and type}

\begin{definition}[{\cite[Definitions p-108]{pisierMartingalesBanachSpaces2016}}]
Let $p\in(1,2]$ and $B,|\cdot|_B)$ be a Banach space. We say that $B$ is of type $p$ if there exists a constant $M_p$ such that for every $N\ge 0$, every iid sequence $X_1\cdots,X_n$ with $\PP(X_1 = 1) = \PP(X_1 = -1) = \frac12$ and every $b_1,\cdots,b_N \in B$, we have
\[\EE
	\left[
		\left|
			\sum_{k=1}^N X_k b_k
		\right|_B
	\right] 
\le M_p 
\left(
		\sum_{k=1}^N 
		|b_k|_B^p
	\right)^{\frac{1}{p}}.
	\]
\end{definition}

Let us first give an example :

\begin{proposition}[{\cite[Proposition 4.33]{pisierMartingalesBanachSpaces2016}}]
Let $\RR^d$ with its Euclidean norm. Then $\RR^d$ is of type $2$. More generally any Hilbert space is of type $2$.
\end{proposition}

\begin{proof}
The proof is quite straightforwerd. Let us denote by $\langle\cdot,\cdot,\rangle$ the inner product. We have for $N\ge 1$, $b_1,\cdots,b_N \in \RR^d$ and $X_1,\cdots,X_n$ a sequence of iid random variables with $\PP(X_1 = 1 ) = \PP(X_1 = -1) = \frac12$,
\begin{align*}
\EE\left[\left|\sum_{k=1}^N X_k b_k\right|^2\right] = \sum_{k,l=1}^N \EE\left[X_k X_l\right] \langle b_k , b_l \rangle
=& 
\sum_{k=1}^N |b_k|^2.
\end{align*}
Applying the Khinchine-Kahane inequality, we have the result. 
\end{proof}

\begin{theorem}[{\cite[Proposition 4.34]{pisierMartingalesBanachSpaces2016}}]\label{thm:lptype}
Let $(A,\cA,\mu)$ be a measured space and let $(B,|\cdot|_B)$ be a type $p$ Banach space with $p\in(1,2]$. Let $q\in[1,+\infty)$. Then $L^q(A;B)$ is of type $p\wedge q$.
\end{theorem}

\begin{proof}
Take $N\in \NN$, $b_1,\cdots,b_N \in L^q(A;B)$ and $X_1,\cdots,X_N$ a sequence of iid random variables with $\PP(X_1 = 1 ) = \PP(X_1 = -1) = \frac12$. 

Applying twice the Khinchine-Kahane inequality, we get
\begin{align*}
\EE\left[\left\|\sum_{k=1}^n X_k b_k\right\|_{L^q(A;B)}\right]
\lesssim & 
\EE\left[\left|\sum_{k=1}^n X_k b_k\right|^q\right]^{\frac1q} \\
& \lesssim
\left(\int_A \EE\left[\left|\sum_{k=1}^n X_k b_k(x)\right|_B^q\right] \dd \mu(x)\right)^{\frac1q} \\
& \lesssim
\left(\int_A \EE\left[\left|\sum_{k=1}^n X_k b_k(x)\right|_B\right]^q \dd \mu(x)\right)^{\frac1q} \\
&  \lesssim
\left(\int_A \left(\sum_{k=1}^n |b_k(x)|_B^p\right)^{\frac{q}{p}} \dd \mu(x)\right)^{\frac1q},\\
\end{align*}
where we have used the fact that $(B,|\cdot|_B)$ is of type $p$ in the last inequality. 

Now, let us suppose that $q\le p$. We have for every $N\ge 1$ and $y_1,\cdots,y_N \in \NN$,
\[\left(\sum_{k=1}^N |y_i|^p \right)^{\frac1p} \le \left(\sum_{k=1}^N |y_i|^q\right)^{\frac1q}.\]
Finally we have 
\begin{equation*}
\EE\left[\left\|\sum_{k=1}^N X_k b_k\right\|_{L^q(A;B)}\right] \lesssim \left(\int_A \sum_{k=1}^N |b_k(x)|_B^q \dd \mu(x)\right)^{\frac1q} \lesssim \left(\sum_{k=1}^N \|b_k\|_{L^{q}(A;B)}^q\right)^{\frac{1}{q}}.
\end{equation*}
Hence, in that case we have proved $L^q(A;B)$ which is of $q$ type.

Let us suppose that $p\le q$. Note that we have
\begin{align*}
\EE\left[\left\|\sum_{k=1}^n X_k b_k\right\|_{L^q(A;B)}\right]
\lesssim &
\left(\left\{\int_A \left(\sum_{k=1}^n |b_k(x)|_B^p\right)^{\frac{q}{p}} \dd \mu(x)\right\}^{\frac{p}{q}}\right)^{\frac1p} \\
\lesssim &
\left( \sum_{k=1}^N \left(\int_A |b_k(x)|_B^q \dd \mu(x) \right)^{\frac{p}{q}} \right)^{\frac1p}\\
\lesssim &
\left( \sum_{k=1}^N \|b_k\|_{L^q(A;B)}^p \right)^{\frac1p}.
\end{align*}
In that case $L^q(A;B)$ if of $q$ type, which ends the proof.
\end{proof}

\begin{corollary}\label{cor:besovtype}
Let $1\le p,q < + \infty$ and $s \in \RR$. Then $B^{s}_{p,q}(\RR^d;\RR)$ is of type $p\wedge q \wedge 2$. 
\end{corollary}

\begin{proof}
The proof follows easily when we see that 
\[\|f\|_{B^s_{p,q}} = \|(2^{js} f)_{j\ge -1}\|_{\ell^q(L^p(\RR^d;\RR))}.\]
By the previous result, $L^p(\RR^d;\RR)$ is of type $p \wedge 2$ and $\ell^q(L^p(\RR^d;\RR))$ is of type $q \wedge p \wedge 2$.
\end{proof}

\subsection{Burkolder-Davis-Gundy in Banach spaces}
We now have all the tools to state and prove the following theorem. One can consult \cite[Chapter 4]{pisierMartingalesBanachSpaces2016} for more details.

\begin{theorem}[{\cite[Proposition 4.37 and Theorem 4.52]{pisierMartingalesBanachSpaces2016}}]
Let $(B,|\cdot|_B)$ be a UMD Banach space of type $p \in (1,2]$. Then for any $r\in (1,+\infty)$ a constant $C_r>0$ exists such that for any $N\ge 1$ and any $L^r$ martingale with value in $B$
\begin{equation}\label{eq:BDG_UMD}\EE[|M_N-M_0|^r] \le C_r^r \EE\left[\left(\sum_{k=0}^{N-1} | M_{k+1}-M_k|_B^{\frac{r}{p}}\right)^{\frac{p}{r}}\right].
\end{equation}
\end{theorem}

\begin{proof}
Let $(M_n)_n$ as in the theorem and $N\ge 1$. Let $(X_1,\cdots,X_N)$ be an iid random vector independent of $M$ with $\PP(X_1 = 1) = \PP(X_1 = -1) = \frac12$.

First, remark that thanks to the UMD property, we have for all $\eps_1,\cdots,\eps_n \in \{-1,1\}$,
\[\EE[|M_N - M_0|_B^r] \lesssim  \EE[|M^\eps_N|_B^r].\]
Hence, we have
\begin{align*}
\EE[|M_N - M_0|_B^r] 
	\lesssim & 
\EE\left[\left|\sum_{k=0}^{N-1} X_k(M_{k+1} - M_k)\right|_B^r\right] & (\text{UMD}) \\
	\lesssim &
\EE\left[ \EE\left[\left|\sum_{k=0}^{N-1} X_k(M_{k+1} - M_k)\right|_B^r\Big|(M_n)_{n\in\{0,\cdots,N}\right]\right]\\
	\lesssim &
\EE\left[ \EE\left[\left|\sum_{k=0}^{N-1} X_k(M_{k+1} - M_k)\right|_B\Big|(M_n)_{n\in\{0,\cdots,N}\right]^r\right] & (\text{Khinchine-Kahane})\\
	\lesssim &
\EE\left[ \left(\sum_{k=0}^{N-1} |M_{k+1} - M_k)|^p_B\right)^{\frac{r}{p}}\right] & (\text{type}-p).
\end{align*}
\end{proof}

Using Corollaries \ref{cor:besovtype} and \ref{cor:besovumd} and Theorems \ref{thm:lptype} and \ref{thm:lpumd}, we have the following useful corollary :

\begin{corollary}
Let $p,q\in(1,+\infty)$ and $s\in \RR$. The previous Inequality \ref{eq:BDG_UMD} holds for $L^p(\RR^d;\RR)$ (respectively $B^{s}_{p,q}(\RR^d;\RR)$) with $p\wedge 2$ instead of $p$ (respectively with $p\wedge 2 \wedge q$ instead of $p$).
\end{corollary}

\begin{remark}
Using the fact that $F^s_{p,q}(\RR^d;\RR)$ is endowed in $L^{q}(\ell^p)$, one could prove excatly the same theorem for Triebel-Lizorkin spaces.
\end{remark}

\section{The multi-parameter and multi-dimensionnal Garsia-Rodemich-Rumsey inequality}\label{appendix:GRR} Let us recall the standard Garsia-Rodemich-Rumsey inequality \cite{garsia1970real}.
\begin{theorem}\label{thm:GRRstd}
Let $p\geq1$, $\alpha>p^{-1}$ and $f\in C([0,1])$. Then, the following inequality holds
\begin{equation}\label{eq:GRRstd}
|f(t)-f(s)| \lesssim_{\alpha,p} \kappa_f |s-t|^{\alpha-1/p}
\end{equation}
for all $t,s\in[0,1]^2$, where
\begin{equation*}
\kappa_f = \left(\int_{[0,1]^2} \frac{|f(u) - f(v)|^p}{|u-v|^{\alpha p + 1}}dudv\right)^{1/p}.
\end{equation*}
\end{theorem}

We can extend the previous result to the case $f\in\cC([0,1]\times B(0,R) )$ where $B(0,R)\subset \RR^d$ is a closed ball of radius $R>0$. By denoting, for any $x,y\in[0,1]^d$ and $s,t\in[0,1]$,
\begin{equation*}
\square_{(s,x)} f(t,y) = f(s,x) - f(s,y) - f(t,x) + f(t,y),
\end{equation*}
we have the following result.
\begin{corollary}\label{cor:GRR}
Let $p\geq1$, $\alpha_1>1/p$, $\alpha_2>d/p$ and $f\in\cC([0,1]\times B(0,R) )$. Then, we have the following estimate
\begin{equation}
|\square_{(s,x)} f(t,y)|^p \lesssim_{d,\alpha_1,\alpha_2,p} \kappa_f |s-t|^{\alpha_1p-1}|x-y|^{\alpha_2 p - d}
\end{equation}
for all $t,s\in[0,1]$ and $x,y\in B(0,R)$, where 
\begin{equation*}
\kappa_f =\int_{[0,1]^{2}\times B(0,R)^2} \frac{|\square_{(u,z)} f(v,w)|^p}{|u-v|^{\alpha_1p + 1}|z-w|^{\alpha_2p+d}}dudv dzdw.
\end{equation*}
\end{corollary}

\begin{proof} We essentially use the arguments from \cite{stroock2010probability} and \cite{hu2013multiparameter}. Let $\{(E_j,d_j)\}_{1\leq j \leq m}$ be a family of separable metric space each endowed with a finite doubling measure $\upsilon_j$ defined on the Borel sets of $E_j$ (we essentially need that Lebesgue's differentiation theorem holds) and $\Psi$ a positive increasing convex function such that $\Psi(0) = 0$ (and, thus, $\Psi^{-1}$ is a positive increasing concave function). We define $\sigma_j(r) = \min_{x\in E_j} \upsilon_j(B(x,r))$ the volume (under $\upsilon_j$) of the smallest ball of radius $r>0$ in $E_j$. For any metric space $(E,d)$ endowed with a measure $\upsilon$ defined on the Borel sets of $E$ and for any scalar-valued function $f\in\cC(E)$, we define
\begin{equation*}
\bar{f}(A) : = \upsilon(A)^{-1} \int_{A} f(x) \upsilon(dx),
\end{equation*}
for any Borel set $A\in E$. For any $(E,d,\upsilon)\in\{(E_j,d_j,\upsilon_j)\}_{1\leq j\leq m}$, we have, by Lebesgue's differentiation theorem
\begin{equation*}
\bar{f}(B(x,r)) \underset{r\to 0}{\to} f(x).
\end{equation*}
We denote in the following $\cE = \bigotimes_{j = 1}^m E_j$ the product space that is can be made metric by considering $d(x,y) = \left(\sum_{j=1}^m d(x_j,y_j)^2\right)^{1/2}$. It can also be made measurable by considering the product measure $\upsilon(A) = \prod_{j = 1}^m \upsilon(A_j)$ for any Borel set $A = \prod_{j = 1}^m A_j$. For any $x = (x_1,x_2,\ldots,x_m),y = (y_1,y_2,\ldots,y_m)\in\cE$, we denote $x' = (x_1,\ldots,x_{m-1})$ (for $m\geq 2$) and define, for any $1\leq k\leq m$,
\begin{equation*}
V_{k,y}x : = (x_1,\ldots,x_{k-1},y_k,x_{k+1},\ldots,x_m).
\end{equation*}
Furthermore, for any function $f$ defined on $\cE$, we denote $V_{k,y}f(x) = f(V_{k,y}x)$. We proceed to define the joint increment of $f$ as
\begin{equation*}
\square_y f(x) : = \prod_{j = 1}^m (\textrm{id} - V_{k,y})f(x).
\end{equation*}
From the previous definitions, we can see that, for any $x,y\in\cE$ and $f$ defined on $\cE$
\begin{equation*}
\square_y f(x) = \square_{y'} f(x',x_m) - \square_{y'} f(x',y_m).
\end{equation*}
We can now state an intermediate result.
\begin{lemma}\label{lem:GRR}
Let $f\in\cC(\cE)$ be such that
\begin{equation*}
\kappa_f : = \int_{\cE}\int_{\cE} \Psi\left(\frac{|\square_yf(x)|}{\prod_{j=1}^m d_j(x_i,y_j)} \right)\upsilon(dx)\upsilon(dy)< +\infty.
\end{equation*}
Then the following inequality holds for any $x,y\in\cE$
\begin{equation*}
|\square_yf(x)| \leq 18^m \int_0^{d_1(x_1,y_1)/2}\ldots\int_0^{d_m(x_m,y_m)/2}\Psi^{-1}\left(\frac{\kappa_f}{\prod_{j = 1}^m \sigma_j(r_j)^2}\right)\dd r_1\ldots \dd r_m.
\end{equation*}
\end{lemma}
\begin{proof}
We proceed by induction; the case $m = 1$ being the classical Garsia-Rodemich-Rumsey. We assume that the result holds for a certain $m-1\geq 1$. For any $x,y\in\cE$, we denote $g_{x',y'}(z) =\square_{y'} f(x',z)$. For any Borel sets $A_1,A_2\subset E_m$, we have
\begin{align*}
\left|\bar{g}_{x',y'}(A_1) - \bar{g}_{x',y'}(A_2)\right| &\leq \int_{A_1}\int_{A_2} \left|g_{x',y'}(x_m) - g_{x',y'}(y_m)\right|\frac{\upsilon_m(dx_m)\upsilon_m(dy_m)}{\upsilon_m(A_1)\upsilon_m(A_2)}
\\ &\leq d_m(A_1,A_2) \int_{A_1}\int_{A_2} \left|\frac{g_{x',y'}(x_m) - g_{x',y'}(y_m)}{d_m(x_m,y_m)}\right|\frac{\upsilon_m(dx_m)\upsilon_m(dy_m)}{\upsilon_m(A_1)\upsilon_m(A_2)},
\end{align*}
where $d_m(A_1,A_2) = \sup_{x_m\in A_1,y_m\in A_2} d_m(x_m,y_m)$. 
By the induction hypothesis, we have that
\begin{align*}
\frac{|g_{x',y'}(x_m) - g_{x',y'}(y_m)|}{d_m(x_m,y_m)} = \left|\square_{y'} \delta_{x_m,y_m}f(x')\right|
\\ \leq 18^{m-1}  \int_{D}\Psi^{-1}\left(\frac{\kappa_{\delta_{x_m,y_m}f}}{\prod_{j = 1}^{m-1} \sigma_j(r_j)^2}\right)\dd r,
\end{align*}
where we denote $D = [0,d_1(x_1,y_1)/2]\times \ldots \times [0,d_{m-1}(x_{m-1},y_{m-1})/2]$, $r = (r_1,\ldots,r_{m-1})$ and $\delta_{x_m,y_m}f(x') = (f(x',x_m) - f(x',y_m))/d_m(x_m,y_m)$. Thus, by Fubini's theorem and Jensen's inequality, we obtain that
\begin{align*}
\left|\bar{g}_{x',y'}(A_1) - \bar{g}_{x',y'}(A_2)\right| 
\\ \leq 18^{m-1}d_m(A_1,A_2)\int_{D} \int_{A_1}\int_{A_2} \Psi^{-1}\left(\frac{\kappa_{\delta_{x_m,y_m}f}}{\prod_{j = 1}^{m-1} \sigma_j(r_j)^2}\right)\frac{\upsilon_m(dx_m)\upsilon_m(dy_m)}{\upsilon_m(A)\upsilon_m(B)}\dd r
\\ \leq 18^{m-1}d_m(A_1,A_2)\int_{D} \Psi^{-1}\left(\int_{A_1}\int_{A_2} \frac{\kappa_{\delta_{x_m,y_m}f}}{\prod_{j = 1}^{m-1} \sigma_j(r_j)^2}\frac{\upsilon_m(dx_m)\upsilon_m(dy_m)}{\upsilon_m(A_1)\upsilon_m(A_2)}\right)\dd r.
\end{align*}
We remark that
\begin{align*}
\int_{A_1}\int_{A_2} \kappa_{\delta_{x_m,y_m}f}\upsilon_m(dx_m)\upsilon_m(dy_m) \leq \int_{E_m}\int_{E_m} \kappa_{\delta_{x_m,y_m}f}\upsilon_m(dx_m)\upsilon_m(dy_m) = \kappa_f,
\end{align*}
and, thus, we obtain
\begin{equation}\label{eq:GRRprf0}
\left|\bar{g}_{x',y'}(A_1) - \bar{g}_{x',y'}(A_2)\right| \leq 18^{m-1}d_m(A_1,A_2)\int_{D} \Psi^{-1}\left(\frac{\kappa_{f}}{\prod_{j = 1}^{m-1} \sigma_j(r_j)^2}\frac{1}{\upsilon_m(A_1)\upsilon_m(A_2)}\right)\dd r.
\end{equation}
We now denote $\bar{g}_{x',y'}(x_m,r) = \bar{g}_{x',y'}(B(x_m,r))$ for any $x_m\in E_m$ and 
\[B(x_m,r) = \{ x\in E_m; d_m(x,x_m)\leq r\}\] for a certain $r>0$. We now fix $x_m,y_m\in E_m$ and let $\lambda_k = d_m(x_m,y_m)2^{-k}$ for any $k\geq 0$. Since $d_m(B(x_m,\lambda_0),B(y_m,\lambda_0)) = 3 \lambda_0$, $\sigma_m(\lambda_k)\geq \sigma_m(\tau)$ for any $\tau\leq \lambda_k$ and $\Psi^{-1}$ is increasing, it follows from \eqref{eq:GRRprf0} that
\begin{align}
&\left|\bar{g}_{x',y'}(x_m,\lambda_0) - \bar{g}_{x',y'}(y_m,\lambda_0)\right| \nonumber
\\ &\hspace{2em}\leq 3\times 18^{m-1}\lambda_0\int_{D} \Psi^{-1}\left(\frac{\kappa_{f}}{\prod_{j = 1}^{m-1} \sigma_j(r_j)^2}\frac{1}{\sigma_m(\lambda_0)^2}\right)\nonumber
\\ &\hspace{2em}\leq 6\times 18^{m-1}\int_0^{d_m(x_m,y_m)/2}\int_{D} \Psi^{-1}\left(\frac{\kappa_{f}}{\prod_{j = 1}^{m-1} \sigma_j(r_j)^2}\frac{1}{\sigma_m(\tau)^2}\right)\dd r \dd \tau.\label{eq:GRRprf1}
\end{align}
Furthermore, by using \eqref{eq:GRRprf0} and since $d_m(B(x_m,\lambda_\ell),B(x_m,\lambda_{\ell+1})) = \lambda_{\ell} + \lambda_{\ell+1} = 6(\lambda_{\ell+1} - \lambda_{\ell+2})$ for any $\ell\geq 0$, $\sigma_m(\lambda_k)\geq \sigma_m(\tau)$ for any $\tau\leq \lambda_k$ and $\Psi^{-1}$ is increasing, we obtain that, for any $k\geq 1$,
\begin{align*}
&\left|\bar{g}_{x',y'}(x_m,\lambda_0) - \bar{g}_{x',y'}(x_m,\lambda_k)\right| \leq \sum_{\ell = 0}^{k-1} \left|\bar{g}_{x',y'}(x_m,\lambda_{\ell+1}) - \bar{g}_{x',y'}(x_m,\lambda_{\ell})\right| 
\\ &\hspace{2em}\leq 6\times18^{m-1}\sum_{\ell = 0}^{k-1}(\lambda_{\ell+1} - \lambda_{\ell+2})\int_{D} \Psi^{-1}\left(\frac{\kappa_{f}}{\prod_{j = 1}^{m-1} \sigma_j(r_j)^2}\frac{1}{\sigma_m(\lambda_{\ell+1})\sigma_m(\lambda_{\ell+2})}\right)\dd r
\\ &\hspace{2em}\leq 6\times18^{m-1}\sum_{\ell = 0}^{k-1}\int_{\lambda_{\ell+2}}^{\lambda_{\ell+1}}\int_{D} \Psi^{-1}\left(\frac{\kappa_{f}}{\prod_{j = 1}^{m-1} \sigma_j(r_j)^2}\frac{1}{\sigma_m(\tau)^2)}\right)\dd r \dd \tau
\\ &\hspace{2em}\leq 6\times18^{m-1}\int_{\lambda_{k+1}}^{\lambda_{1}}\int_{D} \Psi^{-1}\left(\frac{\kappa_{f}}{\prod_{j = 1}^{m-1} \sigma_j(r_j)^2}\frac{1}{\sigma_m(\tau)^2)}\right)\dd r\dd \tau.
\end{align*}
Letting $k\to+\infty$, we deduce that
\begin{align}
&\left|\bar{g}_{x',y'}(x_m,\lambda_0) - g_{x',y'}(x_m)\right| \leq \nonumber
\\ &\hspace{2em}\leq 6\times18^{m-1}\int_{0}^{d_m(x_m,y_m)/2}\int_{D} \Psi^{-1}\left(\frac{\kappa_{f}}{\prod_{j = 1}^{m-1} \sigma_j(r_j)^2}\frac{1}{\sigma_m(\tau)^2)}\right)\dd r\dd \tau\label{eq:GRRprf2},
\end{align}
and, similarly,
\begin{align}
&\left|\bar{g}_{x',y'}(y_m,\lambda_0) - g_{x',y'}(y_m)\right| \leq \nonumber
\\ &\hspace{2em}\leq 6\times18^{m-1}\int_{0}^{d_m(x_m,y_m)/2}\int_{D} \Psi^{-1}\left(\frac{\kappa_{f}}{\prod_{j = 1}^{m-1} \sigma_j(r_j)^2}\frac{1}{\sigma_m(\tau)^2)}\right)\dd r\dd \tau\label{eq:GRRprf3}.
\end{align}
Thanks to \eqref{eq:GRRprf1}, \eqref{eq:GRRprf2} and \eqref{eq:GRRprf3}, we obtain the desired result.
\end{proof}

Corollary \ref{cor:GRR} follows directly from Lemma \ref{lem:GRR} by choosing $E_1 = [0,1]$  endowed with $d_1(t,s) = |t-s|^{\alpha_1+1/p}$ and $E_2 = B(0,R)$ endowed with $d_2(x,y) = |x-y|^{\alpha_2 +d/p}$. We also use, for both, Lebesgue's measure and set $\Psi(t) = |t|^{p}$. It turns out that $\sigma_1(r) \simeq r^{1/(\alpha_1+1/p)}$ and $\sigma_2(r) \simeq r^{1/(\alpha_2/d+1/p)}$. With the previous choices at hand, we obtain
\begin{align*}
\int_0^{d_1(s,t)/2}\int_0^{d_2(x,y)/2}\Psi^{-1}&\left(\frac{\kappa_f}{\sigma_1(r_1)^2\sigma_2(r_2)^2}\right)\dd r_1\dd r_2 
\\ &\simeq \kappa_f^{1/p} \int_0^{d_1(s,t)/2}r_1^{-2/(\alpha_1 p+1)}\dd r_1\int_0^{d_2(x,y)/2} r_2^{-2/(\alpha_2 p/d+1)} \dd r_2
\\&\simeq \kappa_f^{1/p} |t-s|^{\alpha_1-1/p} |x-y|^{\alpha_2-d/p},
\end{align*}
where
\begin{equation*}
\kappa_f = \int_{[0,1]^2}\int_{B(0,R)^2} \frac{|\square_{(s,x)}f(t,y)|^p}{|t-s|^{\alpha_1p +1}|x-y|^{\alpha_2 p +d}}dtdsdxdy.
\end{equation*}
\end{proof}

We can finally state a Kolmogorov's continuity theorem.

\begin{theorem}\label{thm:Kolmogo}
Let $p\geq1$, $\alpha_1>1/p$, $\alpha_2>d/p$ and $f\in\cC([0,1]\times \RR^d )$ a random process. Assume that there exists $\varepsilon_1,\varepsilon_2\in(0,1)$ such that
\begin{equation}\label{eq:hypkolmogo}
\vartheta_f : = \sup_{\substack{s,t\in[0,1]\\ x,y\in\RR^d}}\EE\left[\frac{|\square_{(s,x)} f(t,y)|^p}{|t-s|^{\alpha_1p+\varepsilon_1}|x-y|^{\alpha_2p+\varepsilon_2}} \right]^{1/p} <+\infty.
\end{equation}
Then, for any $\varepsilon_3>\varepsilon_2$, there exists a positive random variable $\Upsilon\in L^p(\Omega)$ which depends on $f,d,\alpha_1,\alpha_2,p,\varepsilon_1,\varepsilon_2$ and $\varepsilon_3$ such that, $\PP$-a.s.,
\begin{equation}
|\square_{(s,x)} f(t,y)| \leq \Upsilon |s-t|^{\alpha_1-1/p}|x-y|^{\alpha_2 - d/p}(1+|x|+|y|)^{(d+\varepsilon_3)/p},
\end{equation}
for all $t,s\in[0,1]$ and $x,y\in \RR^d$.
\end{theorem}
\begin{proof}
Let $\cA = \{ x\in\RR^d;\; 1\leq |x|\leq 2\}$ be an annulus in $\RR^d$ such that $\RR^d = B(0,1)\cup \bigcup_{j = 0}^{+\infty} 2^j\cA$.
It follows from Corollary \ref{cor:GRR} and \eqref{eq:hypkolmogo} that, for any $j\geq 0$,
\begin{align*}
&\EE\left[\sup_{\substack{t,s\in[0,1]\\ x,y\in 2^{j}\cA}}\frac{|\square_{(s,x)} f(t,y)|^p}{|s-t|^{\alpha_1p-1}|x-y|^{\alpha_2 p - d}(1+|x|+|y|)^{d+\varepsilon_3}} \right]
\\ &\hspace{1em}\leq \frac{C_{d,\alpha_1,\alpha_2,p}\vartheta_f }{(1+2^{j+1})^{d+\varepsilon_3}}\int_{[0,1]^{2}\times B(0,1+2^{j+2})^2} \frac{du_1du_2 dz_1dz_2}{|u_1-u_2|^{1-\varepsilon_1}|z_1-z_2|^{d-\varepsilon_2}}
\\ &\hspace{1em}\leq C_{d,\alpha_1,\alpha_2,p}\vartheta_f\frac{2^{(d+\varepsilon_2)(j+2)}}{(1+2^{j+1})^{d+\varepsilon_3}}\int_{[0,1]^{2}\times B(0,2)^2} \frac{du_1du_2 dzdw}{|u_1-u_2|^{1-\varepsilon_1}|z_1-z_2|^{d-\varepsilon_2}}
\\ &\hspace{1em}\lesssim 2^{j(\varepsilon_2-\varepsilon_3)}.
\end{align*}
Thus, we deduce that
\begin{align*}
&\EE\left[\sup_{\substack{t,s\in[0,1]\\ x,y\in \RR^d}}\frac{|\square_{(s,x)} f(t,y)|^p}{|s-t|^{\alpha_1p-1}|x-y|^{\alpha_2 p - d}(1+|x|+|y|)^{d+\varepsilon_3}} \right]
\\ &\hspace{1em}\leq \EE\left[\sup_{\substack{t,s\in[0,1]\\ x,y\in B(0,1)}}\frac{|\square_{(s,x)} f(t,y)|^p}{|s-t|^{\alpha_1p-1}|x-y|^{\alpha_2 p - d}} \right] 
\\ &\hspace{2em}+ \sum_{j = 0}^{+\infty}\EE\left[\sup_{\substack{t,s\in[0,1]\\ x,y\in 2^{j}\cA}}\frac{|\square_{(s,x)} f(t,y)|^p}{|s-t|^{\alpha_1p-1}|x-y|^{\alpha_2 p - d}(1+|x|+|y|)^{d+\varepsilon_3}} \right]
\\ &\hspace{1em}\lesssim 1 + \sum_{j = 0}^{+\infty}2^{j(\varepsilon_2-\varepsilon_3)}<+\infty,
\end{align*}
which ends the proof.
\end{proof}

We also recall another standard Kolmogorov theorem.

\begin{theorem}\label{thm:Kolmogo_std}
Let $f\in C([0,1];L^{\infty}_{\mathrm{loc}}(\RR^d))$ be a random process and $p\geq 1$ such that
\begin{equation*}
\sup_{\substack{t,s\in[0,1]\\ x\in\RR^d}}\EE\left[ \frac{|f(t,x) - f(s,x)|^p}{|t-s|^{\alpha p + \varepsilon}}\right]<+\infty,
\end{equation*}
for a certain $\alpha>1/p$ and $\varepsilon\in(0,1)$. Then, there exists a positive random variable $\Upsilon\in L^p(\Omega)$ which depends on $f,d,\alpha$ and $p$ such that
\begin{equation*}
|f(t,x)-f(s,x)| \leq \Upsilon |s-t|^{\alpha - 1/p}(1+|x|)^{(1+\iota)/p},
\end{equation*}
where $\iota>0$.
\end{theorem}

\subsection*{Acknowledgment} The authors are sincerely grateful to the referees for their comments
which greatly helped to improve this work. \blue{The author also would like to thank M\'at\'e Gerencs\'er and Konstantinos Dareiotis who pointed out a mistake in a previous version of this work. This project was partly funded by the project PEPS 2019 N°192909A of the INSMI Institute (CNRS).}

\bibliographystyle{plain}
\bibliography{bibliography.bib}{}
\end{document}